\documentclass[reqno]{amsart}

\usepackage{amsthm}   
\usepackage{amsfonts, amsmath, amscd}
\usepackage{amssymb, latexsym}
\usepackage[all,cmtip]{xy}     
\usepackage{enumerate} 
\usepackage{color}
\usepackage[usenames,dvipsnames]{xcolor}
\usepackage{verbatim}
\usepackage{xcolor}
 
\usepackage{hyperref}
 \hypersetup{colorlinks=false}
 \usepackage{cite}
 \usepackage{ amssymb }

 \usepackage{graphicx, psfrag}
 \usepackage{pstool}

    \usepackage{xspace}   
    \usepackage{tikz-cd}
\usepackage{tikz}
\usetikzlibrary{arrows}

\usepackage[titletoc]{appendix}
\usepackage[normalem]{ulem}     
\usepackage{amsthm}
\usepackage{amsmath}
\usepackage{amscd}
\usepackage[latin2]{inputenc}
\usepackage{t1enc}
\usepackage[mathscr]{eucal}
\usepackage{indentfirst}
\usepackage{graphicx}
\usepackage{graphics}
\usepackage{pict2e}
\usepackage{epic}

\usepackage{amssymb}
\usepackage{amsmath}
\usepackage{amsfonts}
\usepackage{amscd}
\usepackage{amsthm}
\usepackage{graphicx,psfrag}
\usepackage{wrapfig}
\usepackage{subeqnarray}
\usepackage{amssymb}
\usepackage{yfonts}
\usepackage[makeroom]{cancel}
\usepackage{trimclip}
\usepackage{cancel}
\newif\iflclip
\newif\ifbclip
\newif\ifrclip
\newif\iftclip
\def\CLIP{\dimexpr\fboxrule+.2pt\relax}
\def\nulclip{0pt}
\newcommand\partbox[2]{%
  \lclipfalse\bclipfalse\rclipfalse\tclipfalse%
  \let\lkern\relax\let\rkern\relax%
  \let\lclip\nulclip\let\bclip\nulclip\let\rclip\nulclip\let\tclip\nulclip%
  \parseclip#1\relax\relax%
  \iflclip\def\lkern{\kern\CLIP}\def\lclip{\CLIP}\fi
  \ifbclip\def\bclip{\CLIP}\fi
  \ifrclip\def\rkern{\kern\CLIP}\def\rclip{\CLIP}\fi
  \iftclip\def\tclip{\CLIP}\fi
  \lkern\clipbox{\lclip{} \bclip{} \rclip{} \tclip}{\fbox{#2}}\rkern%
}
\def\parseclip#1#2\relax{%
  \ifx l#1\lcliptrue\else
  \ifx b#1\bcliptrue\else
  \ifx r#1\rcliptrue\else
  \ifx t#1\tcliptrue\else
  \fi\fi\fi\fi
  \ifx\relax#2\relax\else\parseclip#2\relax\fi
}
\usepackage{slashed} 

\theoremstyle{definition}

\theoremstyle{remark}

\numberwithin{equation}{section}

\usepackage{amsmath}
\usepackage{amssymb}
\usepackage{amsfonts}
\usepackage{amsthm}
\usepackage{mathtools}
\usepackage{graphicx}
\usepackage{colonequals}
\usepackage{booktabs}
\usepackage{cancel}
\usepackage[numbers]{natbib}
\usepackage{hyperref}
\usepackage{fancyhdr}
\usepackage{multirow}
\usepackage{multirow}
\usepackage{braket}
\usepackage{tikz}
\usetikzlibrary{arrows,shapes,decorations.pathmorphing}
\usepackage{tikz-cd}
\usepackage[margin=2.9cm]{geometry}
\usepackage{mathabx}
\usepackage{ esint }
\usepackage{slashed}
\usepackage{mathrsfs}
\usepackage{amssymb}
\usepackage{blkarray}
\usepackage{xcolor}
\usepackage{pinlabel}   
\usepackage{scalerel,stackengine}
\stackMath
\newcommand\reallywidehat[1]{%
\savestack{\tmpbox}{\stretchto{%
  \scaleto{%
    \scalerel*[\widthof{\ensuremath{#1}}]{\kern-.6pt\bigwedge\kern-.6pt}%
    {\rule[-\textheight/2]{1ex}{\textheight}}
  }{\textheight}%
}{0.5ex}}%
\stackon[1pt]{#1}{\tmpbox}%
}
\newcommand{\R}{\mathbb{R}}

\newcommand{\Z}{\mathbb{Z}}

\newcommand{\N}{\mathbb{N}}

\newcommand{\C}{\mathbb{C}}

\newcommand{\br}{\langle}
\newcommand{\kt}{\rangle}

\usepackage{graphicx}
\newcommand{\intprod}{\mathbin{\raisebox{\depth}{\scalebox{1}[-1]{$\lnot$}}}}

\theoremstyle{definition}
\newtheorem{thm}{Theorem}[section]
\newtheorem{prop}[thm]{Proposition}
\newtheorem{lm}[thm]{Lemma}
\newtheorem{defn}[thm]{Definition}
\newtheorem{rem}[thm]{Remark}
\newtheorem{eg}[thm]{Example}
\newtheorem{claim}{Claim}[thm]
\newtheorem{cor}[thm]{Corollary}

\newtheorem{assumption}[thm]{Assumption}

\newtheorem{notation}[thm]{Notation}

\newtheorem*{hyp1*}{Hypothesis (I)}
\newtheorem*{hyp2*}{Hypothesis (II)}
\newtheorem*{hyp3*}{Hypothesis (III)}

\newtheoremstyle{exercise}{}{}{\itshape}{}{\bfseries}{:}{.5em}{\thmname{#1} \thmnumber{#2}\thmnote{(#3)}}
\theoremstyle{exercise}

\usepackage{amsfonts}
\DeclareFontFamily{U}{wncy}{}
\DeclareFontShape{U}{wncy}{m}{n}{<->wncyr10}{}
\DeclareSymbolFont{mcy}{U}{wncy}{m}{n}
\DeclareMathSymbol{\sha}{\mathord}{mcy}{"58}

\newcommand{\xlim}[2]{\underset{#1\to#2}\lim}

\renewcommand{\d}[2]
{\frac{d#1}{d#2}}

\newcommand{\nc}{\newcommand}
\nc {\Prod}{(\underset{I}\ph\, r_I^{n_I})}
\nc {\wlw}{\wedge\ldots\wedge}
\nc {\olo}{\otimes\ldots\otimes}
\nc{\bea}{\begin{eqnarray*}}
\nc{\eea}{\end{eqnarray*}}
\nc {\e}{\varepsilon}
\nc{\de}{\delta}
\nc{\A}{\hat a}
\nc{\Ad}{\hat a^\dag}
\nc{\grad}{\nabla}
\nc{\nlim}{\underset{n\to\infty}\lim}
\nc{\ilim}{\underset{i\to\infty}\lim}
\nc{\isum}{\sum\limits_{i=1}^N}
\nc{\xx}{\underset{x\to x_0}\lim}
\nc{\Bnrm}{\Big |\Big|}
\nc{\sym}{\text{Sym}}
\nc{\inte}{\overset{\circ}}
\nc{\del}{\partial}
\nc{\plp}{+\ldots+}
\nc{\lre}{\longrightarrow}
\nc{\be}{\begin{equation}}
\nc{\ee}{\end{equation}}
\nc{\CP}{\mathbb{CP}}
\nc{\End}{\text{End}}
\nc{\mf}{\mathfrak}
\nc{\Hom}{\text{Hom}}
\nc{\spec}{\text{Spec}}
\nc{\sub}{\subseteq}
\nc{\weakto}{\rightharpoonup}
\nc{\ph}{\varphi}
\nc{\leqc}{\lesssim}
\nc{\delbar}{\overline \del}
\nc{\Tr}{\text{Tr}}
\nc{\Lhe}{\mathcal L_{(\Phi^{h_\e}, A^{h_\e})}}

\author{Gregory J. Parker}
\address{Department of Mathematics, Stanford University}
\email{gjparker@stanford.edu}
\begin{document}


\title{Deformations of $\Z_2$-Harmonic Spinors on 3-Manifolds}

\maketitle

\begin{abstract} 

A $\Z_2$-harmonic spinor on a 3-manifold $Y$ is a solution of the Dirac equation on a bundle that is twisted around a submanifold $\mathcal Z$ of codimension 2 called the singular set.  This article investigates the local structure of the universal moduli space of $\Z_2$-harmonic spinors over the space of parameters $(g,B)$ consisting of a metric and perturbation to the spin connection. The main result states that near a $\Z_2$-harmonic spinor with $\mathcal Z$ smooth, the universal moduli space projects to a codimension 1 submanifold in the space of parameters.  The analysis is complicated by the presence of an infinite-dimensional obstruction bundle and a loss of regularity in the first variation of the Dirac operator with respect to deformations of the singular set $\mathcal Z$, necessitating the use of the Nash-Moser Implicit Function Theorem.

\end{abstract}

\tableofcontents

\section{Introduction}

The notion of a $\Z_2$-harmonic spinor was introduced by C. Taubes to describe the limits of renormalized sequences of solutions to generalized Seiberg-Witten equations. $\Z_2$-harmonic spinors are also the simplest type of Fueter section, and are therefore of interest in the study of gauge theories and enumerative theories on manifolds with special holonomy. Beyond their appearance in these theories, $\Z_2$-harmonic spinors are intrinsic objects on low-dimensional manifolds and can be studied independently.

This article investigates the local structure of the universal moduli space of $\Z_2$-harmonic spinors over the space of parameters on a compact 3-manifold. The main result states that this universal moduli space  locally projects to a codimension 1 submanifold, i.e. a ``wall'',  in the space of parameters. This provides a key step toward confirming expectations that $\Z_2$-harmonic spinors should enter into the above theories via wall-crossing formulas. Results in this direction have also been obtained by R. Takahashi using different techniques \cite{RyosukeThesis}. The present work grew out of attempts to develop a more robust analytic framework for these results, with an eye towards applications to gluing problems \cite{PartIII} and other deformation problems. As observed by S. Donaldson \cite{DonaldsonMultivalued}, the same analytic issues arise in many distinct geometric contexts, many of which remain unexplored \cite{SiqiSLag}.    

\subsection{Main Results}
\label{section1.1}

 Let $(Y,g)$ be a closed, oriented, Riemannian 3-manifold, and fix a spin structure with spinor bundle $S\to Y$. Given a smooth, closed submanifold $\mathcal Z\subset Y$ of codimension 2, choose a real line bundle $\ell\to Y\setminus\mathcal Z$. The spinor bundle $S\otimes_\R \ell$ carries a Dirac operator denoted $\slashed D_{\mathcal Z}$ formed from the spin connection and the unique flat connection on $\ell$ with holonomy in $\Z_2$. 
 
 A {\bf $\Z_2$-harmonic spinor} is a solution  $\Phi \in \Gamma(S\otimes_\R \ell)$ of the twisted Dirac equation on $Y\setminus \mathcal Z$ satisfying \be  \slashed D_{\mathcal Z}\Phi=0 \hspace{1.5cm} \text{and}\hspace{1.5cm}\nabla \Phi \in L^2. \label{Z2prelimdef}\ee

\noindent  The submanifold $\mathcal Z$ is called the {\bf singular set}. The latter requirement implies (non-trivially) that $|\Phi|$ extends continuously to the closed manifold $Y$ with $\mathcal Z\subseteq |\Phi|^{-1}(0)$. The existence and abundance of $\Z_2$-harmonic spinors with $\mathcal Z\neq \emptyset$ on closed 3-manifolds was established by Doan--Walpuski in \cite{DWExistence} and strengthened in \cite{ConnectSumSiqi}.

In addition to the submanifold $\mathcal Z$, the Dirac operator relies on a background choice of a Riemannian metric $g$ on $Y$ and possibly a perturbation $B$ to the spin-connection. Let $\mathcal P=\{(g,B)\}$ denote the parameter space of possible smooth choices. Given a pair $(g_0,B_0)$ and a $\Z_2$-harmonic spinor $(\mathcal Z_0, \ell_0, \Phi_0)$ with respect to this pair, the goal of the present work is to study the local deformation problem, i.e. to describe the structure of the set of nearby pairs $(g,B)\in \mathcal P$ for which there exists a $\Z_2$-harmonic spinor.

This problem cannot be addressed with the standard elliptic theory used for classical harmonic spinors \cite{HitchinHarmonicSpinors,LawsonSpinGeometry}. Indeed, if $\ell$ has a non-trivial twist around $\mathcal Z_0$, the  Dirac operator $\slashed D_{\mathcal Z_0}$ degenerates along the singular set $\mathcal Z_0$ and fails to be uniformly elliptic. Instead, it is an {\bf elliptic edge operator} -- a class of operators well-studied in microlocal analysis \cite{MazzeoEdgeOperators, MazzeoEdgeOperatorsII, Melrosebcalculus}. For such operators elliptic regularity fails, nor must the extension to Sobolev spaces necessarily be Fredholm. In particular, for natural function spaces where the integrability condition in (\refeq{Z2prelimdef}) holds, $\slashed D_{\mathcal Z_0}$ possesses an infinite-dimensional cokernel. As a result, the problem of deforming a solution to a one for a nearby parameter cannot be addressed in a straightforward way by an application of the implicit function theorem. The following key idea, first described by Takahshi in \cite{RyosukeThesis}, addresses this issue: {\it the infinite-dimensional obstruction is cancelled by deformations of the singular set $\mathcal Z$.}

Since the Dirac equation $\slashed D_\mathcal Z$ depends on $\mathcal Z$, but $\mathcal Z$ is in turn determined by the vanishing of the norm $|\Phi|$ of a spinor solving (\refeq{Z2prelimdef}), the singular set and the spinor are coupled and must be solved for simultaneously. The problem thus has a similar character to a free-boundary problem, where the domain and solution must be found concurrently, though the ``boundary'' here has codimension 2. In particular, this analysis requires an understanding  of the derivative of the Dirac operator with respect to deformations of the singular set $\mathcal Z$.

Upgrading the singular set $\mathcal Z$ to a variable, define the {\bf universal Dirac operator} to be the operator acting on pairs $(\mathcal Z,\Phi)$ of a singular set and spinor with reference to a background parameter $p\in \mathcal P$ by 
$$\slashed {\mathbb D}_p(\mathcal Z,\Phi):=\slashed D_{\mathcal Z}\Phi$$ 
\noindent where the choice of parameter $p=(g,B)$ is implicit on the right-hand side. 
\begin{defn}\label{modulispacedef}
Given a parameter pair $p=(g,B) \in \mathcal P$ the {\bf moduli space of smooth $\Z_2$-harmonic spinors} is the space \be \mathscr M_{\Z_2}(p):=\left\{(\mathcal Z, \ell , \Phi) \ \Big | \  \slashed{\mathbb D}_p(\mathcal Z,\Phi)=0 \  \ ,   \  \|\Phi\|_{L^2}=1\ \right\}\Big / \Z_2\label{modulidef}\ee 
\noindent where lines bundles $\ell$ are considered up to topological isomorphism. The {\bf universal moduli space of smooth $\Z_2$-harmonic spinors} is the union \be \widehat{\mathscr M}_{\Z_2}:= \bigcup_{p\in \mathcal P} \mathscr M_{\Z_2}(p).\label{universalmoduli}\ee
\end{defn}

\noindent Because $\slashed D_\mathcal Z$ is $\R$-linear and $\Z_2$ acts by $\Phi\mapsto -\Phi$, the moduli space $\mathscr M_{\Z_2}(p)$ at $p\in \mathcal P$ is a real projective space for each fixed pair $(\mathcal Z,\ell)$.

\begin{rem}\label{regularityrem}
It is expected that there exist $\Z_2$-harmonic spinors where $\mathcal Z$ is not smooth, even when the parameter $p=(g,B)$ is. Results of Taubes and Zhang show that, in general, $\mathcal Z$ must be a closed, rectifiable, subset of (Hausdorff) codimension 2 \cite{ZhangRectifiability,TaubesZeroLoci}. Definition \ref{modulispacedef} could be revised to define a larger moduli space
\be\widehat{\mathscr M}_{\Z_2} \subseteq \widehat{\mathscr M}_{\Z_2}^\text{rec}\label{rectifiablemoduli}\ee

\noindent requiring only this weaker degree of regularity of $\mathcal Z$. Taubes has conjectured \cite[pg. 9]{Taubes3dSL2C} that the singular set is a smooth submanifold of codimension 2 for generic $p$; more generally it is expected that it has the structure of an embedded graph except possibly on a set of parameters of infinite codimension. The results of \cite{TaubesWu,MazzeoHaydysTakahashi, SiqiStudentJDG} support this picture. This article considers only the case that $\mathcal Z$ is smooth (although Banach manifolds of finite regularity curves are used along the way). 
\end{rem}

\bigskip 

We now state the main results. The first result, Theorem \ref{maina} describes the linearized deformation theory near a $\Z_2$-harmonic spinor; the next result, Theorem \ref{mainb}, addresses the non-linear version. Throughout, we fix a central parameter $p_0=(g_0,B_0)$ such that there exists a $\Z_2$-harmonic spinor $(\mathcal Z_0, \ell_0, \Phi_0)$ with respect to $p_0$ meeting the following requirements. 

\medskip 

\begin{defn} \label{regulardef}A $\Z_2$-harmonic spinor $(\mathcal{Z}_0, \ell_0, \Phi_0)$ with respect to a parameter pair $p_0=(g_0, B_0)$ is said to be {\bf regular} if the following three conditions hold: 
\begin{enumerate}
\item[(i)] { {(\underline{{{Smooth}}})}} the singular set $\mathcal Z_0\subset Y$ is a smooth, embedded link, and $\ell_0$ restricts to the m\"obius bundle on every sufficiently small disk normal to $\mathcal Z_0$.
\smallskip 

\item[(ii)] { ({\underline{{Isolated}})}} $\Phi_0$ is the unique $\Z_2$-harmonic spinor for the pair $(\mathcal Z_0, A_0)$ with respect to $p_0=(g_0,B_0)$ up to normalization and sign.
\smallskip 

\item[(iii)]{({\underline{{Non-degenerate}})}} $\Phi_0$ has non-vanishing leading-order, i.e. there is a constant $c>0$ such that $$|\Phi_0| \geq c\cdot \text{dist}(-,\mathcal Z_0)^{1/2},$$

\noindent holds on a tubular neighborhood of $\mathcal Z_0$. 
\end{enumerate}


\end{defn}
\smallskip

When $\mathcal Z_0$ is smooth,  
 \be \slashed D_{\mathcal Z_0}:H^1(S\otimes_\R \ell )\to L^2(S\otimes_\R \ell ) \label{fakeL^12}\ee  
 
 \noindent has closed range and infinite-dimensional cokernel, where $H^1$ is the Sobolev space of sections whose  covariant derivative is $L^2$. Let $\Pi_0$ denote the $L^2$-orthogonal projection to the orthogonal complement of the range, which is naturally isomorphic to the cokernel. The linearized deformation theory is described by the following theorem, which gives a precise manifestation of the key idea explained above: 

\bigskip 
\begin{thm}\label{maina} Let $\text{d}_{(\mathcal Z_0, \Phi_0)}\slashed {\mathbb D}$ denote the linearization of the universal Dirac operator at a regular $\Z_2$-harmonic spinor $(\mathcal Z_0,\Phi_0)$. Then the cokernel component of the partial derivative
\be \Pi_0\circ  \text{d}_{(\mathcal Z_0, \Phi_0)}\slashed{\mathbb D}: H^2(\mathcal Z_0; N\mathcal Z_0) \lre \text{coker}(\slashed D_{\mathcal Z_0})\label{mainamap}\ee
\noindent with respect to the singular set is an elliptic pseudo-differential operator, and its Fredholm extension has index $-1$.  
\end{thm} 

\bigskip 

\noindent Here, sections of the normal bundle $N\mathcal Z_0$ is the tangent space to the space of embeddings of $\mathcal Z_0$. In Section \ref{section4}, it is shown that there is an isomorphism $\text{coker}(\slashed D_{\mathcal Z_0})\simeq \Gamma(\mathcal Z_0;\mathcal C_0)$ of the infinite-dimensional cokernel with a space of sections of a vector bundle on $\mathcal Z_0$; composing with this isomorphism, (\refeq{mainamap}) is a map of sections of vector bundles on $\mathcal Z_0$ and the meaning of pseudodifferential operator is the standard one. The order of this pseudodifferential operator depends on the order of chosen isomorphism with $\Gamma(\mathcal Z_0; \mathcal C_0)$, but the image in $\text{coker}(\slashed D_{\mathcal Z_0})$ is independent of this choice (see Remark \ref{obconvention}).

The proof of Theorem \ref{maina} shows that the image of (\refeq{mainamap}) is $\text{coker}(\slashed D_{\mathcal Z_0})\cap H^{3/2}(Y\setminus \mathcal Z_0)$, up to a finite-dimensional space. As a consequence, $\slashed {\mathbb D}$ displays a {\bf loss of regularity}. Here, this loss of regularity manifests as follows: the map (\refeq{mainamap}) has finite-dimensional kernel, and the closure of its range in $L^2$ has finite codimension. The range, however, is not closed since $ \text{coker}(\slashed D_{\mathcal Z_0})\cap H^{3/2}$ is only dense in the $L^2$-norm, thus in particular, the linearization is not surjective. One cannot circumvent this simply by considering singular sets of $3/2$ lower regularity so that the image in $\text{coker}(\slashed D_{\mathcal Z_0})$ is closed, because this makes the $\text{range}(\slashed D_{\mathcal Z_0})$ component and the non-linear terms unbounded. Thus the loss of regularity means one cannot simultaneously arrange that the operator is both bounded and has surjective linearization.

Loss of regularity is an intriguing phenomenon intrinsic to many types of PDE \cite{HamiltonNashMoser, HormanderNashMoser, PsiDONashMoser}. In general, it arises just as above when for every natural function space $\mathcal X$ for the domain, the codomain $\mathcal Y$ of the operator $\slashed {\mathbb D}: \mathcal X\to \mathcal Y$ may be chosen {\it either} so that the non-linear part of $\slashed{\mathbb D}$ is  bounded, in which case the derivative $\text{d}\slashed{\mathbb D}$ does not have closed range, {\it or} it may be chosen so that the derivative is Fredholm, in which case non-linear part is unbounded. Deformation problems for equations displaying a loss of regularity cannot be addressed using the standard Implicit Function Theorem on Banach spaces; instead they usually must invoke a version of the Nash-Moser Implicit Function Theorem on tame Fr\'echet manifolds, denoted in our case by $\mathcal X$ and $\mathcal Y$.   Using the linearized result Theorem \ref{maina} and the Nash-Moser Implicit Function Theorem leads to our main result:

\begin{thm}\label{mainb} 
There exists an open neighborhood $\mathscr U_0$ of the universal moduli space $\widehat{\mathscr M}_{\Z_2}$ centered at $(p_0,(\mathcal Z_0, \ell_0,\Phi_0))$ such that the projection $\pi$ to the parameter space
\begin{center}
\tikzset{node distance=3.0cm, auto}
\begin{tikzpicture}
\node(C){$\mathcal P$};
\node(D)[yshift=-1.8cm][above of=C]{$\widehat {\mathscr M}_{\Z_2}  \subseteq \mathcal P \times \mathcal X $};
\node(E)[right of=D]{$\mathcal Y$};
\draw[<-] (C) to node {$\pi$} (D);
\draw[->] (D) to node {$\slashed{\mathbb D}_p$} (E);
\end{tikzpicture}
\end{center}
restricts to a homeomorphism from $\mathscr U_0$ to $\pi(\mathscr U_0)$, and the image $\pi(\mathscr U_0)$ posseses a Kuranishi chart of virtual codimension 1. 
\end{thm}

\medskip 
\noindent To possess a Kuranishi chart of virtual codimension 1 means that the set is locally modeled by the zero-locus of a smooth map $\kappa: \mathcal P\to \R$ (see e.g. Section 3.3 of \cite{DWDeformations}). In particular, if the map $(\refeq{mainamap})$ has trivial kernel, then $\kappa$ is transverse to 0 and $\pi(\mathscr U_0)$ is a smooth Fr\'echet submanifold of codimension 1. In either case, $\mathscr U_0$ also consists of regular $\Z_2$-harmonic spinors.

More generally, the universal eigenvalue problem has a spectral crossing along $\pi(\mathscr U_0)$:  
\begin{cor}\label{maind}

There is an open neighborhood $\mathscr V_0 \subseteq \mathcal P$ of $p_0$ possessing a Kuranishi chart of virtual codimension 0  such that for $p\in \mathscr V_0$ there exists triples  $(\mathcal Z_p, \Phi_p, \Lambda_p)$ defined implicitly as smooth functions of $p$ satisfying 
\be \slashed D_{\mathcal Z_p}\Phi_p =\Lambda_p \Phi_p \label{eigenvector}\ee
\noindent  for $\Lambda_p \in \R$ and such that $\pi(\mathscr U_0)=\Lambda^{-1}(0)$. 

 \end{cor}
 
 \noindent Of course, the triple coincides with $(\mathcal Z_0, \Phi_0, 0)$ at $p_0$. Analogous to Theorem \ref{mainb}, $\mathscr V_0$ consists of regular $\Z_2$-harmonic eigenvectors, and if the map $(\refeq{mainamap})$ has trivial kernel then $\mathscr V_0$ is an open neighborhood of $p_0$ and $\Lambda: \mathscr V_0\to \R$ is transverse to 0. Once again, the conclusion holds replacing $\mathcal P$ by any tame Fr\'echet submanifold $\mathcal P'\subseteq \mathcal P$ such that $\Lambda$ remains transverse.

\begin{rem}\label{smoothnessrem2}
Theorem \ref{mainb} provides additional evidence for Taubes's conjecture (Remark \ref{regularityrem}) that smoothness of the singular set is a generic property. It shows that smoothness is stable in the sense that if $(\mathcal Z_0, \ell_0, \Phi_0)$ is regular, then there exists a neighborhood $\mathscr U_0$ in the universal moduli space (\refeq{universalmoduli}) consisting of $\Z_2$-harmonic spinors whose singular sets are also smooth. Theorem \ref{mainb} does not rule out the possibility that there are also other points nearby $(\mathcal Z_0, \ell_0, \Phi_0)$ in the larger moduli space (\refeq{rectifiablemoduli}). 
\end{rem}

\subsection{Relations to Gauge Theory}
\label{section1.2}

$\Z_2$-harmonic spinors appear as limiting objects into two distinct settings in gauge theory: i) generalized Seiberg-Witten theory in 2,3, and 4 dimensions, and ii) Yang-Mills and enumerative theories on manifolds with special holonomy in 6,7, and 8 dimensions.  

\smallskip   
\subsubsection{Gauge Theory in Low-Dimensions.}

Generalized Seiberg--Witten theory unifies the majority of noteworthy equations in mathematical gauge theory \cite{DoanThesis, WalpuskiZhangCompactness}, including the standard Seiberg-Witten equations \cite{MorganSW, KM}, the Vafa-Witten equations \cite{VWOriginalPaper, TT1, TT2}, the Kapustin-Witten equations \cite{WittenKhovanovGaugeTheory,WittenFivebranesKnots, MazzeoWittenNahmI,MazzeoWittenNahmII}, the complex ASD equations \cite{Taubes4dSL2C,HaydysFSCategory}, and the ADHM-Seiberg-Witten equations \cite{HaydysG2SW, DWAssociatives}. Generalized Seiberg--Witten equations are systems of non-linear first order PDEs whose variables are a connection $A$ on a principal $G$-bundle for a compact Lie group $G$, and a spinor $\Psi$. 

 When there is an {\it a priori} bound on the $L^2$-norm of $\Psi$, as for the standard Seiberg--Witten equations, the moduli space of solutions (modulo gauge transformations) is compact. In general, the moduli space may include sequences of solutions along which the $L^2$-norm of $\Psi$ diverges. A variety of convergence theorems following pioneering work of Taubes \cite{Taubes3dSL2C} have shown that for many generalized Seiberg--Witten equations such sequences converge after renormalization to a $\Z_2$-harmonic spinor. Thus $\Z_2$-harmonic spinors are the natural candidates for constructing boundary strata to compactify these moduli spaces. 
 
 True $\Z_2$-harmonic spinors as defined in (\refeq{Z2prelimdef}) arise as limits of renormalized sequences of solutions to the 2-spinor Seiberg--Witten equations \cite{HWCompactness} (see \cite{PartI, PartIII} for more detailed exposition). Generalized $\Z_2$-harmonic spinors, for which the spinor bundle $S$ is replaced by an arbitrary Clifford module of real rank $4$, appear as limits of a variety of other equations \cite{Taubes3dSL2C, Taubes4dSL2C, TaubesKWNahmPole, TaubesVW, TaubesU1SW, WalpuskiZhangCompactness}.  In particular, the {\it limiting configurations} of Hitchin's equations (square roots of holomorphic quadratic differentials) on a Riemann surface are a dimensional reduction of (generalized) $\Z_2$-harmonic spinors; thus these objects extend well-studied phenomena on the boundary of the Hitchin moduli space \cite{MWWW, FredricksonSLnC, HitchinAsymptoticGeometry} to the higher-dimensional and non-holomorphic setting. The deformation theory of the case of $\Z_2$-harmonic 1-forms, for which one takes the Clifford module $S=\Omega^0(\R)\oplus \Omega^1(\R)$ has been treated by Donaldson \cite{DonaldsonMultivalued} using a reduction to a scalar equation (the analogue of which is not available in the present setting); various other cases in dimension 4 are the subject of forthcoming work \cite{CoassociativeDeformations}.

\subsubsection{Fueter Sections.}

The {\it Fueter equation} is a non-linear generalization of the Dirac equation on 3 and 4-manifolds for spinors taking values in a bundle of hyperk\"ahler orbifolds rather than a Clifford module \cite{PidstrigachFueter, TaubesNonlinearDirac}. Solutions of the Fueter equation are called {\bf Fueter Sections}. 

Fueter sections play a key role in proposals for constructing gauge-theoretic and enumerative invariants on manifolds with special holonomy in dimensions 6, 7, and 8. In particular, in both cases, they are expected to contribute terms to wall-crossing formulas which relate these theories to generalized Seiberg-Witten theories on low-dimensional calibrated submanifolds and compensate for losses of compactness as parameters vary. See \cite{DoanThesis, DonaldsonSegal, DWAssociatives, HaydysG2SW, HWCompactness, SiqiSLag, WalpuskiG2Gluing,WalpuskiSpin7Gluing} for more detailed exposition (all of which rely on the earlier work \cite{TianYangMillsCompactness,TianTaoCompactness}). In another directions, there are putative applications of Fueter sections to symplectic geometry \cite{WalpuskiFueterCompactness,SalamonHypercontact,SalamonFueter,DoanRezchikov}, and to constructing generalized Floer theories on 3-manifolds \cite{SamanThesis, SamanFueterFloer}. In all these cases, a well-developed theory of Fueter sections is lacking and many aspects remain speculative.

At least In the contexts of coming from gauge theory, it is expected that Fueter sections with singularities are unavoidable. Singularities arise when a Fueter section intersects the orbifold locus of the target hyperk\"ahler orbifold. The data of a $\Z_2$-harmonic spinors as defined in (\refeq{Z2prelimdef}) is equivalent to that of a Fueter section valued in the hyperk\"ahler orbifold $X=\mathbb H/\Z_2$ (see \cite{PartI} Section 2 or \cite{DWDeformations} Section 4 for details), with $\mathcal Z$ being the pre-image of the single orbifold point. For more general hyperk\"ahler orbifolds $X$ there is a stratification by stabilizer subgroups into subsets of codimension 4k, and a singular set arises where a Fueter section hits these strata. The reader is cautioned that even though these strata are codimension at least 4 and the base manifold $Y$ has dimension 3, the singular set of codimension 2 cannot be perturbed away when $\ell$ is non-trivial on small disks normal to $\mathcal Z_0$; indeed, Theorem \ref{mainb} confirms the singular set is stable under perturbations in this setting. Much of the work involving Fueter sections (e.g. \cite{DWDeformations,WalpuskiSpin7Gluing, WalpuskiG2Gluing, HaydysCorrespondence, PlattG2InstantonResolutions}) has dealt only with the case that $\mathcal Z=\emptyset$. This article contributes a step toward understanding Fueter sections with singularities.

\subsection{Outline}
\label{section1.3}
Sections \ref{section2}--\ref{section4} study the semi-Fredholm theory of the Dirac operator with a fixed singular set. Section \ref{section2} begins by establishing analogues of several standard results from elliptic theory, and Section \ref{section3} introduces the local ``polyhomogeneous'' expansions that replace the standard notion of elliptic regularity for the singular Dirac operator. Although many results in these first two sections are particular instances of more general results from the microlocal analysis of elliptic edge operators proved in \cite{MazzeoEdgeOperators,MazzeoEdgeOperatorsII} and subsequent work, we endeavor to give a largely self-contained exposition here. Section \ref{section4} studies the infinite-dimensional cokernel of the singular Dirac operator, and proves the isomorphism asserted following (\refeq{mainamap}) with a space of sections of a bundle on $\mathcal Z_0$. Section \ref{section4} contains (in the author's view) many of the more technical points in the article, and  some readers may prefer to read only the statements in Section 4 on a first pass.

 With the semi-Fredholm theory for fixed singular set established, Sections \ref{section5}--\ref{section6} proceed to study deformations of the singular set. Because the Dirac operator behaves naturally with respect to diffeomorphisms, deforming the singular set $\mathcal Z$ is equivalent to deforming the metric $g_0$ among the family of metrics that arise as pullbacks $F^*g_0$ by diffeomorphisms $F$ moving the singular set. Schematically,   
$$ \begin{pmatrix}\text{varying }\mathcal Z\\ \text{fixed }g_0\end{pmatrix}\hspace{1cm}\frac{\del}{\del \mathcal Z}\slashed D_{\mathcal Z} \ \   \ \overset{\text{pullback}}\Rightarrow \  \ \ \frac{\del}{\del g} \slashed D_{\mathcal Z_0}^g \hspace{1cm}\begin{pmatrix}\text{varying } g\\ \text{fixed }\mathcal Z_0\end{pmatrix}.$$  
\noindent The first-variation of the Dirac operator with respect to metrics on the right hand side is given by a well-known formula of Bourguignon-Gauduchon \cite{Bourguignon} discussed in Section \ref{section5}. Calculating the family of pullbacks by diffeomorphisms leads to an explicit formula for the derivative $\text{d}_{(\mathcal Z_0, \Phi_0)}\slashed{\mathbb D}$ of the universal Dirac operator (Corollary \ref{formofDD}). Section \ref{section6} proves Theorem \ref{maina} by calculating the projection to the cokernel using the description from Section \ref{section4}, during which the loss of regularity becomes manifest.

It is worth emphasizing that while there is a pleasing geometric reason for Theorem \ref{maina}, the fact that the operator (\refeq{mainamap}) is elliptic emerges quite miraculously from the formulas during the proof. Since differentiating the symbol does not preserve ellipticity, Bourguignon-Gauduchon's formula leads to a highly non-elliptic operator on $Y$; the content of Theorem \ref{maina} is to assert that under the isomorphisms from Section \ref{section4} associating this with an operator on sections of $\mathcal Z_0$, ellipticity somewhat surprisingly emerges! Theorem \ref{mainaprecise} provides a more technical version of Theorem \ref{maina}, and an explicit formula for the elliptic operator (\refeq{mainamap}) is given during the proof.

Sections \ref{section7}--\ref{section8} use Theorem \ref{maina} and a version of the Nash-Moser Implicit Function Theorem to prove Theorem \ref{mainb}. Section \ref{section7} gives a brief and practical introduction to Nash-Moser theory, and Section \ref{section8} shows that the universal Dirac operator satisfies the necessary hypotheses. The most challenging of these is to show that Theorem \ref{maina} persists on an open neighborhood of $(p_0,\mathcal Z_0, \Phi_0)$. In this, the difficulty is ensuring that some of the more subtle aspects of Sections \ref{section4} and \ref{section6} are stable.

\subsection*{Acknowledgements}

This article constitutes a portion of the author's Ph.D. thesis. The author is grateful to his advisors Clifford Taubes and Tomasz Mrowka for their insights and suggestions. The author would also like to thank Rafe Mazzeo, and Thomas Walpuski for many helpful discussions. This work was supported by a National Science Foundation Graduate Research Fellowship and by National Science Foundation Grant No. 2105512. It was also partially completed while the
author was in residence at the Simons Laufer Mathematical Sciences Institute (previously
known as MSRI) in Berkeley, California, during the Fall 2022 semester, supported by NSF
Grant DMS-1928930.
\section{Semi-Fredholm Properties}
\label{section2} 
Let $(Y,g_0)$ be a closed, oriented Riemannian 3-manifold, and fix a spin structure $\frak s\to Y$. The associated spinor bundle is denoted by $S_\frak s\to Y$, and its Clifford multiplication by $\gamma_\frak s: T^*Y \to \text{End}(S_\frak s)$. The real inner product on $S_\frak s$ is denoted by $\br \cdot , \cdot \kt$, and the spin connection by $\nabla^\text{spin}$. More generally \footnote{ The $\Z_2$-harmonic spinors arising as limits in gauge theory as in Section \ref{section1.2} may have $B\neq 0$}, consider connections of the form $\nabla_B=\nabla^\text{spin}+B$ where $B\in \Omega^1(\frak{so}(S))$ is a real-linear endomorphism. Fix a choice $B_0$ of such a perturbation.  

  Next, let $\mathcal Z_0\subset Y$ be a smoothly embedded link, i.e. a union of disjoint embedded copies of $S^1$. Choose a real line bundle $\ell_0 \to Y\setminus \mathcal Z_0$, and let $A_0$ denote the unique flat connection on $\ell_0$ with holonomy in $\Z_2$. Let $(S_0, \gamma_0, \nabla_0)$ be the Clifford module defined using the fixed pair $(g_0,B_0)$ by \be S_0:= S_\frak s \otimes_\R \ell_0 \hspace{2cm}\gamma_0=\gamma_\frak s\otimes 1 \hspace{2cm}\nabla_0= \nabla_{B_0} \otimes \text{Id} + 1 \otimes \nabla_{A_0}.  \label{cliffordmod}\ee
  This Clifford module carries a real-inner product still denoted $\br \cdot , \cdot \kt$, and a singular Dirac operator:
  \medskip  
  
  \begin{defn} \label{Z2Dirac}The {\bf  $\Z_2$-Dirac operator} associated to the Clifford module $(S_0, \gamma_0, \nabla_0)$ is defined on sections $\psi \in \Gamma(S_0)$ by  $$\slashed D_{\mathcal Z_0}\psi:= \gamma_0 (\nabla \psi). $$
  \noindent In contexts where no ambiguity will arise, we omit the subscripts and write $S, \gamma, \nabla,$ and $\slashed D$ for the objects associated to the data $(\mathcal Z_0, g_0, B_0)$  (though the latter three always retain their subscripts).   
  \end{defn}
  
  \medskip 
  \noindent In the case that  $B_0=0$ and $\ell_0$ extends over $\mathcal Z_0$ (and {\it a fortiori} if $\mathcal Z_0=\emptyset$), this is the classical spin Dirac operator associated to the spin structure  $\frak s'\to Y$ obtained from twisting $\frak s$ by $\ell_0$. The case of interest to us is that in which $\ell_0$ does not extend over $\mathcal Z_0$ and instead restricts to the mobius line-bundle on the normal planes of $\mathcal Z_0$. The first condition in Definition  \ref{regulardef} restricts to this case.    
  
  When $\ell_0$ does not extend over $\mathcal Z_0$, the Dirac operator $\slashed D$ is singular along $\mathcal Z_0$, and its extension to Banach spaces of sections is only semi-Fredholm. This section introduces Sobolev spaces of sections and describes the semi-Fredholm mapping properties of $\slashed D$.  More general versions of these results for larger classes of singular operators can be found in \cite{FangyunThesis, MazzeoEdgeOperators, DonaldsonMultivalued, Melrosebcalculus, grieser2001basics}. Here, we give a self-contained exposition.

\subsection{Edge Sobolev Spaces}
The ``edge'' Sobolev spaces provide a natural domain on which the singular Dirac operator $\slashed D$ is bounded (see \cite{MazzeoEdgeOperators} for definitions in the context of more general singular operators). Let $r$ denote a smooth weight function equal to $\text{dist}(-,\mathcal Z_0)$ on a tubular neighborhood of $\mathcal Z_0$  and equal to $1$ away from a slightly larger tubular neighborhood. For smooth sections compactly supported in $Y\setminus\mathcal Z_0$, define the $rH^1_e$ and $L^2$-norms respectively by  
\bea
\|\ph\|_{rH^1_e} &:=&\left(\int_{Y\setminus\mathcal Z_0} |\nabla \ph|^2 + \frac{|\ph|^2}{r^2} \ dV\right)^{1/2}\hspace{1.0cm}\text{ and }\hspace{1cm}
\|\psi\|_{L^2}:=\left( \int_{Y\setminus\mathcal Z_0} |\psi|^2 \ dV\right)^{1/2},
\eea
where $\nabla $ is the connection (\refeq{cliffordmod}) on $S$, and $dV$ denotes the volume form of the Riemannian metric $g_0$. In addition, define the $r^{-1} H^{-1}_e$-norm as the dual norm of $rH^1_e$ with respect to the $L^2$-pairing: 

$$  \|\xi\|_{r^{-1}H^{-1}_\e} = \sup_{\|\ph\|_{rH^1_e}=1} \br  \xi , \ph \kt_{L^2}.$$ 

\begin{defn} \label{functiondef} The {\bf edge Sobolev spaces}  of regularity $m=1,0,-1$ are defined respectively by 
\bea  rH^1_e(Y\setminus \mathcal Z_0 ; S)&:=& \{\ \ph  \ | \ \  \|\ph\|_{rH^1_e}  <\infty\}\\ 
L^{2}(Y\setminus \mathcal Z_0 ; S)&:=& \{ \ \psi  \  | \ \ \|\psi\|_{L^2}  <\infty\} \\
r^{-1}H^{-1}_e(Y\setminus \mathcal Z_0 ; S)&:=& \{ \ \xi \hspace{.6mm} \  | \ \ \|\xi\|_{r^{-1}H^{-1}_e}  <\infty\} \eea
i.e. as the completions of compactly supported smooth sections with respect to the above norms. When it is clear from the context, the domain $Y\setminus \mathcal Z_0$ and bundle $S$ are omitted for brevity. By construction, $r^{-1}H^{-1}_e=(rH^1_e)^\star$ is the dual space with respect to the $L^2$-pairing.     \end{defn}

\medskip

\noindent These spaces are equivalent for different choices of the weight function $r$ and of the pair $(g_0,B_0)$. Additionally, $rH^1_e$ and $L^2$ are Hilbert spaces with the inner products arising from the polarization of the above norms.  

 Although $Y\setminus \mathcal Z_0$ is not compact, the weight ensures following version of Rellich's Lemma holds, proved by a standard diagonalization argument.

 \begin{lm} 
The inclusion $$rH^1_e(Y\setminus \mathcal Z_0;S) \ \hookrightarrow \  L^2(Y\setminus \mathcal Z_0; S)$$
is compact.
\label{compactembedding} \qed
\end{lm}

\subsection{Mapping Properties}
The following proposition gives the fundamental mapping properties of the singular Dirac operator on the spaces defined in the previous subsection.  
\begin{prop}\label{injclosedrange}
The operator $$\slashed D: rH^1_e(Y\setminus \mathcal Z_0 ; S) \lre L^2(Y\setminus \mathcal Z_0 ; S).$$  is (left) semi-Fredholm, i.e. it satisfies:  
\begin{itemize}\item $\ker(\slashed D)$ is finite-dimensional, and
\item  $\text{range}(\slashed D)$ is closed. \end{itemize}

\label{semifredholm}
\end{prop}
\begin{proof} 
It is immediate from the definitions of $rH^1_e, L^2$ that $\slashed D$ is a bounded operator. Given $\ph\in rH^1_e$, it suffices to show that there is a constant $C$ such that the estimate 

\be \| \ph\|_{rH^1_e} \leq C \Big (\|\slashed D\ph\|_{L^2} + \|\ph\|_{L^2} \Big ) \label{semiellipticest} \ee
holds. Using the compactness of the embedding from Lemma \ref{compactembedding}, both conclusions of the lemma then follow from standard theory (see, e.g. \cite{LiviuLecturesOnGeometry} Section 10.4.1).

The estimate (\refeq{semiellipticest}) follows from the Weitzenb\"ock formula and integration by parts, as we now show, though some caution must be taken about the boundary term along $\mathcal Z_0$. Let $\ph\in rH^1_e$ be a spinor, and for each $n\in \N$ let $N_{1/n}(\mathcal Z_0)$ denote a tubular neighborhood of $\mathcal Z_0$ of radius $1/n$.  Additionally, let $\chi_n$ denote a cut-off function equal to 1 on $Y\setminus N_{1/n}(\mathcal Z_0)$ and compactly supported in $N_{2/n}(\mathcal Z_0)$ satisfying $$|d\chi_n|\leq \frac{C}{n}\leq \frac{C'}{r}.$$ 

\noindent Then, integrating by parts and using that $\slashed D$ is formally self-adjoint,  
\bea
\int_{Y\setminus \mathcal Z_0} |\slashed D\ph|^2  \ dV &=& \xlim{n}{\infty} \int_{Y\setminus \mathcal Z_0} \br \slashed D\ph, \slashed D\ph\kt \chi_n \ dV \\ 
&=&  \xlim{n}{\infty} \int_{Y\setminus \mathcal Z_0} \br  \ph, \slashed D \slashed D\ph\kt \chi_n \ + \br \ph, \gamma(d\chi_n)\slashed D\ph\kt \ dV.  
\eea 

\noindent The Weitzenb\"ock formula  shows that $$\slashed D \slashed D=\nabla^\star \nabla +F$$ wherein $F$ is a zeroth order term arising from the scalar curvature and the derivatives of the perturbation $B_0$. Substituting this and integrating by parts again yields 

$$\int_{Y\setminus \mathcal Z_0} |\slashed D\ph|^2  \ dV  = \int_{Y\setminus \mathcal Z_0} |\nabla \ph|^2  + \br\ph, F\ph \kt +  \xlim{n}{\infty} \int_{Y\setminus \mathcal Z_0}  \br \ph, d\chi_n \cdot \nabla \ph + \gamma(d\chi_n)\slashed D\ph\kt \ dV$$ 

\noindent where $\cdot$ denotes contraction of 1-form indices. Since $F$ is smooth on $Y$ hence uniformly bounded, rearranging and using Young's inequality yields
\begin{eqnarray} \int_{Y\setminus \mathcal Z_0} |\nabla \ph|^2 \ dV & \leq&  C\left(\|\slashed D\ph\|^2_{L^2} + \|\ph\|^2_{L^2} +  \xlim{n}{\infty} \int_{N_{2/n}(\mathcal Z_0)} |\nabla \ph|^2 + |d\chi_n|^2 |\ph|^2 \ dV\right ) \\
 &\leq &   C\left(\|\slashed D\ph\|^2_{L^2} + \|\ph\|^2_{L^2} +  \xlim{n}{\infty} \int_{N_{2/n}(\mathcal Z_0) } |\nabla \ph|^2 + \frac{|\ph|^2}{r^2} \ dV  \right).
 \label{bdtermlim}
 \end{eqnarray}
 
\noindent Since $\ph \in rH^1_e$, the latter limit vanishes, hence 

\be \|\nabla \ph\|_{L^2} \leq C \Big (\|\slashed D\ph\|_{L^2} + \|\ph\|_{L^2} \Big ). \label{nablaest} \ee

To conclude, we show the left-hand side of (\refeq{nablaest}) dominates the $rH^1_e$ norm. For $n$ sufficiently large, choose local coordinates on $N_{1/n} (\mathcal Z_0)\simeq S^1\times D_{1/n}$. Denote these by $(t,r,\theta)$ for $t$ the coordinate on the $S^1$ factor and $(r,\theta)$ polar coordinates on $D_{1/n}$. For each fixed $t_0, r_0$, the fact that the holonomy around the loop $(t_0, r_0, \theta)$ for $\theta\in [0,2\pi)$ is $-1$ implies that the operator $\nabla_\theta$ has lowest eigenvalue $1/2$ on this loop (see the local expressions in Section \ref{section3.1}). It follows that 

\be \int_{N_{1/n}(\mathcal Z_0)} \frac{|\ph|^2}{r^2} \ dV \leq \frac{1}{4}\int_{N_{1/n}(\mathcal Z_0)} \frac{1}{r^2}|\nabla_\theta\ph|^2  \ dV \leq \frac{1}{4} \|\nabla \ph\|^2_{L^{2}}, \label{Hardy2d}\ee
 
 \noindent on $N_{1/n}(\mathcal Z_0)$, and away from this neighborhood the weight $r$ is uniformly bounded. Combining this estimate with  (\refeq{nablaest})  (possibly increasing the constant) yields  (\refeq{semiellipticest}), completing the lemma. 
\end{proof}

The proof of Proposition \ref{injclosedrange} shows that the space $rH^1_e(Y\setminus \mathcal Z_0; S)$ is equivalent to the space whose norm only includes $|\nabla \ph|^2$, which dominates the second term in the $rH^1_e$-norm by (\refeq{Hardy2d}) It follows that the integrability condition in (\refeq{Z2prelimdef}) holds if and only if $\Phi \in rH^1_e$. We conclude:
\smallskip 

\begin{lm}
A non-zero spinor $\Phi$ is a $\Z_2$-harmonic spinor if and only if it is in the kernel of the operator \be\slashed D: rH^1_e(Y\setminus \mathcal Z_0; S) \lre L^2(Y\setminus \mathcal Z_0 ; S)\label{DL^12}.\ee

\label{Z2harmonicdef}\qed
 \end{lm}

\medskip

 Note that although the estimate (\refeq{semiellipticest}) resembles the standard bootstrapping inequality, it does not imply that an $L^2$ solution of $\slashed D\psi=0$ necessarily lies in $rH^1_e$. In order to establish (\refeq{semiellipticest}) it was necessary to assume {\it a priori} that $\ph \in rH^1_e$, else the boundary term along $\mathcal Z_0$ need not vanish and the proof fails. Since $\slashed D$ is uniformly elliptic on any compact subset  $K\subset Y\setminus \mathcal Z_0$, however, standard theory applies to show that $\ph \in rH^1_{loc}$ (in fact $C^\infty_{loc}$) but the pointwise norm need not be integrable as $r\to 0$. Indeed, as we will see in Section \ref{section4}, the $rH^1_e$-kernel and the $L^2$-kernel are genuinely different spaces, with the latter infinite-dimensional. The term $\Z_2$-harmonic spinor refers only to non-zero kernel elements in $rH^1_e$.

\subsubsection{The Adjoint Operator.} Although the cokernel of (\refeq{DL^12}) is not necessarily finite-dimensional as in standard elliptic theory, it can still be described as the solutions of the formal adjoint operator. As in the proof of Lemma \ref{semifredholm}, formal self-adjointness of $\slashed D$ and integration by parts shows that the relation  \be \br \slashed Dv, \ph \kt_{L^2}= \br v, \slashed D \ph \kt_{L^{2}}\label{adjointsdef}\ee
 
\noindent  holds for $v,\ph \in rH^1_e$. As a consequence of (\refeq{adjointsdef}), the Dirac operator extends to a bounded map 
$$\slashed D: L^2(Y\setminus \mathcal Z_0; S)\lre r^{-1}H^{-1}_e(Y\setminus \mathcal Z_0; S),$$

 \noindent where for $v\in L^2$, the spinor $\slashed D v \in r^{-1}H^{-1}_e$ is the linear functional defined by the relation (\ref{adjointsdef}). To emphasize the domain of definition of various manifestations of the Dirac operator, we write $\slashed D|_{rH^1_e}$ or $\slashed D |_{L^2}$. 
 
 We then have the following: 
\begin{lm}\label{L^2splitting}
The extension $\slashed D|_{L^2}$ defined by (\refeq{adjointsdef}) is the (true) adjoint of $\slashed D|_{rH^1_e}$, and there is a closed orthogonal decomposition   
$$L^2(Y\setminus \mathcal Z_0; S)= \ker (\slashed D|_{L^2}) \ \oplus  \ \text{range}(\slashed D |_{rH^1_e}).$$
\end{lm}
\begin{proof}
Suppose that $\psi \in L^2$ is perpendicular to the range, i.e. $\br \psi,  \slashed D \ph \kt_{L^2}=0$ for all $\ph\in rH^1_e$. The definition of $\slashed D|_{L^2}$ via (\refeq{adjointsdef}) shows that as a linear functional on $rH^1_e$, one has $\slashed D\psi =0.$
\end{proof}

\subsubsection{The Second Order Operator.} The (left) semi-Fredholmness of $\slashed D$ implies that the second order operator $\slashed D \slashed D$ is Fredholm for purely formal reasons. More precisely, we have the following lemma. 
\medskip

\begin{lm}\label{secondorderoperator}The second order operator $\slashed D \slashed D: rH^1_e(Y\setminus \mathcal Z_0;S)\lre r^{-1}H^{-1}(Y\setminus \mathcal Z_0; S)$ 
is Fredholm and $\ker(\slashed D\slashed D)= \ker(\slashed D|_{rH^1_e})\simeq \text{coker}(\slashed D\slashed D)$. In particular, there is is a constant $C$ such that  the elliptic estimate \be\|\ph\|_{rH^1_e} \leq C (\|\slashed D\slashed D\ph\|_{r^{-1}H^{-1}_e} + \|\pi_1(\ph)\|_{L^2}).\label{secondorderest}\ee
holds, where $\pi_1(\ph)$ is the $L^2$-orthogonal projection onto $\ker(\slashed D|_{rH^1_e})$.  
\end{lm} 

\begin{proof}(Cf. \cite{RyosukeThesis} Proposition 4.4) 
By definition of $\slashed D|_{L^2}$ via (\refeq{adjointsdef}), if $\ph \in rH^1_e$ and $\ph \in \ker(\slashed D\slashed D)$,  then
$$ 0 = \br \slashed D\slashed D \ph , \ph \kt_{L^2} = \|\slashed D\ph \|_{L^2}$$  
\noindent hence $\ph \in \ker(\slashed D|_{rH^1_e})$, which is finite dimensional by Proposition~\ref{semifredholm}.  

To show that the range is closed and the cokernel finite-dimensional (and naturally isomorphic to $\ker(\slashed D|_{rH^1_e})$), let $f\in r^{-1}H^{-1}_e$ be such that  $\br f , \Phi\kt=0$ for all $\Phi \in \ker(\slashed D|_{rH^1_e})$. Consider the functional $E_f: rH^1_e\to \R$ given by $$E_f(\ph):= \int_{Y\setminus\mathcal Z_0} |\slashed D\ph |^2 - \br \ph , f\kt \ dV.$$

\noindent The Euler-Lagrange equation of $E_f$ is $$\slashed D\slashed D\ph= f$$
so it suffices to show that $E_f$ admits a minimizer. By standard theory (\cite{EvansGSM} Chapter 8) this holds if $E_f$ is (i) coercive, and 
(ii) weakly lower semi-continuous. The second of these is standard (see e.g.  \cite{EvansGSM} Section 8.2.2). (i) means that \be E_f(\ph)\geq c_1\|\ph\|_{rH^1_e}^2 - c_2\label{coercivitiy}\ee holds for some constants $c_i$, and $\ph$ in the $ L^2$-orthogonal complement of $\ker(\slashed D|_{rH^1_e})$, which follows from the elliptic estimate (\refeq{semiellipticest}) of Proposition \ref{semifredholm} and Young's inequality. Since we require $\br f,\Phi \kt=0$ for $\Phi\in \ker(\slashed D|_{rH^1_e})$, it follows that $\dim \text{coker}(\slashed D\slashed D)\leq \dim \ker(\slashed D|_{rH^1_e})$, and integration by parts establishes equality. This proves Fredholmness, and the estimate (\refeq{secondorderest}) is a routine consequence. 
\end{proof}

As a consequence of the preceding lemma, we may define $P_0: r^{-1}H^{-1}_e\to rH^1_e$ as the solution operator given by \begin{eqnarray} P_0(\xi)= \ph \hspace{1.5cm} \text{s.t.} & & \text{i)} \  \slashed D \slashed D\ph=\xi \mod \ker(\slashed D|_{rH^1_e})  \label{Pdef}\\ \text{and} & & \text{ii)} \  \br \ph, \Phi\kt_{L^2}=0 \ \forall \Phi\in \ker(\slashed D|_{rH^1_e}). \end{eqnarray} 

\noindent As with $\slashed D$, the subscript is omitted and we simply write $P$ when it is clear from context. 

To summarize, we have the following corollary: 

\begin{cor} \label{Pdiagram} The following hold using the splitting $L^2=\ker(\slashed D|_{L^2})\oplus \text{range}(\slashed D|_{rH^1_e})$ of Lemma \ref{L^2splitting}.   

\begin{enumerate}
\item[(A)] The second order operator $\slashed D \slashed D$ factors through the $\text{range}(\slashed D|_{rH^1_e})$ summand of 

\begin{center}
\tikzset{node distance=3.5cm, auto}
\begin{tikzpicture}
\node(C)[yshift=-1.7cm]{$rH^1_e$};
\node(B')[right of=C]{$\begin{matrix} \ker(\slashed D|_{L^2})\\ \oplus \\ \text{range}(\slashed D|_{rH^1_e})\end{matrix}$};
\node(D)[right of=B']{$r^{-1}H^{-1}_e.$};
\draw[->] (B') to node {$\slashed D$} (D);
\draw[->] (C) to node {$\slashed D$} (B');
\draw[<-,>=stealth] (C) edge[bend right=30]node[below]{$P$} (D);
\end{tikzpicture}
\end{center}
\noindent In addition, we can further split $\ker(\slashed D|_{L^2})=\ker(\slashed D|_{rH^1_e})\oplus \ker(\slashed D|_{L^2})^\perp$ wherein the first summand is finite-dimensional.  
\item[(B)] The projections $\Pi_0, 1-\Pi_0$ to $\ker(D|_{L^2})$ and its orthogonal complement may be written 
$$1-\Pi_0= \slashed D P \slashed D \hspace{2cm}  \Pi_0=1-\slashed D P \slashed D.$$\qed
\end{enumerate}
\end{cor}

\subsection{Higher Regularity}
  This subsection extends the results of the previous two subsections  to ``edge'' and ``boundary'' Sobolev spaces of higher regularity  (see \cite{MazzeoEdgeOperators} again for a more general exposition). Beginning with the ``boundary'' spaces, define the space of ``boundary'' vector fields $$\mathcal V_\text{b}:=\{ V \in C^\infty(Y;TY) \  \ | \ \ \  V|_{\mathcal Z_0} \in C^\infty(\mathcal Z_0; T\mathcal Z_0)\}$$ 
\noindent as those tangent to $\mathcal Z_0$ at the boundary. Let $\nabla^\text{b}$ denote the covariant derivative with respect to such vector fields, so that in local coordinates $(t, x,y)$ where $t$ is a coordinate along $\mathcal Z_0$ and $x,y$ are coordinates in the normal directions it is given by 
$$\nabla^\text{b}=dx\otimes  r \nabla_x \ +  \ dy\otimes r \nabla_y \  + \ dt \otimes  \nabla_t $$
\noindent  and is equal to the standard covariant derivative $\nabla$ away from $\mathcal Z_0$. 

For $m \in \N$, define the $H^m_\text{b}$-norm on compactly supported smooth sections by 
\begin{eqnarray}
\|\psi\|_{H^m_\text{b}} &:=&\left(\int_{Y\setminus\mathcal Z_0} |(\nabla^\text{b})^m \psi|^2   \ + \ldots +  |\nabla^\text{b}\psi|^2  \ + \ |\psi|^2 \ dV\right)^{1/2}.\label{bnorm}
\end{eqnarray}

\begin{defn}
The {\bf mixed boundary and edge Sobolev spaces} are defined by
\bea  rH^{m,1}_{\text{b},e}(Y\setminus \mathcal Z_0 ; S)&:=& \Big\{\ph \ | \ \|(\nabla^\text{b})^m\ph\|^2_{rH^1_e}  + \ldots + \|\nabla^\text{b}\ph\|^2_{rH^1_e}   \ +  \ \|\ph\|^2_{rH^1_e}  <\infty\Big\}\\ 
H^m_\text{b}(Y\setminus \mathcal Z_0 ; S)&:=& \Big\{\psi \ | \  \|(\nabla^\text{b})^m\psi\|^2_{L^2}  + \ldots + \|\nabla^\text{b}\psi\|^2_{L^2}   \ +  \ \|\psi\|^2_{L^2} =\|\psi\|_{H^m_\text{b}}^2 <\infty\Big\} \\
r^{-1}H^{m,-1}_{\text{b},e}(Y\setminus \mathcal Z_0 ; S)&:=&\Big\{\xi \ | \ \|(\nabla^\text{b})^m\xi\|^2_{r^{-1}H^{-1}_e}  + \ldots + \|\nabla^\text{b}\xi\|^2_{r^{-1}H^{-1}_e}   \ +  \ \|\xi\|^2_{r^{-1}H^{-1}_e}  <\infty\Big\} \eea
\end{defn}

\noindent equipped with the norms given by the positive square root of the quantities required to be finite. As for $m=0$, changing the weight $r$ or $(g_0,B_0)$ results in equivalent norms. More generally, one can define the spaces for $m \in \R^{\geq 0}$ by interpolation.  

\medskip 

We have the following version of the standard interpolation inequalities: 

 \begin{lm}\label{interpolationinequalities} The following interpolation inequalities hold for $m_1 < m< m_2$:  
$$ \|\psi\|_{H^m_\text{b}} \leq C \| \psi \|^\alpha _{H^{m_1}_\text{b}} \|\psi\|_{H^{m_2}_\text{b}}^{1-\alpha} \hspace{1cm}\text{\   \ }\hspace{1cm} \|\ph\|_{H^{m,1}_{\text{b},e}} \leq C \| \ph \|^\alpha _{H^{m_1,1}_{\text{b},e}}  \|\ph\|_{H^{m_2,1}_{\text{b},e}}^{1-\alpha}$$

\noindent where $\alpha= \frac{m_2-m}{m_2-m_1}$, and the constants may depend on the triple $m_1, m, m_2$. 
\end{lm}

\begin{proof}
Choose local cylindrical coordinates $(t, r, \theta)$ on a tubular neighborhood of $\mathcal Z_0$, where $t$ a coordinate along $\mathcal Z_0$ and $(r,\theta)$ polar coordinates in the normal directions. The coordinate change $s=\log(r)$ is a diffeomorphism between $Y\setminus \mathcal Z_0$ and the manifold $Y^\circ$ given by attaching a cylindrical end $T^2\times (-\infty, r_0)$ near $\mathcal Z_0$.  Under this coordinate change, $H^m_\text{b}$ is taken to the standard Sobolev spaces $e^{-s}H^m$ with the an exponential weight. After multiplying by an exponential weight function, the inequalities for $H^m_\text{b}$ follow from the standard ones on $Y^\circ$ (see, e.g. \cite{GagliardoNirenbergInterpolationRef}). 

For the mixed boundary and edge spaces, note that $\|[\nabla, \nabla^b] \ph\|_{L^2}\leq \|\nabla \ph\|_{L^2}$, and  iterating these commutators shows that \be \|\ph\|^2_{H^{m,1}_{\text{b},e}}  \sim \|\nabla \ph\|^2_{H^m_\text{b}} + \|\tfrac{\ph}{r}\|^2_{H^m_\text{b}}\label{equivalentbenorm}\ee
is an equivalent expression for the norm, after which the interpolation inequalities for $H^{m,1}_{\text{b},e}$ follow from those for $H^m_\text{b}$ applied to $\nabla \ph$ and $\tfrac\ph r$.  
\end{proof}

Applying the elliptic estimate   (\refeq{semiellipticest}) to $(\nabla^\text{b})^m\ph$ and iterating commutators $[\nabla, \nabla^\text{b}]$ also establishes the following higher-regularity elliptic estimates: 

\begin{cor}\label{higherordersemielliptic}There are constants $C_m$ depending on up to $m+3$ derivatives of the pair $(g_0,B_0)$ such that the following elliptic estimates hold for $\ph \in rH^{m,1}_{\text{b},e}$: 
\bea
\|\ph\|_{rH^{m,1}_{\text{b},e}} &\leq& C_m(\|\slashed D \ph\|_{H^m_\text{b}} + \|\ph\|_{H^m_\text{b}} ) \\ 
\|\ph\|_{rH^{m,1}_{\text{b},e}} &\leq& C_m(\|\slashed D\slashed D \ph\|_{r^{-1}H^{m,-1}_{\text{b},e}} + \|\ph\|_{r^{-1}H^{m,-1}_{\text{b},e}} )\eea\qed
\end{cor}

Using this, we immediately deduce the higher-regularity version of Corollary \ref{Pdiagram}.

\begin{cor} \label{higherregularitymappings}For all $m>0$, the following statements hold: 

\begin{enumerate}
\item[(A)] There is an $H^m_\text{b}$-closed decomposition $$H^m_\text{b}= \ker(\slashed D|_{H^m_\text{b}}) \oplus  \text{range}(\slashed D|_{rH^{m,1}_{\text{b},e}}) $$
\noindent orthogonal with respect to the $L^2$-inner product. Moreover, the latter two spaces coincide with $ \ker(\slashed D|_{H^m_\text{b}}) = \ker(\slashed D|_{L^2})\cap H^m_\text{b}$ and $\text{range}(\slashed D|_{rH^{m,1}_{\text{b},e}})= \text{range}(\slashed D|_{rH^1_e})\cap H^m_\text{b}$. 

\item[(B)] The second order operator $\slashed D \slashed D$ factors through the $\text{range}(\slashed D|_{rH^1_e})\cap H^m_\text{b}$ summand of 

\begin{center}
\tikzset{node distance=3.5cm, auto}
\begin{tikzpicture}
\node(C)[yshift=-1.7cm]{$rH^{m,1}_{\text{b},e}$};
\node(B')[right of=C]{$\begin{matrix} \ker(\slashed D|_{L^2})\cap H^m_\text{b}\\ \oplus \\ \text{range}(\slashed D|_{rH^1_e})\cap H^m_\text{b} \end{matrix}$};
\node(D)[right of=B']{$r^{-1}H^{m,-1}_{\text{b},e}.$};
\draw[->] (B') to node {$\slashed D$} (D);
\draw[->] (C) to node {$\slashed D$} (B');
\draw[<-,>=stealth] (C) edge[bend right=30]node[below]{$P$} (D);
\end{tikzpicture}
\end{center}
\item[(C)] The projections to the two summands in Item (B) of Corollary \ref{Pdiagram} respect regularity in the sense that  
$$(1-\Pi_0)= \slashed D P \slashed D: H^m_\text{b}\to H^m_\text{b} \hspace{2cm}  \Pi_0=1-\slashed D P \slashed D:H^m_\text{b}\to H^m_\text{b} $$
\noindent are bounded operators.  \qed

\end{enumerate}

\end{cor}

\section{Local Expressions}
\label{section3}
This section studies the expressions for the Dirac operator and its solutions in local coordinates on a tubular neighborhood of $\mathcal Z_0$. By Proposition \ref{semifredholm} and Lemma \ref{L^2splitting}, there is a dichotomy between two distinct types of solution: 
\be\Phi \in \ker(\slashed D|_{rH^1_e}) \hspace{2cm} \psi \in \ker(\slashed D|_{L^2}) \text { \ \  s.t. \ \  }\psi \notin  r H^1_e\label{cokdichotomy}\ee
\noindent with the former being the finite-dimensional space of $\Z_2$-harmonic spinors. 

It is instructive to first consider the model case of $Y_\circ=S^1\times \R^2$ equipped with the product metric, which is done in Section \ref{section3.1}. Sections \ref{section3.2} and \ref{section3.3} then deal with local expressions on a general 3-manifold.

\subsection{The Model Operator}
\label{section3.1}

Let $Y_\circ=S^1 \times \R^2$ with coordinates $(t,x,y)$ and equipped with the product metric $g_0=dt^2  + dx^2  + dy^2$. Take $\mathcal Z_0= S^1\times \{0\}$ and $\ell_0\to Y_0\setminus\mathcal Z_0$ the pullback of the mobius bundle on $\R^2\setminus\{0\}$ by the projection to the second factor.

The twisted spinor bundle of the product spin structure can be identified with $S=\underline \C^2 \otimes_\R \ell_0$. A section $\psi \in \Gamma(\underline \C^2 \otimes_\R \ell)$ may be written as \be \psi= e^{i\theta/2}\begin{pmatrix}\psi^+  \\   \psi^-  \end{pmatrix} \label{sectionforma}\ee

\noindent where $\psi^\pm$ are $\C$-valued functions and $(r,\theta)$ are polar coordinates on $\R^2$. Indeed, on each normal plane $\R^2\setminus\{0\}$, the bundle $\underline\C\otimes_\R \ell$ can be constructed as the bundle with fiber $\C$ glued along two (thickened) rays by the transition functions $+1$ and $-1$. Consequently, $e^{i\theta/2}$, gives rise to a global nowhere-vanishing section of this bundle. When a section is written in the form (\refeq{sectionforma}), the connection arising from the spin connection and $\nabla_{A_0}$ on $\ell_0$ (with perturbation $B_0=0$) is simply $\nabla=\text{d}$.  The Dirac operator then takes the form 
\begin{equation}\slashed D=\begin{pmatrix} i \del_t & -2\del_z \\ 2\delbar_z & -i\del_t \end{pmatrix} \label{dirac}\end{equation}  
\noindent where $z=x+iy$. That is to say, it is just the normal spin Dirac operator on $Y_0$, but the spinors have an additional $e^{i\theta/2}$ term which is differentiated as expected.

\begin{rem}
Although it is convenient for computation, the expression (\refeq{dirac}) hides the singular nature of the Dirac operator. It can alternatively be written in the following equivalent way which makes the singular nature manifest. 

Multiplication $e^{-i\theta/2}: \underline \C^2\otimes \ell_0 \simeq \underline \C^2$ provides an alternative trivialization, in which spinor are written $\psi=(\psi^+, \psi^-)$
\noindent where $\psi^\pm$ are still $\C$-valued functions. In this trivialization, Dirac operator is instead given by $$\slashed D=\begin{pmatrix} i \del_t & -2\del_z \\ 2\delbar_z & -i\del_t \end{pmatrix} +\frac{1}{2} \gamma(d\theta)=\begin{pmatrix} i \del_t & -2\del_z \\ 2\delbar_z & -i\del_t \end{pmatrix}+\frac{1}{4}\gamma\left(\frac{dz}{z}-\frac{d\overline z}{\overline z}\right)$$ where $\gamma$ denotes Clifford multiplication. Thus $\slashed D$ is a uniformly elliptic operator plus a singular zeroth order term (i.e. one unbounded on $L^2$). Equivalently, $r\slashed D$ is an elliptic operator with bounded zeroth order term, but {\it the symbol degenerates along $\mathcal Z_0$.} This type of operator is called an {elliptic operator with edge-type degeneracies} or simply an {\bf elliptic edge operator}. The theory of operators of this type has been studied extensively in microlocal analysis and many results in Section \ref{section2} hold in considerable generality (see \cite{FangyunThesis, MazzeoEdgeOperators, DonaldsonMultivalued, Melrosebcalculus, grieser2001basics} and the references therein). \label{singularoperator}
\end{rem}

\begin{eg}
\label{euclideanexample}

Let us now identify the $L^2$-kernel of $\slashed D$ on $Y_\circ$ (Cf. \cite{RyosukeThesis} Section 3). As in Lemma \ref{L^2splitting}, this also identifies the cokernel of the operator on $rH^1_e$ since $\text{coker}(\slashed D|_{rH^1_e})\simeq \text{ker}(\slashed D|_{L^2})$ continues to hold. Here, the weight function is given by $r$ globally on $Y_\circ$. Writing a general section in Fourier series as $$\psi=\sum_{k,\ell} e^{i\ell t} e^{i(k+\tfrac{1}{2})\theta} \begin{pmatrix} \psi^+_{k,\ell}e^{-i\theta} \\ \psi^-_{k,\ell} \ \ \ \ \  \end{pmatrix}$$
and using the polar expressions  $$\delbar_z = \frac{1}{2}e^{i\theta}(\del_r + \frac{i}{r}\del_\theta) \hspace{2cm}\delbar_z = \frac{1}{2}e^{-i\theta}(\del_r - \frac{i}{r}\del_\theta),$$
\noindent the Dirac equation (\ref{dirac}) becomes the following system of ODEs for $\psi_{k,\ell}^\pm(r)$ which decouple for distinct pairs $(k,\ell)$: 
\begin{equation}
\d{}{r}\begin{pmatrix}\psi_{k,\ell}^+\vspace{.25cm} \\ \psi_{k,\ell}^-\end{pmatrix}= \begin{pmatrix}\frac{(k-\tfrac{1}{2})}{r} & -\ell \\ -\ell & -\frac{(k+\tfrac{1}{2})}{r}\end{pmatrix}\begin{pmatrix}\psi_{k,\ell}^+ \vspace{.25cm}\\ \psi_{k,\ell}^-\end{pmatrix}.
\label{ODE}
\end{equation}

This system of equations can be solved by substituting the second equation into the first, after which the general solution is given in terms of modified Bessel functions (of the second kind). If $k\neq 0$, the pair $(k,\ell)$ admits no solutions in $L^2(S^1\times \R^2)$; for $(k,\ell)=(0,\ell)$ with $\ell\neq 0$,

\be \Psi_\ell^\circ=\sqrt{|\ell|}e^{i\ell t} e^{-|\ell|r}\begin{pmatrix} \frac{1}{\sqrt{z}}  \\  \frac{\text{sgn}(\ell)}{\sqrt{\overline z}}
\end{pmatrix}\label{euclideancokernel} \ee
 
\noindent  is an infinite-dimensional set of orthonormalized solutions in $L^2$, and $\ker(\slashed D|_{L^2})$ is their $L^2$-closure. Indeed, there can be no other solutions since (\refeq{euclideancokernel}) and equivalent expression with the modified Bessel function of the first kind (which is not $L^2$) exhaust the possible solutions in each Fourier mode by standard ODE theory.

Disregarding the issues of the integrability of the $\ell=0$ solutions as $r\to \infty$ (which is immaterial in the upcoming case of $Y$ compact) and formally including this element leads to an isomorphism
\be L^2(S^1;\C) \simeq \text{ker}(\slashed D|_{L^{2}})\label{cokeriso}\ee defined by the linear extension of $e^{i\ell t}\mapsto \Psi_\ell^\circ$. In this example there are no $\Z_2$-harmonic spinors. 
  
  There is a second choice of spin structure on $Y_\circ=S^1 \times \R^2$ which has monodromy $-1$ around the $S^1$ factor parallel to $\mathcal Z_0$ in addition to around the meridian. For this second spin structure, spinors may be written with half integer Fourier modes $e^{i \ell t} e^{it/2}$, and the calculation is identical but the solutions are indexed by $\ell' \in \Z+\tfrac12$.   \qed

\end{eg}
\bigskip 

Example \ref{euclideanexample} suggests that the $L^2$-kernel of $\slashed D$ on a closed 3-manifold is also infinite-dimensional, and thus the failure to prove Fredholmness of $\slashed D$ in Section \ref{section2}  was not simply a shortcoming of the techniques employed. Indeed, this will be shown to be the case. In fact, besides simply being infinite-dimensional, $\ker (\slashed D|_{L^2})$  displays the following salient properties in the model case, which generalize to the case of $Y$ closed: 
\medskip 
\begin{center}

\begin{minipage}{5.15in}
 \begin{enumerate}
\item[{\bf Expansion}:] Solutions $\Psi_\ell^\circ$ have asymptotic expansions with terms $r^{k-\tfrac12}$ for $k\in \Z$. 
\item[{\bf Isomorphism}:] \label{isoprop}There is an isomorphism $\ker(\slashed D|_{L^2})\simeq  L^2(\mathcal Z_0; \mathcal \C)$ given by associating a kernel element to each eigenfunction of the Dirac operator $i\del_t$ on $\mathcal Z_0$. 
\item[{\bf Rapid Decay}:]  \label{rapiddecayprop}For eigenvalues $|\ell|>>0$, solutions $\Psi_\ell^\circ$ decay exponentially away from $\mathcal Z_0$. 
\end{enumerate}
\end{minipage}
\end{center}
\medskip 
\noindent The first item follows from the power series expansion of $e^{-|\ell| r}$. The remainder of Section \ref{section3} defines and establishes the asymptotic expansions in the first item more precisely, while precise statements and proofs of the second and third items are the subject of Section \ref{section4}. 

\medskip 
\begin{rem} There are no $\Z_2$-harmonic spinors in Example \ref{euclideanexample} because there are no solutions with finite $rH^1_e$-norm. There are, however, still explicit solutions given in terms of modified Bessel functions for $(k,\ell)= (\pm 1, \ell)$ which have leading order $z^{1/2}$ and $\overline z^{1/2}$ which lie in $r H^1_{loc}$ near $\mathcal Z_0$ but grow exponentially as $r\to \infty$. Therefore, intuitively, the existence of a $\Z_2$-harmonic spinor on a closed manifold $Y$ is a rare phenomenon and occurs only when one of these exponentially growing solutions can be patched together with a bounded solution on the complement of a neighborhood of $\mathcal Z_0$ in $Y$. 
\end{rem}
\subsection{Local Expressions}
\label{section3.2}

From here on, we return to the case that $(Y,g_0)$ is a closed, oriented Riemannian 3-manifold and $\mathcal Z_0$ a smoothly embedded link. In order to write local expressions, we endow a tubular neighborhood $N_{r_0}(\mathcal Z_0)$ of (a component of) $\mathcal Z_0$ with a particular set of coordinates. 

Let $\gamma: S^1\to \mathcal Z_j$ be an arclength parameterization of a chosen component $\mathcal Z_j$ of $\mathcal Z_0$ whose length is denoted by $|\mathcal Z_j|$, and fix a global orthonormal frame $\{n_1, n_2\}$ of the pullback $\gamma^*N\mathcal Z_0$ of the normal bundle to $\mathcal Z_0$. We are free to arrange that $\{\dot \gamma, n_1, n_2\}$ is an oriented frame of $TY$ along $\mathcal Z_j$.

\begin{defn}
\label{fermi}

A system of {\bf Fermi coordinates} for $r_0 < r_{\text{inj}}$ where $r_{\text{inj}}$ is the injectivity radius of $Y$ is the diffeomorphism $S^1\times D_{r_0}\simeq N_{r_0}(\mathcal Z_i)$ for a chosen component of $\mathcal Z_0$ given by $$(t,x,y)\mapsto \text{Exp}_{\gamma(t)}(xn_1 + yn_2),$$

\noindent where $t\in (0, |\mathcal Z_j|]$ is the arclength coordinate on $\mathcal Z_j\simeq S^1$. In these coordinates the Riemannian metric $g_0$ can be written \be g_0=dt^2 + dx^2 + dy^2 \  +\  O(r)\label{metriccoords}\ee

\noindent Given such a coordinate system, $(t,r,\theta)$ are used to denote the corresponding cylindrical coordinates, and $(t,z,\overline z)$ the complex ones on the $D_{r_0}$ factor. \end{defn}

\medskip 

\begin{rem}
There are different conventions on the usage of ``Fermi coordinates'' in the literature, with some requiring that the curve is question is a geodesic. In that situation, $n_x$ and $n_y$ can be chosen to locally solve an ODE so that $g=dt^2 +dx^2+ dy^2 + O(r^2)$. Here, we make no such assumption and the difference from the product metric is $O(r)$. Explicitly, the correction to the product metric is 
$$ \big (2x\frak m_x(t)  +2y\frak m_y(t) \big) \  dt^2\  + \ \big( \mu(t) y \big)   \ dt dx \  +\  \big(-\mu(t)x \big) \  dt dy \ + \ O(r^2)$$
\noindent where $ \frak m_\alpha(t) = \br \nabla_{\dot \gamma}\dot \gamma , n_\alpha \kt $  for $\alpha=x,y$ and $
\mu(t) =  \br \nabla_{\dot \gamma} n_x, n_y\kt= -\br \nabla_{\dot \gamma} n_y, n_x\kt.
$

\end{rem}

\medskip 

A choice of Fermi coordinates induces a trivialization of the frame bundle of $Y$ on $N_{r_0}(\mathcal Z_j)$ as follows: it is given by the global orthonormal frame $\{e_t, e_1, e_2\}$ uniquely defined by the property that it restricts to $\{\del_t, \del_x, \del_y\}$ along $\mathcal Z_j$ and is defined by radial parallel transport for $0<r<r_0$. There are two possibilities for the isomorphism class of the restricted spin structure:

\medskip 

\begin{minipage}{6in}

\begin{enumerate}
\label{list2}
\item[{\bf Case 1:}] The spin structure restricts to the product $\frak s_0|_{ N_{r_0}(\mathcal Z_i)}\simeq  N_{r_0}(\mathcal Z_j) \times \text{Spin}(3)$, so that \be S|_{ N_{r_0}(\mathcal Z_i)}\simeq \underline{\C}^2\otimes \ell_0 \label{spiniso}\ee
\item[{\bf Case 2:}] The spin structure restricts to $N_{r_0}(\mathcal Z_j)$ as the double cover of $Fr(Y)|_{N_{r_0}(\mathcal Z_j)} \simeq N_{r_0}(\mathcal Z_j) \times SO(3)$ that is non-trivial in the $\mathcal Z_j$ factor, so that \be S|_{ N_{r_0}(\mathcal Z_j)}\simeq\underline{\C}^2 \otimes \ell_t \otimes  \ell_0 \label{spinisotwisted}\ee
where $\ell_t$ is the pullback of the mobius bundle on $\mathcal Z_j$. 
\end{enumerate}
\end{minipage} 

\bigskip 
\noindent It is worth noting that, in general, there are some rather subtle topological restrictions on which combinations of Case 1 and Case 2 can occur when $\mathcal Z_0$ has multiple components. For instance, if $Y=S^3$ and $\mathcal Z_0$ has a single component, then the unique spin structure on $S^3$ always restricts to Case 2 on a tubular neighborhood of $\mathcal Z_0$; if $\mathcal Z_0$ has multiple components then the number which fall in Case 1 must be even.

First consider Case 1. The trivialization (\refeq{spiniso}) may chosen so that the factors of $\C^2$ are given by the $\pm i$ eigenspaces of $\gamma(e^t)$, in which case Clifford multiplication is given by
$$\gamma(e^t)=\sigma_t=\begin{pmatrix} i & 0 \\ 0 &  -i \end{pmatrix}\hspace{1.5cm}\gamma(e^1)=\sigma_x=\begin{pmatrix} 0 & -1 \\ 1 &  0 \end{pmatrix}\hspace{1.5cm}\gamma(e^2)=\sigma_y=\begin{pmatrix} 0 & i \\ i &  0 \end{pmatrix}.$$ 

\noindent As in the model case  (Example \ref{euclideanexample}), spinors can be written in this trivialization in the form (\refeq{sectionforma}) where $\psi^\pm$ are $\C$-valued function on $N_{r_0}(\mathcal Z_j)$. In Case 2, the same holds after changing the trivialization by $e^{i\ell t/2}$ which alters the Dirac operator by $\tfrac{i}{2}\gamma(dt)$. This leads to the following: 

\begin{lm}\label{Dlocalexpression}
In both Case (1) and Case (2), the $\Z_2$-Dirac operator in local coordinates around a component $\mathcal Z_i\subseteq \mathcal Z_0$ and the above trivialization takes the form
  $$\slashed D= \slashed D_\circ +\frak d $$
where
\begin{itemize}
\item $\slashed D_\circ$ is the Dirac operator in the product metric on $N_{r_0}(\mathcal Z_i)$, given by (\refeq{dirac})
\item $\frak d $ is a first order perturbation arising from the $O(r)$ terms of $g_0$, the perturbation $B_0$ and $\tfrac{i}{2}\gamma(dt)$ in Case 2, so that $$|\frak d \psi|\leq C (r |\nabla \psi| + |\psi|) $$
\noindent holds pointwise.\qed

\end{itemize}

\end{lm}
\subsection{Asymptotic Expansions}
\label{section3.3}
This subsection establishes that $\Z_2$-harmonic spinors have local power series expansions by half integer powers of $r$. These result follow from the general regularity theory for elliptic edge operators in \cite{MazzeoEdgeOperators}.

Fix a choice of Fermi coordinates near each component of $\mathcal Z_0$.  

\begin{defn} \label{polyhomogenousdef}A spinor $\psi\in L^2(Y\setminus \mathcal Z_0; S)$ is said to admit a {\bf polyhomogenous expansion} with index set $\Z^++\tfrac{1}{2}$ if $$\psi\sim   \sum_{ n,p\geq 0  }\sum_{k \in \Z} \begin{pmatrix}  \ \ c_{n,k,p}(t)  e^{ik\theta}  \ \  \\  \ \ d_{n,k,p}(t) e^{ik\theta} \ \ \end{pmatrix} r^{n+1/2}\log(r)^pe^{-i\theta/2}$$

\noindent where $c_{n,k,p}(t), d_{n,k,p}(t)\in C^\infty(S^1;\mathbb \C)$, and where $\sim$ denotes convergence in the following sense: for every $N\in \N$, the partial sums $$\psi_N=   \sum_{ n\leq N  }\sum_{k= -2n }^{2n+1} \sum_{p\leq n-1} \begin{pmatrix}  \ \ c_{n,k,p}(t)  e^{ik\theta}  \ \  \\  \ \ d_{n,k,p}(t) e^{ik\theta} \ \ \end{pmatrix} r^{n+1/2}\log(r)^pe^{-i\theta/2} $$ satisfy the pointwise bounds 
\begin{eqnarray} |\psi-\psi_N| &\leq& C_{N} r^{N+1+\tfrac14}\label{polyhom1}  \hspace{2cm} |\nabla_t^{\alpha}\nabla^{\beta}(\psi-\psi_N) |\leq C_{N,\alpha,\beta} r^{N+1+\tfrac14-|\beta|}\label{polyhom2}\end{eqnarray}

\noindent for constants $C_{N,\alpha,\beta}$ determined by the background data and choice of local coordinates and trivialization. Here, $\beta$ is a multi-index of derivatives in the directions normal to $\mathcal Z_0$. 
\end{defn}

The work of Mazzeo \cite{MazzeoEdgeOperators} implies the following regularity result about $\Z_2$-harmonic spinors (see also Appendix A of \cite{SiqiSLag}). 

\begin{prop} \label{asymptoticexpansion}Suppose that $\Phi_0\in rH^1_e(Y\setminus \mathcal Z_0; S)$ is a $\Z_2$-harmonic spinor. Then $\Phi_0$ admits a polyhomogenous expansion with index set $\mathbb Z^++\tfrac12$. Moreover, $c_{n,k,p}$ and $d_{n,k,p}$ vanish unless $-2n\leq k \leq 2n+1$ and $p\leq n-1$. Thus $\Phi_0$ has a local expression 

\begin{eqnarray} \Phi_0&\sim &\begin{pmatrix} c(t) \sqrt{z} \\ d(t)\sqrt{\overline z} \end{pmatrix} + \sum_{ n\geq 1  }\sum_{k= -2n }^{2n+1} \sum_{p=0}^{n-1} \begin{pmatrix}  \ \ c_{n,k,p}(t)  e^{ik\theta}  \ \  \\  \ \ d_{n,k,p}(t) e^{ik\theta} \ \ \end{pmatrix} r^{n+1/2}\log(r)^pe^{-i\theta/2} \label{Phipolyexp} \end{eqnarray}   
 
\noindent where $c(t), d(t),c_{k,m,n}(t), d_{k,m,n}(t) \in C^\infty(S^1;\C)$. In this form, non-degeneracy in the sense of Definition \ref{regulardef} is equivalent to the requirement that $|c(t)|^2 + |d(t)|^2>0$ is nowhere-vanishing. The same result holds for an $rH^1_e$-solution of the operator $\slashed D-\lambda \text{Id}$.  
\end{prop}

\begin{proof}
The existence of such an expansion is a consequence of the regularity theory in \cite{MazzeoEdgeOperators} (Section 7, Proposition 7.17) and the fact that the indicial roots are $j +\tfrac12$ for $j\in \Z$ in this case.  See also \cite{MazzeoHaydysTakahashi, SiqiSLag}. The constraints on the expansion compared to Definition \ref{polyhomogenousdef} then follow from writing the equation $\slashed D \Phi_0 -\lambda \Phi_0=0$ in Fermi coordinates as $$\begin{pmatrix} 0 & -2\del \\  2\delbar & 0 \end{pmatrix} \Phi_0 =- \frak d \Phi_0 - \begin{pmatrix} -i\del_t & 0 \\ 0 & i\del_t \end{pmatrix}\Phi_0 +\lambda \Phi_0$$
with $\frak d$ as in Lemma \ref{Dlocalexpression}, and formally solving term by term.   \end{proof}

The expansion (\refeq{Phipolyexp}) depends on the choice of Fermi coordinates in the following way. Another choice of Fermi coordinates arises from an alternative choice of normal frame $n_x, n_y$. This change of frame may be taken to be the descent of a change of trivialization of the spin structure, thus may be written in complex coordinates on $N\mathcal Z_0$ as  $$n_1 + in_2 \mapsto e^{-2i\sigma(t)}(n_1+in_2)$$
where $\sigma(t): \mathcal Z_0 \to S^1$ (the minus sign in the exponent is due to the convention that Clifford multiplication is by cotangent vectors). The new complex coordinates $(t,z',z')$ resulting from such a transformation are likewise related to the original coordinates by

$$(t,z', \overline z')=  (t, e^{-2i\sigma(t)}z, e^{2i\sigma(t)}\overline z').$$

This shows the following: 

\begin{cor} \label{transformationlem} For a term of a polyhomogenous expansion  $$\psi(t,z,\overline z)= \begin{pmatrix}  \ a(t)  e^{ik\theta}  \  \\   \ b(t) e^{ik\theta} \ \end{pmatrix} r^{n+1/2}\log(r)^pe^{-i\theta/2} 
$$  

\noindent the coefficients are naturally sections $a(t) \in C^\infty (\mathcal Z_0; N\mathcal Z_0^{-k})$ and $b(t)\in C^\infty(\mathcal Z_0; N\mathcal Z_0^{-k +1})$. In particular, the leading coefficients $c(t), d(t)$ of  (\refeq{Phipolyexp}) are sections  of $N\mathcal Z_0^{-1}, N\mathcal Z_0$ respectively. \qed 
\end{cor}

\begin{rem}\label{weakexpansion}
More generally, $L^2$ kernel elements elements have similar asymptotic expansions, but it is no longer necessarily the case that the coefficients are smooth. In general, the coefficients only make sense as distributions (see Section 7 of \cite{MazzeoEdgeOperators} for a more general discussion). If $\psi \in  \text{ker}(\slashed D) \cap L^2(Y\setminus \mathcal Z_0; S)$, then it admits a {\bf weak asymptotic expansion} of the form 

$$\psi\sim \begin{pmatrix} \frac{c_0(t)}{\sqrt{z}} \\ \frac{d_0(t)}{\sqrt{\overline z}} \end{pmatrix} \ + \ \sum_{ n\geq 1  }\sum_{k= -2n-1 }^{2n+2} \sum_{p=0}^{n-1} \begin{pmatrix}  \ \ c_{n,k,p}(t)  e^{ik\theta}  \ \  \\  \ \ d_{n,k,p}(t) e^{ik\theta} \ \ \end{pmatrix} r^{n-1/2}\log(r)^pe^{-i\theta/2} \label{Phipoly}    $$
\noindent where $c_{n,k,p}, d_{n,k,p} \in L^{-1/2-n}(S^1;\C)$ are understood in a distributional sense and are sections of an appropriate power of $N\mathcal Z_0$ as in Corollary \ref{transformationlem}. {\it There is  no nice sense in which these weak expansions converge}. In particular, if $\psi\in L^2$ has such an expansion, then the difference $|\psi-\psi_N|$ will not necessarily lie in $L^2$. Consequently, there is no meaningful sense in which the later terms are ``smaller'' than the earlier ones. If there were stronger notions of convergence for such weak asymptotic expansions, it is possible that Theorem \ref{mainb} could be proved without the use of Nash-Moser theory.

\end{rem}
\bigskip 

\section{The Obstruction Space}
\label{section4}

This section studies the infinite-dimensional cokernel of the operator \be \slashed D: rH^1_e(Y\setminus \mathcal Z_0; S)\lre L^2(Y\setminus \mathcal Z_0;S),\label{hascokernel}\ee

\noindent which coincides with $\ker(\slashed D|_{L^2})$ by Lemma \ref{L^2splitting}. The main results of this section, Propositions \ref{cokproperties} and \ref{cokproperties2} generalizes the three key properties noted below Example \ref{euclideanexample} to the case of a compact manifold.

\begin{defn}Define the {\bf Obstruction Space} associated to the data $(\mathcal Z_0, g_0, B_0)$ by 
$$ \text{\bf Ob}(\mathcal Z_0):= \{\psi \in L^2 \ | \ \psi \in \ker(\slashed D|_{L^2}) \}.$$

\noindent In addition, define $\text{\bf Ob}(\mathcal Z_0)^\perp = \{\psi \in \text{\bf Ob}(\mathcal Z_0) \ | \ \br \psi, \Phi\kt_{L^2}=0 \ \forall \Phi \in \ker(\slashed D_{rH^1_e})\}$. The $L^2$-orthogonal projections to $\text{\bf Ob}(\mathcal Z_0), \text{\bf Ob}(\mathcal Z_0)^\perp$ are denoted $\Pi_0, \Pi_0^\perp$ respectively. 
\label{Obdefinition}
\end{defn}

 \noindent Although this definition appears to be a redundant renaming of $\ker(\slashed D|_{L^2})$, it is made in preparation for Section \ref{section8}. There, the obstruction space will be extended to a vector bundle over the data $(\mathcal Z, g, B)$ whereas $\ker(\slashed D|_{L^2})$ may not be locally trivial (analogously to the finite-dimensional case where the dimension may jump).

The upcoming Proposition \ref{cokproperties} provides an isomorphism between $\text{\bf {Ob}}(\mathcal Z_0)^\perp$ and the space of sections of a vector bundle on $\mathcal Z_0$ (cf. the discussion following Theorem \ref{maina}). The fibers of this vector bundle are given as follows. Let $\underline \C^2\to \mathcal Z_0$ denote the trivial bundle with fiber $\C^2$. Sections of each summand may be decomposed in Fourier series (using the orientation given by the fixed choice of Fermi coordinates in Section \ref{section3.3}; denote by $H$ the {\bf modified Hilbert transform}, i.e. the pseudo-differential operator on $\mathcal Z_0$ whose symbol is given by $\text{sgn}(\ell)$ where $\ell$ is the Fourier index.

\begin{defn} \label{calderondef}Define the {\bf Calder\'on Subbundle} $\mathcal C_0\subseteq \underline \C^2\to \mathcal Z_0$ as the trivial complex line bundle given by the first summand, and the {\bf Calder\'on Subspace} as the subspace $$\Lambda_\circ = \left\{\begin{pmatrix} \xi(t) \\ H\xi(t)\end{pmatrix} \ \Big | \  \xi \in L^2(\mathcal Z_0; \mathcal C_0)\right\}\subseteq L^2(\mathcal Z_0;  \C^2).$$

\noindent Clearly, there is a canonical isomorphism $\Lambda_\circ \simeq L^2(\mathcal Z_0; \mathcal C_0)$. 
\end{defn}

 The first main result of Section \ref{section4} is the following. 
\begin{prop}\label{cokproperties}
There is an isomorphism \be \text{ob}: L^2(\mathcal Z_0; \mathcal C_0 ) \lre \text{\bf Ob}(\mathcal Z_0)^\perp \label{obisomorphism}.\ee 
\end{prop}
\noindent It follows that $\text{ob}\oplus \iota:  L^2(\mathcal Z_0; \mathcal C_0 )\oplus \ker(\slashed D|_{rH^1_e})\to \text{\bf {Ob}}(\mathcal Z_0)$ is an isomorphism, where $\iota$ is the inclusion of the second factor. 
\medskip

Proposition \ref{cokproperties} may be viewed as a Poisson extension result for a ``codimension 2'' boundary-value problem. Although the proof does not require making these notions precise, the heuristic guides the remainder of Section \ref{section4} and is worth describing in some detail (see \cite{MazzeoEdgeOperatorsII} for a formal theory of boundary-value problems for edge operators). There is a ``boundary trace'' operator 
$$\text{tr}:  \text{\bf {Ob}}(\mathcal Z_0)^\perp \lre H^{-1/2}(\mathcal Z_0 ; \C^2)$$

\noindent given by taking the leading coefficients $(c_0(t),d_0(t))$ in the polyhomogenous expansion in Remark \ref{weakexpansion}. The fact that these leading coefficients are sections of the trivial bundle is a consequence of Corollary \ref{transformationlem}. More generally, the leading coefficients are valued in a ``trace bundle'' defined precisely in \cite[Thm. 2.5]{MazzeoEdgeOperatorsII}, but by a coincidence of the transformation rule for spinors and the exponents in the expansions of Remark \ref{weakexpansion}, this bundle is trivial in our case. 

Example \ref{euclideanexample} shows that on $Y_\circ=S^1\times \R^2$,  $\text{tr}: \ker(\slashed D|_{L^2})\to \Lambda_\circ $ is an isomorphism (to the closure in $H^{-1/2}$), with inverse given by the Poisson extension operator \begin{eqnarray}\frak P_\circ : H^{-1/2}(\mathcal Z_0; \mathcal C_0) &\lre& \ker(\slashed D|_{L^2}) \\   \xi(t)&\mapsto&  \sum_{\ell \in \Z}  |\ell|^{-1/2}\br  \xi(t), e^{i\ell t}\kt \Psi_\ell^\circ \label{Poissondef}\end{eqnarray}
\noindent where $\Psi_\ell^\circ$ are as in Equation (\refeq{euclideancokernel}), and where we implicitly use the isomorphism in Definition \ref{calderondef}. Note the regularization factor $|\ell|^{-1/2}$ is needed because $\Psi_\ell^\circ$ is defined to be normalized in $L^2$. More generally, on a compact manifold, there is a Calder\'on subspace $\Lambda =\text{tr}(\text{\bf Ob}^\perp(\mathcal Z_0))\subseteq H^{-1/2}(\mathcal Z_0; \C^2)$ (Lagrangian with respect to the natural symplectic form on $H^{-1/2}(\mathcal Z_0;\C^2)$) and a Poisson operator $\frak P: \Lambda \to {\text{\bf Ob}(\mathcal Z_0)}^\perp$ giving its inverse. 

The Calder\'on subspace $\Lambda_\circ$ for the model operator has the pleasing property that it is canonically identified with the space of sections of the vector bundle $\mathcal C_0$. In general, the Calder\'on subspace $\Lambda$ has no reason to have such an identification, and in order to view the operator of Theorem \ref{maina} as a pseudodifferential operator, one must chose such an identification. The diagram below, {\it which does not commute}, depicts two natural (inequivalent) choices for identifying the model Calder\'on subspace $\Lambda_\circ$ with $\text{\bf Ob}(\mathcal Z_0)^\perp$.

\begin{center}
\tikzset{node distance=3.0cm, auto}
\begin{tikzpicture}
\node(A){$\Lambda_\circ$};
\node(X)[right of=A,xshift=-1.7cm,yshift=1.0cm]{$\cancel\circlearrowleft$};
\node(C)[above of=A, yshift=-1cm]{$\Lambda$};
\node(C')[right of=A]{$L^2(Y\setminus \mathcal Z_0;S)$};
\node(D')[above of=C',yshift=-1cm]{$\text{\bf Ob}(\mathcal Z_0)^\perp$};
\draw[->] (A) to node {$\pi_\Lambda$} (C);
\draw[->][swap] (A) to node {$\chi \frak P_\circ$} (C');
\draw[->] (C') to node {$ \Pi_0^\perp$} (D');
\draw[<-][swap] (D') to node {$\frak P$} (C);
\end{tikzpicture}\end{center}

\noindent In the upper pathway, $\pi_\Lambda$ is the $L^2$-orthogonal projection to $\Lambda$ in $H^{-1/2}(\mathcal Z_0; \C^2)$, and $\frak P$ is the Poisson extension operator on $Y$. In the lower pathway, $\frak P_\circ$ is the model Poisson extension operator on $Y_\circ$ pasted onto the compact $Y$ with a cut-off function $\chi$ supported near $\mathcal Z_0$, and $\Pi_0^\perp$ is the $L^2$-orthogonal projection to $\text{\bf Ob}(\mathcal Z_0)^\perp$ from Definition \ref{Obdefinition}. The above is to say that the two natural choices are to project to $\Lambda$ in the boundary space then extend, or first extend using the model Poisson operator and then project on the 3-manifold $Y\setminus \mathcal Z_0$. While the first option (top path) is arguably more natural from the perspective of \cite{MazzeoEdgeOperatorsII}, it leads to difficulties showing certain error terms arising from the metric are bounded (let alone compact). To prove Proposition \ref{cokproperties}, the second (bottom path) was found to be more robust, and the definition of the map $\text{ob}$ is a minor modification of the composition $\text{ob}=\Pi_0^\perp\circ \chi \frak P_\circ$. 

In this approach, we also use the $L^2$ (rather than $H^{-1/2}$) normalization to eliminate the normalization factor in (\refeq{Poissondef}). (See also Remark \ref{obconvention}).

The map $\text{ob}$ provides, in a very loose sense, some sort of ``coordinates'' on the obstruction by identifying with the easily-described space $\Lambda_\circ$. To calculate the image of a spinor $\text{ob}^{-1}\circ \Pi_0^\perp \Psi$ in $\Lambda_\circ$, one may choose a basis. Associated to our choice of the bottom pathway in the above diagram, there is a natural basis for $\text{\bf {Ob}}(\mathcal Z_0)^\perp$, given by the image of Fourier modes. Let  $e^{i\ell t_j} \in \Gamma(\mathcal Z_0;\mathcal C_0)$ denote the $\ell^{th}$ Fourier mode on the $j^{th}$ component of $\mathcal Z_0$, where $\ell \in {2\pi \Z}/{|\mathcal Z_{j}|}$. Given Proposition \ref{cokproperties}, set 
$$\Psi_{j\ell}=\text{ob}(e^{i\ell t_j}).$$

\noindent This basis satisfies the following. In the upcoming proposition statement, we tacitly assume that $\mathcal Z_0$ consists of a single component and omit the subscript $j$

\begin{prop}\label{cokproperties2} 
(A) When $\slashed D$ is complex linear, the $\text{\bf Ob}(\mathcal Z_0)$-component of a spinor $\psi\in L^2$ under (\refeq{obisomorphism}) is given by

\be
 \text{ob}^{-1}(\Pi_0^\perp\psi)= \sum_{\ell} \br\psi,  \Psi_{\ell}  \kt_{{\C}} \phi_\ell.\label{innerprods} \hspace{3cm} \iota^{-1}(\psi)=\sum_k \br  \psi, \Phi_k\kt \Phi_k.\ee

\noindent where $\br -,-\kt_\C$ is the hermitian inner product, and $\Phi_k$ a (real) basis of $\ker(\slashed D|_{rH^1_e})$. Moreover,  $$\Psi_{\ell} =\chi\Psi_\ell^\circ + \zeta_{j\ell}+ \xi_{\ell}$$ where 

\begin{itemize}
\item $\Psi_\ell^{\circ}$ 
are the $L^2$-orthonormalized Euclidean obstruction elements from Example \ref{euclideanexample} (in the trivialization \refeq{Dlocalexpression}) and $\chi$ is a cutoff function supported on a tubular neighborhood of $\mathcal Z_0$.
 \item $\zeta_{\ell}$  is a perturbation with $L^2$-norm $O(|\ell|^{-1})$ which decays exponentially away from $\mathcal Z_0$ in the following sense:  
 \be \|\zeta_{\ell}\|_{L^2(A_{n\ell}^{})} \leq \frac{C}{|\ell|}\text{Exp}\left(-\frac{n}{c_1}\right) 
\label{expdecayzeta}. \ee
where $A_{n\ell}$ denotes the collection of annuli \be A_{n\ell}= \left\{ \tfrac{n}{|\ell|}R_0 \leq r^{} \leq \tfrac{n+1}{|\ell|}R_0 \right\}\label{annuli}\ee for some constant $R_0$, and $r$ denotes the geodesic distance to $\mathcal Z_{0}$. Additionally,  in Fermi coordinates on $N_{r_0}(\mathcal Z_0)$ and in the trivialization of Lemma \ref{Dlocalexpression},  $\zeta_{\ell}$ is a linear combination of only Fourier modes $e^{ipt}$ in the range $\ell - {|\ell|}/{2} \leq p \leq \ell + {|\ell|}/{2}$. 
\item $\xi_{\ell}$ is a perturbation of $L^2$-norm $O(|\ell|^{-2})$ i.e. satisfying
$$\|\xi_{\ell}\|_{L^2} \leq \frac{C}{|\ell|^2}$$
for a universal constant $C$.  

\end{itemize}

\noindent (B) In the case that $\slashed D$ is only $\R$-linear, $\Psi_{l}^\text{Re}=\text{ob}(\text{Re}(\ph_{\ell}))$ and likewise for the imaginary part form a real basis and $$\Psi_{\ell}^{\text{Re}}=\chi\Psi_\ell^\circ + \zeta^{\text{Re}}_{\ell}+ \xi^{\text{Re}}_{\ell}\hspace{2cm} \Psi_{\ell}^{\text{Im}}=i(\chi\Psi_\ell^\circ) + \zeta^{\text{Im}}_{\ell}+ \xi^{\text{Im}}_{\ell}$$ satisfying identical bounds where the  inner product in (\refeq{innerprods}) is replaced by $$\br \psi, \Psi_\ell^{}\kt_\C= \br \psi, \Psi_{\ell}^{\text{Re}}\kt \ + \ i \br \psi, \Psi_{\ell}^{\text{Im}}\kt.$$

\noindent Moreover, in the case that $\mathcal Z_0$ has multiple components, either (A) or (B) holds mutatis mutandis with an additional index $j$ ranging over the components of $\mathcal Z_0$. 
\end{prop}

\medskip

\noindent The reader is cautioned that the basis $\Psi_{\ell}$ is not necessarily orthogonal, and cannot be orthogonalized without disrupting the decay properties in the second bullet point. 

The remainder of Section \ref{section4} proves Propositions \ref{cokproperties} and \ref{cokproperties2} concurrently. Section \ref{section4.1} studies the Poisson extension on a normal neighborhood of $\mathcal Z_0$ and proves a preliminary version of Proposition \ref{cokproperties2}. Sections \ref{section4.2} and \ref{subsection4.3} construct the map $\text{ob}$ and show, respectively, that it is Fredholm and has index zero. Section \ref{section4.4} makes a compact correction so that $\text{ob}$ is an isomorphism, thereby completing the proofs. Section \ref{section4.5} discusses the higher-regularity analogues of both propositions.

\begin{notation} Throughout the remainder of the section we tacitly assume (i) $\slashed D$ is complex linear, and (ii) $\mathcal Z_0$ consists of a single component. The proof in the real-linear and multi-component situation is a trivial extension. In the remainder of the section we make the following conventions to avoid cluttering notation: 

\begin{itemize}
\item[(i)] The subscript $j$ is omitted, and $\br - , - \kt$ denotes the Hermitian inner product.    
\item[(ii)] The data $(\mathcal Z_0, g_0, B_0, S_0, \ell_0, \nabla_0, \gamma_0)$ is fixed and the subscript $0$ is omitted. 
\item[(iii)] The subscript $\circ$ denotes the structures in (ii) in the model case of Example \ref{euclideanexample}. 
\item[(iv)] The subscript $N$ denotes the structures in (ii) on a tubular nbhd. of $\mathcal Z_0$ (cf. Section \ref{section4.1}).
\item[(v)] The choice of Fermi coordinates from Section \ref{section3.3} is fixed throughout. 
\item[(vi)] The `pullback' normalization on $Y$ is used so that the domain of $\text{ob}$ is $L^2(\mathcal Z_0; \mathcal C_0)$ ( not $H^{-1/2}$). 
\end{itemize}
\end{notation}

\subsection{The Model Obstruction}
\label{section4.1}

This section proves a preliminary version of Propositions \ref{cokproperties}-\ref{cokproperties2} on the normal bundle $N\mathcal Z_0$. Choose $r_N>0$ small, and let $\chi_N$ be a cut-off function vanishing for $r> r_N$ and equal to 1 for $r<r_N/2$. Set 
\begin{eqnarray}
(N,g_N)&:=&(N\mathcal Z_0 \ , \  \chi_N g_0 + (1-\chi_N)g_\circ)\\
B_N&:=&\chi_N B_0 \\
\Pi_N&:=& (1- \slashed D_N P_N \slashed D_N) \label{PiN}\\
\frak P_N&:=&  \Pi_N \frak P_\circ \label{PoissonN} \\ 
\Psi_\ell^N&:=& \frak P_N(e^{i\ell t}) \label{PsiN}
\end{eqnarray}

\noindent where $g_\circ=dt^2 + dx^2 + dy^2$ is the product metric, $\slashed D_N, P_N$ are formed using  $(g_N,B_N)$, $\Pi_N$ is the projection to the obstruction as in Corollary \ref{Pdiagram}, and and $\frak P_\circ(e^{i\ell t})= \Psi_\ell^\circ$ is the $L^2$-normalized version of  (\refeq{Poissondef}) so that $\frak P_\circ: L^2(\mathcal Z_0;\mathcal C_0)\to \ker(\slashed D_\circ |_{L^2})$ is an isomorphism. 

The Dirac operator may be written $$\slashed D_N:= \slashed D_\circ +   \frak d$$

\noindent with $\slashed D_\circ, \frak d$ as in Lemma \ref{Dlocalexpression}. If the spin structure falls in Case 2 as in (\ref{list2}), then we truncate $\gamma(idt/2)$ to $\chi_N \gamma(i dt/2)$ for convenience. The preliminary version of Propositions \ref{cokproperties}--\ref{cokproperties2} is the following.

\begin{lm}\label{PsiNbasis}
For $r_N$ sufficiently small,

$$ \frak P_N: L^2(\mathcal Z_0; \mathcal C_0)_0\to \ker(\slashed D_N|_{L^2})$$
\noindent   is an isomorphism, where $L^2(\mathcal Z_0; \mathcal C_0)_0$ is the $L^2$-orthogonal complement of the constant Fourier mode.  

\end{lm}

\begin{rem}Lemma \ref{PsiNbasis} implicitly includes the assertion that $\ker(\slashed D_N|_{rH^1})=0$. The $\ell=0$ mode is omitted simply because the $r^{-1/2}$ asymptotics fail to be $L^2$ on the non-compact space $N$. These modes are treated separately in the index calculation in Section \ref{subsection4.3}. 
\end{rem}

\begin{proof} By Lemma \ref{Dlocalexpression}, $\frak d$ satisfies \be |\frak d \ph| \leq C (r |\nabla \ph| + |\ph| )\label{frakdbound}\ee
pointwise on $\text{supp}(\chi_N)$, and $\frak d=0$ elsewhere, hence $\|\frak d\ph\|_{L^2}\leq C r_N \|\ph\|_{rH^1_e}$.  Thus for $r_N$ sufficiently small, 
$$\|\ph\|_{rH^1_e}\leq C \|\slashed D_\circ \ph\|_{L^2} \leq C \|(\slashed D_\circ +\frak d)\ph\|_{L^2} + C\|\frak d\|_{L^2}\hspace{1cm}  \Rightarrow  \hspace{1cm}\|\ph\|_{rH^1} \leq C' \|\slashed D_N \ph\|_{L^2}.$$
Consequently, $\ker(\slashed D_N|{rH^1_e})=0$ and $$\slashed D_N \slashed D_N: rH^1_e\to rH^{-1}_e$$
is an isomorphism by Lemma \ref{secondorderoperator} with inverse $P_N$. 

Since $\frak P_\circ: L^2(\mathcal Z_0;\mathcal C_0)_0\to \ker(\slashed D_\circ |_{L^2})$ is an isomorphism by Example \ref{euclideanexample}, in order to show that $\frak P_N$ is an isomorphism it suffices to show that $\Pi_N: \ker(\slashed D_\circ|_{L^2})\to \ker(\slashed D_N|_{L^2})$ is an isomorphism, where $\Pi_N$ and $\frak P_N$ are as in (\refeq{PiN}--\refeq{PoissonN}). For injectivity, observe that for $\Psi\in \ker(\slashed D_\circ|_{L^2})$, 

\bea
\|\Psi\|_{L^2} &=& \|\text{Id}-\slashed D_\circ P_\circ \slashed D_\circ \Psi\|_{L^2}\\
&\leq &  \|\text{Id}-\slashed D_N P_N \slashed D_N \Psi\|_{L^2}  +  \|\slashed D_N P_N\slashed D_N-\slashed D_\circ P_\circ \slashed D_\circ \Psi\|_{L^2}\\
&\leq & \|\Pi_N \Psi\|_{L^2} + \|\slashed D_N P_N \frak d\Psi\|_{L^2}\\
&\leq & \|\Pi_N \Psi\|_{L^2} + Cr_N \|\Psi\|_{L^2}, 
\eea
\noindent where the last inequality follows from integrating $\|\frak d \Psi\|_{r^{-1}H^{-1}}=\sup \br \frak d\Psi, \ph \kt$ by parts and applying (\refeq{frakdbound}) for $\ph \in rH^1_e$, then using the fact that $\slashed D_N, P_N$ are bounded. For $r_N$ sufficiently small, it follows that $\Pi_N$ is injective with closed range on $\ker(\slashed D_\circ|_{L^2})$.

For surjectivity, we argue by contraction: suppose that there were $\eta \in \ker(\slashed D_N|_{L^2})$ such that $\br \Pi_N \Psi, \eta\kt_{L^2}=0$ held for all $\Psi \in \ker(\slashed D_\circ|_{L^2})$. Assume that $\|\eta\|_{L^2}=1$ is normalized. Since $\eta \in \ker(\slashed D_N|_{L^2})$, it is orthogonal to $\text{ran}(\slashed D_N|_{r H^1})$, thus writing $\slashed D_N=\slashed D_\circ +\frak d$ we have 
$$0=\br \slashed D_N \ph , \eta \kt  \ \ \Rightarrow \ \ |\br \slashed D_\circ \ph, \eta \kt |=|\br \frak d \ph, \eta\kt| \leq Cr_N \|\ph\|_{rH^1_e},$$

\noindent i.e. the component of $\eta$ in the range of $\slashed D_\circ|_{rH^1}$ is small. Consequently, there is a $\Psi_\circ \in \ker(\slashed D_\circ |_{L^2})$ such that $\eta= \Psi_\circ +w $ with $\|w\|_{L^2}\leq Cr_N$, (hence $1-Cr_N\leq \|\Psi_\circ\|_{L^2}$). But this would imply that  
$$0=\br \Pi_N \Psi_\circ, \eta \kt=\br \Psi_\circ + \slashed D_N P_N\frak d\Psi_\circ \ , \  \Psi_\circ +w  \kt \geq 1 - C'r_N,$$ a contradiction once $r_N$ is sufficiently small. 
\end{proof}

\bigskip 
The next lemma is analogue of Proposition \ref{cokproperties2} on $(N,g_N)$: 

\begin{lm} \label{expdecayN}For $r_N$ sufficiently small, the set  $\{\Psi_\ell^N\}$ for $\ell \in \Z\setminus\{0\}$ form a basis of $\ker(\slashed D_N|_{L^2})$,  and may be written $$\Psi_\ell^N=\Psi_\ell^\circ + \zeta_\ell^N + \xi_\ell^N$$
\noindent  where $\zeta_\ell^N, \xi_\ell^N$ satisfy the conclusions of Proposition \ref{cokproperties2}.
\end{lm}

The proof of Lemma \ref{expdecayN} is an iteration argument bootstrapping the decay of certain error terms. Since $\frak d=O(r)$, and the initial error term has size $\|O(r)\Psi_\ell^\circ\|_{L^2}=O(|\ell|^{-1})$, which follows from direct integration of $r^ke^{-|\ell|r}$ for the correct power of $k$. The iteration process corrects the error term by first solving for the error using $\slashed D_\circ$. Provided the corrected solution retains the exponential decay properties of $\Psi_\ell^\circ$, the new error terms picks up a factor of $r$, thus a factor of $|\ell|^{-1}$ after integration. The process may then be iterated to obtain arbitrarily large powers. Lemma \ref{expdecayN} (cf Proposition \ref{cokproperties2}) employ the first stage of the iteration to obtain an exponentially decaying correction $\zeta_\ell$ and a $O(|\ell|^{-2})$ correction $\xi_\ell$. Corollary \ref{higherordercokernel} continues the iteration to obtain higher regularity estimates. 

The iteration process relies on the following lemma, which is applied repeatedly to conclude that corrected solutions in the proof of Lemma \ref{expdecayN} retain the desired exponential decay properties. Morally, it should be viewed as a statement about the exponential decay of the Green's function of $\slashed D_\circ \slashed D_\circ$ in certain Fourier modes and is proved in Appendix \ref{appendixI} using a discrete maximum principle argument due to Taubes  \cite[App. A.2.1.]{SivekTaubesECH} (which serves as a proxy for explicit computations with the Green's function). For the statement of the lemma, $A_{n\ell}$ denotes the sequence of annuli (\refeq{annuli}) from Part (B) of Proposition \ref{cokproperties}, and we set $B_{n\ell}=A_{(n-1)\ell}\cup A_{n\ell} \cup A_{(n+1)\ell}$.

 \begin{lm}\label{discretedecay} Let $m$ be a non-negative integer, and assume that $|\ell|\geq 2m$. Suppose that $u_\ell \in  rH^1_e(N)$ is the unique solution of \be \slashed D_\circ\slashed D_\circ u_\ell = f_\ell\label{secondordereq}\ee 
 
 \noindent where $f_\ell\in r^{-1} H^{-1}_e$ satisfies the following two properties:  
 \begin{enumerate}
 \item $f_\ell$ has only Fourier modes in $e^{ipt}$ for $p$ in the range \be \ell -L_0 \leq p \leq \ell + L_0\label{409}\ee
 \noindent where $|L_0|\leq |\ell|/2$. 
 \item For $m$ as above, there are constants $C_m,c_m$ independent of $\ell$ such that $f_\ell$ satisfies the bounds \be\|f_\ell\|^2_{r^{-1}H^{-1}_e(B_{n\ell})}\leq  \frac{C_m}{|\ell|^{2+2m}}\text{Exp}\left(-\frac{2n}{c_m}\right) \label{expBnl}\ee
 \noindent on the sequence of annuli $B_{n\ell}$. 
 \end{enumerate} 
 Then there are constants $C'_m, c_m'$ independent of $\ell$ such that $u_\ell$ similarly satisfies  
\be \|u_\ell\|^2_{r H^1_e(A_{n\ell})} \leq \frac{C'_m}{|\ell|^{2+2m}}\text{Exp}\left(-\frac{2n}{c_m'}\right) \label{expdecayzetalemma}. \ee
\noindent Moreover, $u_\ell$ has only Fourier modes in the same range as $f_\ell$.  \qed
 \end{lm}

We now prove Lemma \ref{expdecayN} using Lemma \ref{discretedecay}: 

\begin{proof} Let $\Psi_N$ be as defined by (\refeq{PsiN}). With $\slashed D_N=\slashed D_0 + \frak d$ as in Lemma \ref{Dlocalexpression}, $\frak d$ can be explicitly written in the form $$\frak d=\sum_{ij=1}^3 a_{i,j}(t,x,y)\sigma_i \del_j + \sum_{k=0}^3 \Gamma_k(t,x,y)\sigma_k$$ 
where $|a_{ij}|\leq Cr$ and $|\Gamma|\leq C$ and $\sigma_i=\gamma(e^i)$ with $\sigma_0=I$ in the second sum. Decomposing  $a_{ij}(t,x,y), \Gamma_k(t,x,y)$ into the Fourier modes in the $t$-direction on $N\simeq S^1\times \R^2$, this operator can be written as 
$$\frak d= \frak d^\text{low}+ \frak d^\text{high}$$
where $\frak d^\text{low}$ consists of the Fourier modes of $a_{ij}, \Gamma_k$ with Fourier index $|p|\leq |\ell|/4$. 

Since $\Psi^N_\ell=\Pi_N\Psi_\ell^\circ=(\text{Id}-\slashed D_N P_N\slashed D_N)\Psi_\ell^\circ$ by definition, 
\bea
\Psi^N_\ell - \Psi_\ell^\circ&=& -\slashed D_N P_N( \frak d \Psi_\ell^\circ)=-\slashed D_N P_N (f_\ell^\text{low}+f_\ell^\text{high})
\eea 

\noindent where $f_\ell^\text{low}:=\frak d^\text{low} \Psi_\ell^N$ and $f_\ell^\text{high}:=\frak d^\text{high} \Psi_\ell^N$. Set 
\bea
\zeta_\ell^N&:= &\slashed D_Nu_\ell \hspace{1cm}\text{where }\hspace{1cm}u_\ell:= -P_\circ(f_\ell^\text{low}) \\
\xi_\ell^N&:=& \slashed D_N v_\ell \hspace{1cm}\text{where }\hspace{1cm} v_\ell:=-P_N(f_\ell^\text{high}- ( \slashed D_N\slashed D_N-\slashed D_\circ\slashed D_\circ)u_\ell)
\eea

\noindent so that $(\zeta_\ell^N+\xi_\ell^N)=\Psi_\ell^N-\Psi_\ell^\circ$ as intended, since $u_\ell+v_\ell$ satisfies $\slashed D_N \slashed D_N(u_\ell+v_\ell)=f_\ell^\text{low}+f_\ell^\text{high}$.    

The desired decay properties for $\zeta_\ell^N$ follow from applying Lemma \ref{discretedecay} in the case that $m=0$. To elaborate, the first hypothesis of that lemma is satisfied by construction, because $f_\ell^\text{low}=\frak d^\text{low}\Psi_\ell^N$ was defined to be the Fourier modes in the necessary range. To verify the second hypothesis, observe that

$$\|f_\ell^\text{low}\|_{r^{-1}H^{-1}_e(B_{n,\ell})}\leq \sup_{\|u\|=1} \br u, f_{\ell}^\text{low}\kt_{L^2}\leq \sup_{\|u\|=1}\|u\|_{rH^1_e} \|r f_\ell^\text{low}\|_{L^2(B_{n\ell})}\leq \|r f_\ell^\text{low}\|_{L^2(B_{n\ell})} $$

\noindent hence using the bounds $|a_{ij}|\leq Cr$ and $|\Gamma_k|\leq C$ for $\frak d^\text{low}$, 

 \begin{eqnarray}   \int_{B_{n\ell}}r^2|f_\ell^\text{low}|^2 \ dV &\leq &C \frac{n^2}{|\ell|^2}R_0^2\int_{B_{n\ell}} |r \nabla_j \Psi_\ell^\circ|^2 + |r \nabla_t \Psi_\ell^\circ|^2 + |\Psi_\ell^\circ|^2 \ rdrd\theta dt  \label{expdecayetaell1} \\ &\leq &C \frac{n^2}{|\ell|^2}R_0^2\int_{B_{n\ell}} (1+ r^2|\ell|^2+1) \frac{e^{-2|\ell| r}}{r}|\ell|  \ rdrd\theta dt \\ 
&  \leq & C \frac{n^5}{|\ell|^2}R_0^5  e^{-nR_0}  \leq   \frac{C'}{|\ell|^2}  e^{-2n/c_1}.  \label{expdecayetaell}  \end{eqnarray}

\noindent where $R_0$ is as in the second bullet of Proposition \ref{cokproperties2}. Thus we conclude from Lemma \ref{discretedecay} and the fact that $\slashed D_N$ is bounded that 

\bea
\|u_\ell\|_{rH^1_e(A_{n\ell})} \leq \frac{C_0}{|\ell|}\text{Exp}\left(-\frac{n}{c_0}\right)  \hspace{1.5cm}\Rightarrow \hspace{1.5cm}\|\zeta^N_\ell\|_{L^2(A_{n\ell})} \leq \frac{C_0}{|\ell|}\text{Exp}\left(-\frac{n}{c_0}\right)\label{expdecayetaell2}  
\eea
\noindent as desired. 

It remains to show the asserted bound on $\xi_\ell^N$ holds. Since $\slashed D_N: rH^1_e\to L^2$ and $P_N: rH^{-1}_e\to rH^1_e$ are bounded, it suffices to show that \be \|f_\ell^\text{high}- ( \slashed D_N\slashed D_N-\slashed D_\circ\slashed D_\circ)u_\ell\|_{r^{-1}H^{-1}_e}\leq \frac{C}{|\ell|^2}. \label{xil2decay} \ee

\noindent Addressing the two terms on the left separately, one has $\slashed D_N\slashed D_N-\slashed D_0\slashed D_0=\frak d\slashed D_0 + \slashed D_0\frak d + \frak d^2$ which shows 

\be \|( \slashed D_N\slashed D_N-\slashed D_0\slashed D_0)u_\ell\|^2_{r^{-1}H^{-1}_e}  \leq C\sum_{n}\sup_{A_{n\ell}} (r^2 \|u_\ell\|^2_{rH^1_e(A_{n\ell})})\leq \frac{C}{|\ell|^4}.\label{xibound1}\ee

\noindent For $f_\ell^\text{high}$, note that the coefficients $a_{ij}, \Gamma_k$ are smooth and $\frak d^\text{high}$ and have only Fourier modes $p$ with $|p|\geq |\ell|/4$. Applying the Sobolev embedding for each fixed $(x,y)$ therefore shows that 
 \be
 \|a^\text{high}\|_{C^0(Y)}\leq \sup_{x,y} \|a^\text{high}(t)\|_{C^0(S^1)} \leq C\sup_{x,y} \|a^\text{high}(t)\|_{H^{1}(S^1)} \leq \frac{C}{|\ell|^2}\sup_{x,y} \|a^\text{high}(t)\|_{H^{3}(S^1)} \leq \frac{C}{|\ell|^2}\label{highmodes}
 \ee
and likewise for $\Gamma^\text{high}$. Combining the bounds (\refeq{xibound1}) and (\refeq{highmodes}) shows (\refeq{xil2decay}), completing the proof.  
\end{proof}

The above procedure may be iterated to bootstrap the bounds on $\xi^N_\ell$ without disrupting the bounds on $\zeta_\ell^N$. In the following statement, $\nabla_z$ is tacitly used to denote a covariant derivative in a direction normal to $\mathcal Z_0$, and $\nabla_t$ a tangential one.  

\begin{cor}\label{higherordercokernel}
For every $m$ there is an alternative decomposition 

$$\zeta_\ell^N + \xi_\ell^N= \zeta_\ell^{(m)}+ \xi_\ell^{(m)}$$
where 
\begin{itemize}
\item There are constants $C_m$ and $C_m'$ such that \be\|\zeta_\ell^{(m)}\|_{L^2(A_{n\ell})}\leq \frac{C_m}{|\ell|}\text{Exp}\left(-\frac{n}{c_m}\right) 
\hspace{1cm}\|(r\nabla_z)^\alpha (\nabla_t)^\beta\zeta_\ell^{(m)}\|_{L^2(A_{n\ell})}\leq \frac{C'_m|\ell|^\beta}{|\ell|}\text{Exp}\left(-\frac{n}{c_m'}\right)  \label{expdecayzetak}. \ee

\noindent for $A_{n\ell}$ is as in Proposition \ref{cokproperties} and multi-indices $\alpha,\beta$.  \item The latter perturbation satisfies

\be \|\xi_\ell^{(m)}\|_{L^2}\leq \frac{C_m}{|\ell|^{2+m}}. \hspace{2.5cm} \| (r\nabla_z)^\alpha (\nabla_t)^\beta\xi_\ell^{(m)}\|_{L^2}\leq \frac{C'_m|\ell|^\beta}{|\ell|^{2+m}} \label{squaredboundk}\ee

\end{itemize} Moreover, $\zeta_\ell$ contains only Fourier modes $e^{ipt}$ with $\ell - \tfrac{|\ell|}{2} \leq p \leq \ell + \tfrac{|\ell|}{2}.$ The constants $C_m, c_m$ are independent of $\ell$, and depend on up to the $H^{m+3}$-norm of the metric, and $C'_m, c_m' $ on up to the $H^{m+|\alpha|+|\beta|+3}$-norm. 
\end{cor}

\begin{proof}
For $\alpha=\beta=0$, this follows from applying Lemma \ref{discretedecay} inductively. Instead of solving for $\xi_{\ell}$ with $f_\ell^\text{high}-(\slashed D_N\slashed D_N-\slashed D_0 \slashed D_0)u_\ell$ on the right hand side as in the proof of Lemma \ref{expdecayN}, instead set $(f_\ell^{\text{low}})^{1}= -(\slashed D_N\slashed D_N-\slashed D_0 \slashed D_0)u_\ell$ and apply Lemma \ref{discretedecay} again to the low Fourier modes to obtain a second correction $\zeta_\ell'$ and set $\zeta_\ell^{(1)}=\zeta_\ell+ \zeta_\ell'$. Proceeding in this fashion, each iteration yields an additional power of
$r$ from the difference $\slashed D_N-\slashed D_\circ$. Integrating against the exponential, this becomes an additional power of $|\ell|^{-1}$ in the new remainder.

To control the range of Fourier modes, define the low modes instead by truncating at $L_0=|\ell|/4m$, so that each iteration expands the range of modes appearing in $\zeta_\ell^N$ by $L_0/2m$. The bounds on $\xi_\ell^N$ then follow as before, using higher Sobolev norms in (\refeq{highmodes}) to bound the remainder after $m$ iterations. The higher derivative estimates follow from repeating the argument applying estimates for nested sequences of commutators $[r\nabla_z,\slashed D_{N}]$ and $[\nabla_t, \slashed D_{N}]$. Each application of $\nabla_t$ requires increasing the bound by a factor of $|\ell|$, but each application of $\nabla^\text{b}$-derivatives only by a universal constant. 
\end{proof}

\subsection{Fredholm Properties}
\label{section4.2}
This subsection defines the map $\text{ob}$ from Proposition \ref{cokproperties} and proves that it is Fredholm (more precisely, this is a preliminary version of $\text{ob}$, which is later corrected by a compact operator). Let $\chi_1$ be a smooth cut-off function supported in the region $r<r_N/2$, where $r_N$ is sufficiently small that Lemmas \ref{PsiNbasis} and \ref{expdecayN} hold, and equal to $1$ in the region $r\leq r_N/4$. Define $\text{ob}: L^2(\mathcal Z_0; \mathcal C_0)\to \text{\bf{Ob}}(\mathcal Z_0)^\perp$ by 
\be
\text{ob}(\xi):=  \Pi^\perp_0 (\chi_1 \Pi_N \frak P_\circ(\eta)) \label{obprelimdef}
\ee

\noindent where $\Pi_0^\perp$ is the $L^2$-orthogonal projection to $\text{\bf Ob}(\mathcal Z_0)^\perp$ as in Definition \ref{Obdefinition}, so that $e^{i\ell t} \mapsto  \Pi^\perp_0 (\chi_1 \Psi_\ell^N)$, where $\Psi_\ell^N$ are as in Section \ref{section4.1}. The definition is extended to include the $\ell=0$ Fourier modes by setting $\Psi_0^N:= \Psi_0^\circ$ (this $\ell=0$ mode is not $L^2$ on $Y_\circ$, but $\chi_1 \Psi_0^N\in L^2(Y\setminus \mathcal Z_0)$).

\begin{lm}
$\text{ob}:L^2(\mathcal Z_0; \mathcal C_0)\to \text{\bf{Ob}}(\mathcal Z_0)^\perp$ as defined by (\ref{obprelimdef}) is Fredholm. \label{MFredholm}
\end{lm}

\begin{proof}
The $\ell=0$ mode may be ignored as it spans a finite-dimensional space; likewise the distinction between $\Pi_0,\Pi_0^\perp$ may be ignored since it has finite rank. Precomposing with the isomorphism from Lemma \ref{PsiNbasis} shows Fredholmness of $\text{ob}$ is equivalent to Fredholmness of the map $M$ defined by \bea M:\ker(\slashed D_N|_{L^2}) &\to&  \text{\bf Ob}(\mathcal Z_0)^\perp \\ \Psi & \ \mapsto \ & \chi_1 \Psi -\slashed D v_\Psi \hspace{1.5cm}\text{where }\hspace{1.5cm}v_\Psi:=P \slashed D(\chi_1 \Psi). \eea 

\noindent Since $\Pi_0=\text{Id}-\slashed DP\slashed D$ by Corollary \ref{Pdiagram}. 

Define a pseudo-inverse \bea M^\dag: \text{\bf Ob}(\mathcal Z_0)^\perp &\to &  \ker(\slashed D_N|_{L^2}) \\ \Phi &\  \mapsto \ & \chi_1 \Phi -\slashed D_N u_\Psi \hspace{1.5cm}\text{where }\hspace{1.5cm}u_\Phi:=P_N \slashed D_N(\chi_1 \Phi).  \eea

\noindent To prove the lemma, it suffices to verify that $M^\dag M=\text{Id} + A_1$ and $MM^\dag=\text{Id}+A_2$ for compact operators $A_1,A_2$. First, note that standard elliptic theory implies the following: if $K\Subset Y\setminus \mathcal Z_0$ is compactly contained in the complement of $\mathcal Z_0$, then the restriction \be R:\text{\bf Ob}(\mathcal Z_0)^\perp \to rH^1_e(K)\label{restriction}\ee
 
\noindent  is compact.  Indeed, since $\slashed D$ is uniformly elliptic away from $\mathcal Z_0$, this follows from standard elliptic bootstrapping and Rellich's Lemma. The equivalent statement holds on $K_N\Subset N$, but compactness then also {\it a priori} requires that $K_N$ be bounded in the non-compact $N$.

A straightforward computation shows 
 \begin{eqnarray} (M M^\dag-\text{Id})\Phi&=& (\chi_1^2-1) \Phi -  \chi_1 \slashed D_N u_{\Phi} - \slashed Dv_{M^\dag \Phi}.\label{QQ^*2}\\
 (M^\dag M-\text{Id})\Psi&=& (\chi_1^2-1) \Psi -  \chi_1  \slashed Dv_{\Psi} - \slashed D_N u_{M\Psi}.\label{Q^*Q}
 \end{eqnarray}

\noindent and we claim the right hand sides of both expressions are compact. For the first expression, $\text{supp}(\chi_1^2-1)\Subset Y\setminus \mathcal Z_0$ hence compactness follows from what was said about the restriction map (\refeq{restriction}). Likewise, (\refeq{restriction}) implies that the map $\Phi \mapsto u_{\Phi}$ is compact since it may be written as the composition $$u=P_N\circ d\chi_1. \circ R|_{\text{supp}(d\chi_1)}.$$
Similarly, $\Psi\mapsto v_\Psi$ is compact. Since the remaining terms on the right hand side of \label{QQ^*} factor through these, we conclude that $M M^\dag-\text{Id}$ is compact. The only difference for $M^\dag M-\text{Id}$ is that $(\chi_1^2-1)$ is not compactly supported on $Y_N$. Nevertheless, a standard diagonalization using the decay properties of $\Psi^N_\ell = \Psi_\ell^\circ + \zeta^N_\ell + \xi_\ell^N$ shows that it is compact on elements of $\ker(\slashed D_N|_{L^2})$ (choose subsequences on that simultaneously converge on $r\leq n$ and on the span of $|\ell|\leq n$).
\end{proof}

\subsection{The Index via Concentration}
\label{subsection4.3}
This subsection proves $\text{ob}: L^2(\mathcal Z_0; \mathcal C_0) \to \text{\bf Ob}(\mathcal Z_0)^\perp$, which is Fredholm by Lemma \ref{MFredholm}, has index 0. This is done by introducing a family of perturbations depending on $\mu \in \R$ \footnote{This approach was suggested to the author by Clifford Taubes.} $$\slashed D_\mu:= \slashed D+\mu J$$ 
where $J$ is a complex anti-linear map with $J^2=-\text{Id}$. As $\mu \to \infty$, elements of $\text{\bf {Ob}}(\mathcal Z_0)^\perp$ become increasingly concentrated near $\mathcal Z_0$, and for $\mu$ sufficiently large we may conclude that the $\mu$-version of $M_\mu$ is an isomorphism. There are two subtleties in this. First, one must be careful to ensure the family $\text{ob}_\mu$ can be viewed on a fixed Banach space (as $ \ker(\slashed D_\mu |_{rH^1_e})$ may jump in dimension as $\mu$ varies). Second, the role of the $\ell=0$ modes for the index must be clarified. 

To elaborate on the second point: recall that on $Y_\circ$ from Example \ref{euclideanexample} there are two linearly independent solutions in the $\ell=0$ Fourier mode, these being $(1/\sqrt{z},0)$ and $(0,1/\sqrt{\overline z})$. It is not at first clear which subset of these should contribute to the index; it will be shown that as $\mu\to 0$ this four (real) dimensional space splits into two subspaces of exponentially growing and decay modes, and only the decaying modes contribute.

\begin{lm}\label{Mindex}
The Fredholm map $$\text{ob}: L^2(\mathcal Z_0; \mathcal C_0) \to \text{\bf Ob}(\mathcal Z_0)^\perp $$
\noindent has index zero. 
\end{lm}

\begin{proof}
Let $\overline {\text{ob}}:L^2(\mathcal Z_0; \mathcal C_0)\oplus \ker(\slashed D|_{rH^1_e}) \to \text{\bf Ob}(\mathcal Z_0)$ be defined by $\overline {\text{ob}}=\text{ob}\oplus \iota$ where $\iota$ is the inclusion.  Similarly, let $\overline{\slashed D}= (\slashed D, \pi_1)$ where $\pi_1: rH^1\to \ker(\slashed D|_{rH^1_e})$ is the $L^2$-orthogonal projection. The problem may be recast as a problem on fixed Banach spaces by considering the operator 

$$\overline{\mathcal Q}_0:= \begin{pmatrix} \overline{\text{ob}} & 0 \\ 0 & \overline {\slashed D} \end{pmatrix}: \begin{matrix}L^2(\mathcal Z_0; \mathcal C_0)\oplus \mathcal K_0\\ \oplus  \\ r H^1_e \end{matrix} \ \lre \  \begin{matrix} \text{\bf Ob}(\mathcal Z_0)  \\ \oplus \\ \text{range}(\slashed D)\oplus \mathcal K_0 \end{matrix}=L^2(Y;S_0)\oplus \mathcal K_0$$ 

\noindent where $\mathcal K_0$ is shorthand for $\ker(\slashed D|_{rH^1_e})$. $\overline{ \slashed D}$ is an isomorphism (hence Fredholm with index 0) by fiat, so $\overline{\mathcal Q}_0$ is Fredholm by Lemma \ref{MFredholm}. It therefore suffices to show that $\overline{\mathcal Q}_0$ has Index 0.

Recall that the definition (\refeq{obprelimdef}) depends implicitly on the choice of parametrix $P$ employed in the projections $\Pi=\text{Id}-\slashed D P \slashed D$. If this parametrix $P$ is replaced by another parametrix $P'$ for $\slashed D\slashed D: rH^{1}_e\to rH^{-1}_e$ then the resulting 
\be \overline {\mathcal Q}'_0:=\overline {\text{ob}}'  \oplus \overline{\slashed D}\label{unbarredQ}\ee
\noindent differs by compact operators, hence is Fredholm of the same index as $\overline{\mathcal Q}_0$.   

Now set $\slashed D_\mu:= \slashed D+ \mu J$ for $\mu \geq 0$. Since the Weitzenb\"ock formula becomes \be \slashed D_\mu^\star \slashed D_\mu= (\slashed D-\mu J)(\slashed D+\mu J)=\slashed D^\star \slashed D+\mu^2\label{muweitzenbock},\ee the proofs of Proposition \ref{injclosedrange} and Lemma \ref{secondorderoperator} apply to show that $\slashed D_\mu: rH^1_e\to L^2$ has finite-dimensional kernel and  closed range, and $\slashed D_\mu^\star \slashed D_\mu: rH^1_e\to rH^{-1}_e $ is Fredholm. Let $P_\mu$ be the corresponding parametrix defined by (\refeq{Pdef}). The proofs of Lemmas  \ref{PsiNbasis} and \ref{MFredholm} apply equally well to define a map $\text{ob}_\mu$ and show that $$\overline{\mathcal Q}_\mu=L^2(\mathcal Z_0; \mathcal C_0)\oplus \mathcal K_0 \oplus r H^1_e \ \lre \ L^2(Y;S)\oplus \mathcal K_0  $$ 

\noindent is a Fredholm operator for each $\mu$. Note that inclusion $\iota$ and projection $\pi_1$ are still those for the $\mu=0$ operator and its kernel $\mathcal K_0$. $\overline Q_\mu$ is not {\it a priori} a continuous family, since jumps in the dimension of $\mathcal K_\mu$ result in discontinuities of $\overline{P}_\mu$ as defined by (\refeq{Pdef}). Instead, let $P_\mu$ be a continuous family of parametrices for $\slashed D_\mu^\star \slashed D_\mu$. As in (\refeq{unbarredQ}), the resulting family of operators differs by compact operators, resulting in a now continuous family of Fredholm operators $\overline{\mathcal Q}'_\mu$ with the same index as ${\overline {\mathcal Q}}_0$. After this alteration, it suffices to show that the index is zero for $\mu>>0$.

For $\mu$ sufficiently large, the Weitzenb\"ock formula (\refeq{muweitzenbock}) implies that $\ker(\slashed D_\mu|_{rH^1_e})=0$, so it may be arranged by a further homotopy of parametrices that $\mathcal Q_\mu$ is formed using $\overline P_\mu= (\slashed D^\star_\mu \slashed D_\mu)^{-1}$ once $\mu$ is large. For fixed large $\mu$, removing the $\mathcal K_0$ summands form both the domain and range does not disrupt Fredholmness nor alter the index, so these may be safely ignored. Furthermore, there is new splitting $L^2=\ker(\slashed D^\star_\mu|_{L^2})\oplus \text{range}(\slashed D_\mu|_{rH^1_e})$ in which one may now write $$ \mathcal Q^\mu =\begin{pmatrix}\text{ob}_\mu & 0 \\ 0 & \slashed D_\mu \end{pmatrix} : \begin{matrix}L^2(\mathcal Z_0; \mathcal C_0) \\ \oplus  \\ r H^1_e \end{matrix} \ \lre \  \begin{matrix}\ker(\slashed D^\star_\mu|_{L^2})  \\ \oplus \\ \text{range}(\slashed D_\mu) \end{matrix}$$
\noindent where $\text{ob}_\mu$ is the $\mu$-version of (\refeq{obprelimdef}). Since $\slashed D_\mu$ is injective, hence an isomorphism onto its range, it suffices now to show that $\text{ob}_\mu$ is an isomorphism for $\mu>>0$. Finally, since $\slashed D_\mu$ is injective once $\mu$ is sufficiently large independent of small variations in the metric, it may be arranged by a further homotopy through Fredholm operators that the metric is a product for $r\leq r_0$. The proof is then completed by the subsequent two lemmas. 
\end{proof}

The next lemma shows that the perturbation $\mu J$ means the $L^2$-kernel enjoys an additional decay factor of $e^{-\mu r}$ compared to the $\mu=0$ case, thus it is concentrated more strongly near $\mathcal Z_0$. The proof is an elementary exercise in solving ODEs by diagonalizing matrices since the Fourier modes decouple. Let $(N,g_\circ)$ be the tubular neighborhood from Section \ref{section4.1} equipped with the product metric, and $\slashed D_{N,\mu}$ the perturbed Dirac operator on it.

\begin{lm}

\label{Dmu} The perturbed Dirac operator $$\slashed D_{N,\mu}:rH^1_e\lre L^2$$
  is injective, and its extension to $L^2$ has $\ker(\slashed D^\star_{N,\mu}|_{L^2})$ characterized by the following. 
  
  \begin{itemize}
  \item There is a real 2-dimensional subspace of $\ker(\slashed D^\star_{N,\mu}|_{L^2})$ in the $\ell=0$ modes. It is given by the span over $\R$ of $$\Psi^+_0=\begin{pmatrix} \frac{e^{-\mu r}}{\sqrt{z}} \\ 0  \end{pmatrix}\hspace{2cm}\Psi^-_0=\begin{pmatrix} 0 \\ \frac{e^{-\mu r}}{\sqrt{\overline z}}   \end{pmatrix} $$
  \item There is a real 4-dimensional subspace of $\ker(\slashed D_{N,\mu}^\star|_{L^2})$ in the $\pm\ell$ modes spanned over $\R$ by spinors $$\Psi_{|\ell|,k}=\frac{e^{\pm i \theta/2}}{r^{1/2}} e^{ -\sqrt{ \ell^2 + \mu^2}}  e^{\pm i \ell t}v_k$$
  \noindent where $v_k \in \R^4$ for $k=1,\ldots, 4$.  \qed
  
  \end{itemize}
  \label{muperturbation}
\end{lm}

It may be assumed that $\text{ob}_\mu$ sends the real and imaginary parts of the constant mode to $\Psi_0^\pm$ respectively. 

\begin{lm}
For $\mu>>0$, $$\text{ob}_\mu: L^2(\mathcal Z_0; \mathcal C_0)\to \ker(\slashed D^\star_\mu|_{L^2})$$
is an isomorphism. 
\end{lm}

\begin{proof}By Lemma \ref{Dmu}, $\Pi_{N,\mu}\frak P_\circ : L^2(\mathcal Z_0; \mathcal C_0)=\ker(\slashed D^\star_{N,\mu} |_{L^2})$ is an isomorphism (where the extension to the $\ell=0$ modes is as stated preceding the lemma). As in the proof of Lemma \ref{MFredholm}, it therefore suffices to show the following maps are isomorphisms: define $M_\mu :\ker(\slashed D^\star _{N,\mu}|_{L^2})\to \ker(\slashed D^\star _\mu |_{L^2}) $ and $M_\mu^\dag: \ker(\slashed D^\star _\mu |_{L^2}) \to \ker(\slashed D^\star _{N,\mu}|_{L^2})$ by
\bea  M_\mu(\Psi) & \ = \ & \chi_1 \Psi -\slashed D_\mu v_\Psi \hspace{1.8cm}\text{where }\hspace{1.5cm}v_\Psi:=P_\mu \slashed D^\star_\mu(\chi_1 \Psi). \\ M_\mu ^\dag (\Phi) &\ = \ & \chi_1 \Phi -\slashed D_{N,\mu} u_\Psi \hspace{1.5cm}\text{where }\hspace{1.5cm}u_\Phi:=P_{N,\mu} \slashed D^\star_{N,\mu}(\chi_1 \Phi).  \eea 

\noindent Here, $P_\mu, P_{N,\mu}$ are the true inverses. Note also that $(1-\pi_1)=\text{Id}$ once $\slashed D_\mu$ is injective, so the different between the $L^2$-orthogonal projections $\Pi_\mu, \Pi_\mu^\perp$ is again immaterial. 

By the explicit forms in Lemma \ref{muperturbation}, every $\Psi\in \ker(\slashed D^\star _\mu|_{L^2})$ on $N$ satisfies \be \|\Psi\|_{L^2(\text{supp}(d\chi_1))} \leq C e^{-\mu r_0/c_1}\|\Psi\|_{L^2}\label{exp-mu}\ee on $\text{supp}(d\chi_1)$. It then follows from the expression (\refeq{Q^*Q}) that 

$$\|(M_\mu^\dag M_\mu-\text{Id})\Psi\|_{L^2}\leq C e^{-\mu r_0/c_1}\|\Psi\|_{L^2},$$
hence for $\mu$ sufficiently large, $M_\mu^\dag M_\mu$ is an isomorphism thus $M_\mu$ is injective. 

Surjectivity follows by the same argument with $M_\mu M_\mu^\dagger$ where (\refeq{exp-mu}) is replaced by the bound  \be \|\Phi\|_{L^2(\text{supp}(d\chi_1))} \leq \frac{C}{\mu} \|\Phi\|_{L^2(Y)}\label{1/mu}\ee
for $\Phi \in \ker(\slashed D^\star_\mu)$ on $Y$. To prove (\refeq{1/mu}), let $\rho$ denote a cut-off function supported equal to 1 on $Y-N_{r_0/8}(\mathcal Z_0)$ so that $\rho=1$ on $\text{supp}(\chi_1)$. Integrating by parts shows 
\bea \int_{Y\setminus \mathcal Z_0}  \rho \br  J \Phi, \slashed D \Phi\kt&=&\int_{Y\setminus \mathcal Z_0} \rho \br J\slashed D\Phi, \Phi\kt + \br d\rho.J\Phi, \Phi\kt  \ dV\\
&=& -\int_{Y\setminus \mathcal Z_0}  \rho \br  \slashed D \Phi,  J\Phi\kt+ \int_{Y\setminus \mathcal Z_0} \br d\rho. J\Phi, \Phi\kt  \ dV\eea
since $\slashed DJ=J\slashed D$ and $ J^\dag =-J$. Consequently, since $d\rho$ is bounded by a universal constant, 
\be 2 \text{Re}\br \rho  J\Phi, \slashed D\Phi \kt_{L^2}\leq C \|\Phi\|_{L^2}. \label{thisone}\ee
Then, if $\Phi\in \ker(\slashed D^\star_\mu)$,  
\bea
0= \br \rho  J \Phi, (\slashed D-\mu J )\Phi\kt_{L^2}&= & -\mu \br \rho \Phi, \Phi\kt_{L^2} + \br \rho  J \Phi, \slashed D \Phi\kt_{L^2} \ \ \ \ \overset{(\refeq{thisone})}\Rightarrow \ \ \ \  \mu\|\Phi\|_{L^2(\rho=1)} \leq C \|\Phi\|_{L^2(Y)}.
\eea

 \noindent The latter gives (\refeq{1/mu}) which implies $M_\mu$ is surjective for $\mu$ sufficiently large. This completes the lemma and thus the proof of Lemma \ref{Mindex}.
\end{proof}
 
 \subsection{The Obstruction Map}
 \label{section4.4}
This subsection completes the proof of Propositions \ref{cokproperties} and \ref{cokproperties2}. This is done by altering the preliminary version of $\text{ob}$ defined by (\refeq{obprelimdef}), which is Fredholm of index 0 by  Lemma \ref{Mindex}, by a compact operator.

Let $L^2(\mathcal Z_0; \mathcal C_{0})_{L_0}$ denote the subspace spanned by $e^{i\ell t}$ for $|\ell| \geq L_0$. 

\begin{lm}\label{injbigN}
For $L_0$ sufficiently large, the restricted map $$\text{ob}|_{L_0}:L^2(\mathcal Z_0; \mathcal C_{0})_{L_0}\to \text{\bf Ob}(\mathcal Z_0)^\perp$$
is injective. Moreover, $\Psi_\ell=\text{ob}|_{L_0}(e^{i\ell t})$ admits a decomposition satisfying the conclusions of Proposition \ref{cokproperties2}. 
\end{lm}

\begin{proof}  Since $\frak P_N=\Pi_N\frak P_\circ: L^2(\mathcal Z_0; \mathcal C_{0})\to \ker(\slashed D_N|_{L^2})$ is a bounded linear isomorphism with bounded inverse by Lemma \ref{PsiNbasis}, it suffices to show that $M=(1-\pi_1)\circ \Pi_0: \ker(\slashed D_N|_{L^2})\to \text{\bf Ob}(\mathcal Z_0)^\perp$ is injective. Thus let $\Psi^N\in \text{im}(\text{ob}|_{L_0})$ be such that $\|\Psi^N\|_{L^2}=1$. We may write $\Psi^N=\Psi^\circ + \zeta+\xi$ as in Lemma \ref{expdecayN}, where each term is the sum over $\zeta_\ell$ for $|\ell|\geq L_0$ of the corresponding terms in Lemma \ref{expdecayN}.  

Each $\Phi\in \ker(\slashed D|_{rH^1_e})$ is polyhomogeneous by Proposition \ref{asymptoticexpansion}, thus for every $m\in \N$ there is a bound $|\Phi^\text{high}|\leq C_mL_0^{-m}$, where $\Phi^\text{high}$ denotes the restriction to the Fourier modes $|\ell|\geq L_0/2$ in the $t$-direction in Fermi coordinates. For $\Phi_\alpha$ a basis of $\ker(\slashed D|_{rH^1_e})$, it follows that 
\bea
\pi_1(\chi_1\Psi^N)=\sum_{\alpha=1}^K \br \chi_1( \Psi^\circ + \zeta + \xi) , \Phi_\alpha \kt_{L^2}\leq C_m L_0^{-m}
\eea
\noindent where the Fourier mode restrictions from Lemma \ref{expdecayN} is used to bound the $(\Psi^\circ +\zeta)$ terms, and the bounds from Corollary \ref{higherordercokernel} are used to bound the $\xi$ term. 

The same bounds of Lemma \ref{expdecayN} and Corollary \ref{higherordercokernel} imply that $\Psi^\circ+\zeta$ is exponentially small on $\text{supp}(\chi_1)$, thus since $\slashed D, P$ are bounded,  
\bea
\|\slashed D P \slashed D(\chi_1 \Psi^N)\|_{L^2} = \|\slashed D P (d\chi_1 \Psi^N)\|_{L^2} \leq C \text{Exp}(-\tfrac{L_0}{c}) + C L_0^{-m}.
\eea

\noindent Combining these, we find that 

\bea
\|(1-\pi_1)\Pi_0(\chi_1 \Psi^N)\|_{L^2}&=& \|(1-\pi) (\text{Id}-\slashed D P \slashed D )(\chi_1 \Psi^N)\|_{L^2}\\&\geq& \|\Psi^N\|_{L^2}- \|\pi_1(\chi_1\Psi^N)\|_{L^2} - \|(1-\pi_1)\slashed DP\slashed D(\chi_1 \Psi^N)\|_{L^2} \\ &\geq& 1-C_mL_0^{-m}
\eea

\noindent and injectivity follows once $L_0$ is sufficiently large for $m=4$. 

The final statement that $\Psi_\ell$ admits a decomposition satisfying the conclusion of Proposition \ref{cokproperties2} is immediate since $\Psi^N_\ell$ satisfies the conclusions of Proposition \ref{cokproperties2} by Lemma \ref{expdecayN}. Indeed, repeating the argument above for each index $\ell$ individually shows that the difference $\Psi_\ell-\Psi_\ell^N=O(L_0^{-m})$ may can be absorbed into $\xi_\ell$ without disrupting the bound for each $|\ell|\geq L_0$. 
\end{proof}

Given Lemma \ref{injbigN}, $\text{ind}(M)=0$ means that the (complex) codimension of $\text{Im}(M|_{L_0})\subseteq \text{\bf Ob}$ is $2L_0+1$, and we can make the following definition: 

\begin{defn}\label{Psibasisdef}
The {\bf Obstruction Basis} is defined as 

$$\Psi_\ell:= \begin{cases} \Pi_0^\perp(\chi_1 \Psi_\ell^N)\hspace{2cm}  |\ell|> L_0 \\ \Psi_\ell \hspace{3.3cm} |\ell|\leq  L_0\end{cases}$$
where $\Psi_\ell$ for $|\ell|\leq L_0$ is chosen to be an orthonormal basis of the orthogonal complement of $\text{Im}(M|_{L_0})\subseteq \text{\bf Ob}(\mathcal Z_0)$. It then follows that the map amended in these low modes \bea \text{ob}^\text{pre}: L^2(\mathcal Z_0; \mathcal C_0)\oplus \ker(\slashed D|_{rH^1_e}) &\to& \text{\bf Ob}(\mathcal Z_0) \\
(e^{i\ell t}, \Phi)& \mapsto &\Psi_\ell + \Phi \eea
is an isomorphism. Additionally, by the proof of Lemma \ref{injbigN}, each $\Psi_\ell$ admits a decomposition \be \Psi_{\ell} =\chi_1\Psi_\ell^\circ + \zeta_{\ell}+ \xi_{\ell}\label{finaldecomposition}\ee
satisfying the desired conclusion of Proposition \ref{cokproperties2} (the statement of which are vacuous on the finite range $|\ell|\leq L_0$). 
\end{defn}

The above map $\text{ob}^\text{pre}$ is a preliminary version of the map $\text{ob}$. Thus far, we have shown that $\text{ob}^\text{pre}$ obeys the necessary bounds for the decomposition in Proposition \ref{cokproperties2}. What remains to be shown is that the projection can be calculated by the sequence of inner product (\refeq{innerprods}). Arranging this requires altering the definition of $\text{ob}^\text{pre}$ to obtain the final map $\text{ob}$. 

Indeed, {\it a priori}
since the basis $\Psi_\ell$ is not necessarily orthonormal, the coefficients of $\Psi=c_\ell \Psi_\ell$ are not calculated by the $L^2$-inner product, i.e. in general 
  
  $$(\text{ob}^\text{pre})^{-1}(\Pi_0\psi)\neq \left(\sum_{\ell\in \Z} \br \psi , \Psi_\ell\kt_{\C} e^{i\ell t} \ , \  \sum_\alpha \br  \psi, \Phi_\alpha\kt\Phi_\alpha\right),$$ 
  
 \noindent where $\alpha$ indexes a basis of $\ker(\slashed D|_{rH^1_e})$. Rather frustratingly, one cannot orthonormalize and retain the decay properties of Proposition \ref{cokproperties2} (disrupting these would lead to certain error terms being unbounded later, so the decay properties are essential). To amend this without orthonormalizing, we precompose $\text{ob}^\text{pre}$ with a change of basis\footnote{Equivalently, this may be viewed as endowing $L^2(\mathcal Z_0;\mathcal C_0)$ with an alternative inner product with comparable norm.} $U: L^2(\mathcal Z_0;\mathcal C_0)\to L^2(\mathcal Z_0;\mathcal C_0)$. Specifically, let $U$ be defined by the linear extension of \be U(c_k e^{ikt}):=\sum_{\ell \in \Z} \br  \text{ob}^\text{pre}(c_k e^{ikt}), \Psi_\ell\kt \  e^{ikt}=\sum_{\ell \in \Z} \br  c_k \Psi_k, \Psi_\ell\kt  \ e^{ikt}.\label{Udef}\ee

 \begin{lm}\label{Ucorrection}
 For $L_0$ sufficiently large, $U: L^2(\mathcal Z_0;\mathcal C_0)\to L^2(\mathcal Z_0;\mathcal C_0)$ is an isomorphism, and $$\text{ob}:= \text{ob}^\text{pre} \circ U^{-1}$$ satisfies the conclusions of Proposition \ref{cokproperties} and Proposition \ref{cokproperties2}]. 
 \end{lm}  
 
 \begin{proof} Provided that $U$ is an isomorphism, the conclusion of the propositions follow from directly from the definition (\refeq{Udef}). Indeed, $\text{ob}$ is clearly an isomorphism if $U$ is since it has already been established that $\text{ob}^\text{pre}$ is an isomorphism (in Definition \ref{Psibasisdef}), which is the assertion of Proposition \ref{cokproperties}. Additionally, using (\refeq{Udef}), one has that for a spinor $\psi \in L^2$ 
 \bea
  \text{ob}^{-1}(\Pi_0\psi)=   U U^{-1}(\text{ob}^{-1}\Pi_0\psi)&= &  \sum_{\ell \in \Z} \br \text{ob}^\text{pre} U^{-1}(\text{ob}^{-1}\Pi_0(\psi)), \Psi_\ell  \kt \  e^{i\ell t}\\
  &=& \sum_{\ell \in \Z} \br \text{ob}(\text{ob}^{-1}\Pi_0(\psi)), \Psi_\ell  \kt_{}   e^{i\ell t}= \sum_{\ell \in \Z} \br \psi , \Psi_\ell  \kt_{} \ e^{i\ell t} \eea
  \noindent which is (\refeq{innerprods}). Since $\Psi_\ell$ is unaltered from the case of $\text{ob}^\text{pre}$ in Definition \ref{Psibasisdef}, the conclusions of Proposition \ref{cokproperties2}  follows. It therefore suffices to show $U$ is an isomorphism, for which we show that $$U=\text{Id} + K$$ where $\|K\|_{L^2\to L^2} \leq CL_0^{-1/8}$. 
  
  To prove this bound on $K$, write $\Psi_\ell= \chi_1 \Psi_\ell^\circ + \Xi_\ell$ where $\Xi_\ell=\zeta_\ell + \xi_\ell$. We claim the following four bounds hold where all inner products are the hermitian inner product on $L^2$: 
  
  \begin{enumerate}
\item[(i)] $\br  \Psi_k, \Psi_\ell\kt=\delta_{k\ell}$ unless both $|k|>L_0$ and $|\ell|>L_0$.
\item[(ii)] $\br \Xi_k, \Xi_\ell\kt\leq \tfrac{C}{|k||\ell|}$. 
\item[(iii)] $\br \Xi_k, \chi_1\Psi_\ell^\circ \kt\leq \tfrac{C}{|k||\ell|}$. 
\item[(iv)] $\br \chi_1\Psi^\circ_k, \chi_1\Psi_\ell^\circ \kt=\delta_{k\ell} + a_{k\ell}$ where $|a_{k\ell}|\leq \tfrac{C}{|k|^{1/2} |\ell|^{1/2}} $ and if  $|k-\ell|\geq |k\ell|^{1/4}$ then $| a_{k\ell}|\leq \tfrac{C}{|k|^2|\ell|^2}$. 
\end{enumerate} 

\noindent (i) holds by Definition \ref{Psibasisdef}. (ii) is immediate from the bounds on $\zeta_\ell + \xi_\ell$ and Cauchy-Schwartz. For (iii), recall from Definition \ref{Psibasisdef} that $\Xi_\ell=\Pi_0^\perp(\chi_1\Psi_\ell^N)$, hence 
\bea \br \Xi_k, \chi_1 \Psi_\ell^\circ\kt=\br \Xi_k, \Pi_0^\perp(\chi_1 \Psi_\ell^\circ)\kt&=&\br\Xi_k , \Pi_0^\perp(\chi_1\Psi_\ell^N) - \Pi_0^\perp (\chi_1\zeta_\ell^N + \chi_1\xi_\ell^N) \kt \\ &=&\br \Xi_k, \Xi_\ell \kt + \br \Xi_k, \Pi_0^\perp(\chi_1\zeta_\ell^N+\chi_1 \xi_\ell^N)\kt  \eea
after which the bound follows from (ii) and the bounds on $\zeta_\ell^N, \xi_\ell^N$ from Lemma \ref{expdecayN}. Finally, for (iv) the integral may be written explicitly as 
$$(1+\text{sgn}(k)\text{sgn}(\ell))\int_{N_{r_0}(\mathcal Z_0)} \chi_1^2 |k|^{\tfrac12}|\ell|^{\tfrac12}\frac{e^{-(|\ell| + |k|)r}}{r} e^{i(k-\ell) t}|g|^{1/2} d\text{vol}$$

\noindent were $|g|^{1/2}=|g_\circ|^{1/2}+O(r)$ is the volume form in Fermi coordinates with $g_\circ$ the product metric. For the term coming from $|g_\circ|^{1/2}$ the integral is (exponentially close to) $\delta_{k\ell}$ by orthogonality in the product case. The $a_{k\ell}$ term arises from integrating the $O(r)$ term, for which direct integration shows that $|a_{k\ell}|\leq C|k|^{-1/2} |\ell|^{-1/2}$. Additionally, since the metric is smooth, the $e^{i(k-\ell)t}$ Fourier mode of the volume form is bounded by $|k-\ell|^{m}$ for $m$ large; the stronger bound in the case that $|k-\ell|\geq |k\ell|^{1/4}$ follows. 

With (i)-(iv) established,  we calculate the $L^2$-norm of $Kc(t)$ for $c(t)=\sum_{k}c_k e^{ikt}$,

\begin{eqnarray}
\big\|(Kc(t)\big\|^2_{L^2}&=& \sum_{|\ell|\geq L_0} \Big |\sum_{|k|\geq L_0} \br c_k \Psi_k, \Psi_\ell\kt -\delta_{k\ell} \Big |^2 \label{4885}\\
&\leq &\sum_{|\ell|\geq L_0} \Big |\sum_{|k|\geq L_0} c_k a_{k\ell} + c_k\br \Xi_k ,\chi_1 \Psi_\ell^\circ\kt + c_k\br \chi_1 \Psi_k^\circ , \Xi_\ell\kt + c_k \br \Xi_k, \Xi_\ell\kt \Big |^2\\
&\leq & C\|c(t)\|_{L^2}\sum_{|\ell|, |k|\geq L_0} |a_{k\ell}|^2 + \frac{1}{|k|^2|\ell|^2}
\label{489}
\end{eqnarray}
where we have used Cauchy-Schwartz and (i)-(iv) from above. 

The $|k|^{-2}|\ell|^{-2}$ term is easily summable, with sum bounded by $1/L_0$. For the $a_{k\ell}$ term, we split the sum over $k$ into two parts, and apply the two cases of item (iv): 

\be
\leq \frac{C}{L_0} + \sum_{|\ell|\geq L_0}\Big(\sum_{|k-\ell| \leq |k\ell|^{1/4}} \frac{1}{|k||\ell|} + \sum_{|k-\ell|\geq  |k\ell|^{1/4}}\frac{1}{|k|^4|\ell|^4}\Big)\label{490}
\ee

\noindent The $|\ell|^4 |k|^4$ term is once again summable and bounded by a constant multiple of $1/L_0$. For the remaining term, observe that $|k-\ell|\leq |k\ell|^{1/4}$ implies that $|\ell|/2\leq |k|\leq 2|\ell|$ provided $L_0$ is large enough. This in turn implies that $|k-\ell|\leq 4|\ell|^{1/2}$, from which it follows that 

\be \sum_{|\ell|\geq L_0}\sum_{|k-\ell|\leq |k\ell|^{1/4}}\frac{1}{|k||\ell|}\leq \sum_{|\ell|\geq L_0}\frac{1}{|\ell|^2} \sum_{|k-\ell|\geq |k\ell|^{1/4}} 1 \leq \sum_{|\ell|\geq L_0}\frac{1}{|\ell|^{3/2}}\leq \frac{C}{L_0^{1/4}}.\label{491}\ee

\noindent It follows that $\|K\|_{L^2\to L^2}\leq CL_0^{-1/8}$ hence $U=\text{Id}+K$ is an isomorphism after possibly increasing $L_0$. This completes the proof of Lemma \ref{Ucorrection}, thus the proofs of Propositions \ref{cokproperties} and \ref{cokproperties2}. 
 \end{proof}

 To conclude this subsection, we briefly note the following higher-regularity extension of the previous lemma: 
 
  \begin{lm} \label{Utame}The map $U$ defined by \refeq{Udef} restricts to an isomorphism $$U: H^m(\mathcal Z_0; \mathcal C_0)\to H^m(\mathcal Z_0; \mathcal C_0)$$ 
 for every $m>0$. 
 \end{lm}
 \begin{proof}As in the proof of the previous Lemma \ref{Ucorrection}, write $U=\text{Id} + K$. It suffices to show that $K: H^{m}\to H^{m+1/8}$ is bounded, i.e. that $K$ is a smoothing operator of order $\tfrac18$. Knowing this, the lemma follows from the ``elliptic estimate''
   \be \|\phi\|_{m}\leq C_m \ ( \|U \phi\|_m + \|\phi\|_{m-1/8} )\label{Kellipticest}\ee
   derived by writing $Id= U-K$ and using the triangle inequality and the fact that $U: L^2\to L^2$ is an isomorphism.

Saying that $K:H^{m}\to H^{m+1/8}$ is bounded is to say that the sum

$$\sum_{|\ell|\geq L_0} \Big |\sum_{|k|\geq L_0} \br c_k \Psi_k, \Psi_\ell\kt -\delta_{k\ell} \Big |^2 |\ell|^{2m+1/4}$$

\noindent is bounded by a constant multiple of $\|c(t)\|_{H^m}$. For $m=0$, this is immediate from (\refeq{491}), where a factor of $|\ell|^{1/4}$ can be spared without disrupting the summability. For $m>0$ the assertion follows from repeating the bounds of (\refeq{4885}--\refeq{489}) in the proof of Lemma \ref{Ucorrection} using the additional bounds that
   
     \be
\frac{|\ell|^{2m+1/4} |a_{k\ell}|^2}{|k|^{2m}} \leq  {C_m}{|\ell|^{1/4}} |a_{k\ell}|^2. \label{492}
  \ee
and applying Cauchy-Schwartz with the grouping $(\tfrac{a_{k\ell}}{|k|^m})(c_k |k|^m)$. The equivalent bound to (\refeq{492}) likewise holds with $b_{k\ell}=\br \Xi_k ,\chi_1 \Psi_\ell^\circ\kt + \br \chi_1 \Psi_k^\circ , \Xi_\ell\kt +  \br \Xi_k, \Xi_\ell\kt$ in place of $a_{k\ell}$; both of these follow from similar considerations as the proofs of (i)--(iv) in Lemma \ref{Ucorrection}, using the Fourier mode restriction on $\zeta_\ell$ the higher-order bounds of Corollary  \ref{higherordercokernel}. 
 \end{proof}
 
  \subsection{The Higher Regularity Obstruction}
  \label{section4.5}
 This subsection refines Propositions \refeq{cokproperties} and \refeq{cokproperties2} to cover the cases of higher regularity. The Dirac operator $$\slashed D: H^{m,1}_{\text{b},e}(Y\setminus \mathcal Z_0;S)\to H^m_\text{b}(Y\setminus \mathcal Z_0; S)$$
 has infinite-dimensional cokernel equal to $\text{\bf Ob} \cap H^m_\text{b}$ by Corollary \ref{higherregularitymappings}. It is not {\it a priori} clear that this cokernel coincides with the natural restriction $\text{\bf Ob}^m:= \text{Im}(\text{ob}|_{H^m(\mathcal Z_0;\mathcal C_0)})$. The next lemma asserts that this is indeed the case. 
 
 \begin{lm}\label{regularitycoincides}
 There is equality $$\text{\bf Ob}^m = \text{\bf Ob} \cap H^m_\text{b}$$ as subspaces of $H^m_{\text b}(Y\setminus \mathcal Z_0;S_0)$. In particular, $\text{ob}|_{H^m}$ restricts to an isomorphism making the following diagram commute. 
  \begin{center}
\tikzset{node distance=3.0cm, auto}
\begin{tikzpicture}
\node(A){$H^{m}(\mathcal Z_0;\mathcal C_0)$};
\node(C)[above of=A,yshift=-1.5cm]{$ L^2(\mathcal Z_0;\mathcal C_0)$};
\node(C')[right of=A]{$\text{\bf Ob} $};
\node(D')[above of=C',yshift=-1.5cm]{$\text{\bf Ob} $};
\node(C'')[right of=C', xshift=-1.6cm, yshift=-.025cm]{$\cap H^m_\text{b}(Y\setminus \mathcal Z_0)$};
\node(D'')[above of=C'',xshift=-.1cm, yshift=-1.5cm]{$\cap L^2(Y\setminus \mathcal Z_0)$};
\draw[->] (A) to node {$\iota$} (C);
\draw[->] (A) to node {$\text{ob}|_{H^m}$} (C');
\draw[->] (C') to node {$ \iota$} (D');
\draw[<-][swap] (D') to node {$\text{ob}$} (C);
\end{tikzpicture}\end{center}
 \end{lm}

 \begin{proof}
 Lemma \ref{Utame} shows that there are equivalences of norms
 \bea H^{m}(\mathcal Z_0; \mathcal C_0) &\overset{U}\sim & H^{m}(\mathcal Z_0; \mathcal C_0)
 \eea
is a bounded linear isomorphism with bounded inverse. It is therefore enough to show that $$\sum_{\ell} c_\ell \Psi_\ell \in H^m_\text{b} \hspace{1cm} \Leftrightarrow \hspace{1cm} \sum_{\ell }|c_\ell | |\ell|^{2m} < \infty.$$ 

\noindent The right hand side equivalent to the $H^m_\text{b}$-norm of $\sum c_\ell \Psi_\ell^\circ$, and the statement then follows from the fact that the projection operator $\slashed D P \slashed D: H^m_\text{b}\to H^m_\text{b}$ is bounded by Corollary \ref{higherregularitymappings}. 
 \end{proof}

 \bigskip

\section{The Universal Dirac Operator}
\label{section5}

This section begins the analysis of the Dirac operator allowing the singular set $\mathcal Z_0$ to vary. This is done by introducing a ``universal'' Dirac operator which is the infinite-dimensional family of Dirac operators parameterized by embedded singular sets near $\mathcal Z_0$. The main result of this section, Proposition \ref{abstractfirstvariation} calculates the derivative of this universal Dirac operator with respect to variations in the singular set. 

For the remainder of the article we assume $(\mathcal Z_0, \ell_0, \Phi_0)$ is regular in the sense of Definition \ref{regulardef}. 

\subsection{Trivializations}
\label{section5.1}

Before calculating the derivative with respect to embeddings, we define the universal Dirac operator more precisely as a map on Banach vector bundles. In this, care must be taken to construct explicit trivializations of these vector bundles; indeed, the present situation is more subtle than the case of scalar-valued functions appearing in \cite{DonaldsonMultivalued}, and imprecision about certain isomorphisms can lead to incorrect formulas for the derivative with respect to deformations of the singular set. 

Consider deformations of the singular set $\mathcal Z_0$ as follows.  Let $$\mathcal E_0 \subseteq \text{Emb}^{2,2}(\mathcal Z_0; Y)$$

\noindent denote an open neighborhood of $\mathcal Z_0$ in the space of embedded links of Sobolev regularity $(2,2)$. For each $\mathcal Z\in \mathcal E_0$, let $(S_{\mathcal Z}, \gamma, \nabla)$ denote the Clifford module defined analogously to $S_0$ in  (\refeq{cliffordmod}) so that $S_{\mathcal Z}:=S_{\frak s_0}\otimes \ell_{\mathcal Z}$. Here $\ell_\mathcal Z\to Y\setminus \mathcal Z$ is the real line bundle whose holonomy representation agrees with that of $\ell_0$ (up to homotopy) equipped with its unique flat connection with holonomy in $\Z_2$. The Dirac operator $\slashed D_\mathcal Z$ is defined as in Definition  \ref{Z2Dirac}, and the Hilbert spaces $rH^1_e(Y\setminus \mathcal Z, S_\mathcal Z), L^2(Y\setminus \mathcal Z, S_\mathcal Z)$ are defined for $\mathcal Z\in \mathcal E_0$ analogously to  \ref{functiondef} but using a weight $r_\mathcal Z\approx \text{dist}(-,\mathcal Z)$. 

Define families of Hilbert spaces 
\bea
\mathbb H^1_e(\mathcal E_0)&:= &\{(\mathcal Z,\ph) \ | \ \mathcal Z \in \mathcal E_0 \ , \ \ph \in rH^1_e(Y\setminus \mathcal Z; S_\mathcal Z) \}\\
\mathbb L^2(\mathcal E_0)&:= &\{(\mathcal Z,\psi) \ | \ \mathcal Z \in \mathcal E_0 \ , \ \psi \in L^2(Y\setminus \mathcal Z; S_\mathcal Z) \}
\eea 

\noindent  which come equipped with projections $p_1: \mathbb H^1_e(\mathcal E_0)\to \mathcal E_0$ and $p_0: \mathbb L^2(\mathcal E_0)\to \mathcal E_0$ respectively.

\begin{lm}\label{trivializationlem}
There are trivializations 
\bea
\Upsilon: \mathbb H^1_e(\mathcal E_0)&\simeq &\mathcal E_0 \times rH^1_e(Y\setminus \mathcal Z_0; S_0) \\ 
\Upsilon:\mathbb L^2(\mathcal E_0)&\simeq& \mathcal E_0 \times \  L^2(Y\setminus \mathcal Z_0; S_0) 
\eea
\noindent which endow the spaces on the left with the structure of locally trivial Hilbert vector bundles. 
\end{lm}

Assuming this lemma momentarily, we define 

\begin{defn}
The {\bf Universal Dirac Operator} is the section $\slashed{\mathbb D}$ defined by

\begin{center}
\tikzset{node distance=2.0cm, auto}
\begin{tikzpicture}
\node(C)[yshift=-1.7cm]{$\mathbb H^1_e(\mathcal E_0)$};
\node(D)[above of=C]{$p_{1}^*\mathbb L^2(\mathcal E_0)$};
\node(D')[xshift=2.0cm, yshift=-.6cm]{$\slashed{\mathbb D}(\mathcal Z, \ph):=\slashed D_{\mathcal Z}\ph$};
\draw[<-] (C) to node {$$} (D);
\draw[->] (C) edge[bend right=45]node[below] {$$} (D);
\end{tikzpicture}
\end{center}
\end{defn}

\bigskip

Before proving Lemma \ref{trivializationlem},  we first construct a chart around $\mathcal Z_0\in \text{Emb}^{2,2}(\mathcal Z_0; Y)$. A choice of Fermi coordinates $(t,x,y)$ on $N_{r_0}(\mathcal Z_0)$ induces an isomorphism $T_{\mathcal Z_0}\text{Emb}^{2,2}(\mathcal Z; Y)\simeq H^2(\mathcal Z_0; N\mathcal Z_0)$.  For a fixed cut-off function $\chi(r): N_{r_0}\to \R$ equal to $1$ for $r\leq r_0/2$ and vanishing for $r\geq r_0$, define an exponential map as follows: given $\eta \in H^2(\mathcal Z_0;N\mathcal Z_0)$ with $\|\eta\|_{2}<\rho_0$ set \be F_\eta(t,z)= (t, z+\chi(r)\eta(t)).\label{diffeos}\ee

\noindent Then define 
\bea
\text{Exp}: H^2(\mathcal Z_0;N\mathcal Z_0)&\to& \text{Emb}^{2,2}(\mathcal Z_0; Y)\\
\eta &\mapsto& \mathcal Z_\eta:= F_\eta[\mathcal Z_0],
\eea

\noindent where $F_\eta[\mathcal Z_0]$ denotes the image under $F_\eta$. Let $\mathcal E_0:= B_{\rho_0}(\mathcal Z_0)\subset H^2(\mathcal Z_0; N\mathcal Z_0)$ be the open ball of radius $\rho_0$.

\begin{lm}
For $\rho_0$ sufficiently small, $F_\eta:Y\to Y$ is a diffeomorphism for each $\eta \in \mathcal E_0$, and the map $\text{Exp}: \mathcal E_0 \to \text{Emb}^{2,2}(\mathcal Z_0; Y)$ is a homeomorphism onto its image. 
\end{lm}
\begin{proof}
Since $\|\eta\|_{C^1}\leq C \|\eta\|_{H^2}\leq C\rho_0$ by the Sobolev embedding theorem, it follows that $$\text{d}F_\eta=\begin{pmatrix} 1 & 0 & 0 \\ \chi \eta'_x  & 1+ \del_x \chi \eta_x  &   \del_y \chi \eta_x  \\ \chi \eta'_y  & \del_x \chi \eta_y  & 1+  \del_y\chi \eta_y   \end{pmatrix}$$ is close to the identity, hence invertible for $ \rho_0$ sufficiently small. $F_\eta$ is therefore a local diffeomorphism by the Inverse Function Theorem. To show it is a diffeomorphism, it then suffices to show it is injective. Note that $F_\eta$ preserves the normal disks $\{t_0\}\times D_{ r_0}$ to $\mathcal Z_0$,  and for each $t_0$, $F_\eta$ increases the coordinate in the direction parallel to $\eta(t_0)$, hence it is injective on each normal disk.  

For the second statement, observe that $F_\eta(t,0,0)=(t, \eta(t))$ is distinct for distinct $\eta \in C^1$, hence $\text{Exp}$ is injective. For surjectivity, since any embedding $\mathcal Z$ close to $\mathcal Z_0$ in $H^2$ is also close in $C^1$, such an embedding must be a graph over $\mathcal Z_0$ in Fermi coordinates. Thus $\mathcal Z=\text{Exp}(\eta)$ for $\eta$ the function defining this graph. Continuity of $\text{Exp}$ and its inverse are verified by standard methods.   
\end{proof}   

\begin{rem}
For each $\eta \in \mathcal E_0$, $F_{s\eta}$ for $s\in (-1,1)$ is a family of diffeomorphisms whose derivative along $\mathcal Z_0$ is equal to $\eta$, but it is not the flow of a time-independent vector field on $Y$ extending $\eta$. This choice simplifies several formulas. 
\end{rem}

We now prove Lemma \ref{trivializationlem} by constructing the trivializations $\Upsilon$. The only slight subtlety here is the association of spinor bundles for different metrics. To highlight the metric dependence, we denote by $S_{h}$ the spinor bundle (without tensoring with $\ell_0$) formed with the spin structure $\frak s_0$ using the metric $h$. 

The spinor bundles for two distinct metrics $h_1,h_2$ are isomorphic, though not canonically. A convenient choice of isomorphism is given via parallel transport on cylinders, following \cite[Sec. 5]{Bourguignon}. Let $h_s$ be a 1-parameter family of metrics interpolating between $h_0$ and $h_1$, for $s\in [0,1]$, consider the (generalized) 4-dimensional cylinder $$X=([0,1]\times Y, ds^2 + h_{s}).$$

\noindent $X$ is spin since $w_2(X)=w_2(Y)=0$, and Spin structures on $X$ are in 1-1 correspondence with those on $Y$. Let $S^\pm_X\to X$ denote the positive and negative spinor bundles on $X$ arising from the spin structure corresponding to the fixed spin structure $\frak s_0$ on $Y$. There is a natural isomorphism $S_X^+|_{Y\times\{s\}}\simeq S_{h_{s}}$ (see \cite[Sec. 4.3]{KM} or \cite[Pg. 4]{HitchinHarmonicSpinors}). Let $\nabla_X$ denote the spin connection on $S_X^+$. Parallel transport along the curve $\gamma_y(s)=(s,y)$ in the $-s$ direction defines a linear isometry $$ \tau_{h_0}^{h_s}(y,s): (S_{h_{s}})_y\to (S_{h_{0}})_y\label{taudef},$$
\noindent where the subscript denotes the fiber over a point $y\in Y$. Together, parallel transport for $s=1$ along all such curves define an isomorphism  \be\tau_{h_0}^{h_1}: S_{h_{1}}\to S_{h_{0}}\label{taudef}\ee 
denoted by the same symbol which is a fiberwise isometry, and likewise for any $s\in [0,1]$. 

We now prove Lemma \ref{trivializationlem} by constructing the trivialization $\Upsilon$. This trivialization is the composition of three isomorphism specified during the proof. 
\begin{proof} [Proof of Lemma \ref{trivializationlem}] For each $\eta\in \mathcal E_0$, let $g_\eta:= F_\eta^*g_0$ denote the pullback metric. In addition, we continue to denote $\mathcal Z_\eta=F_\eta[\mathcal Z_0]$. The proof now has four steps. 

\medskip 

\noindent {\it Step 1:} The pullback $F_\eta^*$ induces a canonical isomorphism 
\be S_{g_0}\otimes \ell_{\mathcal Z_\eta}\simeq  F_\eta^*S_{g_0}\otimes F_\eta^* \ell_{\mathcal Z_\eta}.\label{pullbackvectorbundles}\ee

\noindent There are furthermore canonical isomorphisms 
\be F_\eta^* S_{g_0}\simeq S_{g_\eta} \hspace{2cm} F_\eta^*\ell_{\mathcal Z_\eta}\simeq \ell_{\mathcal Z_0},\label{canonicalbundleisos}\ee

\noindent between the pullback of the spinor bundle and the spinor bundle of the pullback metric, and the real line bundles (the latter up to a global choice of sign). In fact, it is straightforward to check that these isomorphisms naturally intertwine the connections in the sense that they send $F_\eta^\star \nabla^\text{spin}_{g_0}\mapsto \nabla^\text{spin}_{g_\eta}$ and $F_\eta^\star \nabla_{A_\eta}\mapsto \nabla_{A_0}$, where $A_\eta$ denotes the flat connection with holonomy in $\Z_2$ on $\ell_{\mathcal Z_\eta}$. The tensor product of these isomorphisms is denoted 

\be\iota: F^*_\eta S_{g_0}\otimes F_\eta^* \ell_{\mathcal Z_\eta}\to S_{g_\eta}\otimes \ell_{\mathcal Z_0}.\ee

\medskip 
\noindent {\it Step 2:} For $s\in [0,1]$, consider the family of metrics $g_{s\eta}=F_{s\eta}^*g_0$ interpolating between $g_0$ and $g_\eta$. Let 
$$\tau_{g_0}^{g_\eta}: S_{g_\eta}\to S_{g_0}$$

\noindent denote the fiberwise isometry defined in (\refeq{taudef}) setting $h_0=g_0$ and $h_1=g_\eta$. In a slight abuse of notation, we use the same symbol to denote the induced fiberwise isometry $S_{g_\eta}
\otimes \ell_{\mathcal Z_0}\to S_{g_0}\otimes \ell_{\mathcal Z_0}$ which would be more correctly written as $\tau_{g_0}^{g_\eta}\otimes \text{Id}$. 

\medskip 
\noindent {\it Step 3:} For each $\eta\in \mathcal E_0$, define $\Upsilon_\eta:=(\tau_{g_0}^{g_\eta})^{}\circ\iota\circ F_\eta^*$ as the composition

\begin{center}
\tikzset{node distance=3.0cm, auto}
\begin{tikzpicture}
\node(C){$S_{g_0}\otimes \ell_{\mathcal Z_0}$};
\node(D)[left of=C]{$S_{g_\eta}\otimes \ell_{\mathcal Z_0}$};
\node(E)[left of=D]{$ F_\eta^*S_{g_0}\otimes  F_\eta^*\ell_{\mathcal Z_\eta}$};
\node(F)[left of=E]{$S_{g_0}\otimes \ell_{\mathcal Z_\eta}$};
\draw[<-][swap] (C) to node {$\tau_{g_0}^{g_\eta}$} (D);
\draw[<-][swap] (D) to node {$\iota$} (E);
\draw[<-][swap] (E) to node {$ F_\eta^*$} (F);
\end{tikzpicture}

\end{center}

\noindent where $F_\eta^*$ denotes the pullback as before, and $\iota,\tau_{g_0}^{g_\eta}$ are as defined in Step 1 and Step 2 respectively. 

Together, the maps $\Upsilon_\eta$ for $\eta\in \mathcal E_0$ yield a universal trivialization: let $\mathcal Y\to \mathcal E_0$ be the bundle whose fiber over $\eta$ is the Riemannian manifold $(Y\setminus \mathcal Z_\eta,g_0)$, and $\mathcal S\to \mathcal Y\to \mathcal E_0$ be the vector bundle whose restriction to the fiber $Y\setminus \mathcal Z_\eta$ over $\eta$ is $S_{g_0}\otimes \ell_{\mathcal Z_\eta}\to Y\setminus \mathcal Z_\eta$. Together, the maps $\Upsilon_\eta$ yield a map 

$$\Upsilon: \mathcal S\to \mathcal E_0\times (S_{g_0}\otimes \ell_{\mathcal Z_0})$$

\noindent given by $\Upsilon_\eta$ on the fiber over $\eta\in \mathcal E_0$, which is diffeomorphism on each such fiber (these fibers being themselves the total space of a vector bundle). Moreover, for each fixed $\eta$, this diffeomorphism is a linear isometry on the fibers of $S_{g_0}\otimes \ell_{\mathcal Z_\eta}\to (Y\setminus \mathcal Z_\eta,g_0)$. 

\medskip 
\noindent {\it Step 4:} The fiberwise isomorphism $\Upsilon_\eta$ induces a map
$$\Upsilon_\eta: r H^1_e(Y\setminus \mathcal Z_\eta, S_{g_0}\otimes \ell_{\mathcal Z_\eta})\lre r H^1_e(Y\setminus \mathcal Z_0, S_{g_0}\otimes \ell_{\mathcal Z_0})$$
\noindent on sections via pullback, which is denoted by the same symbol. This map is an isomorphism by the naturality of the pullback, and it remains to show that it is bounded. This is obvious provided $\eta$ and thus $g_\eta$ has sufficient regularity \footnote{The case that $\eta\in C^\infty$ is sufficient for the proof of Theorem \ref{mainb}, but the low regularity case is included for completeness.}. 

The maps on sections induced by $F_\eta^*$ and $\iota$ are isometries by construction, thus it suffices to show that the map on sections induced by $\tau_{g_0}^{g_\eta}$ defined in (\refeq{taudef}) is bounded between the versions of $rH^1_e(Y\setminus \mathcal Z_0)$ formed with the metric and spin connections of $g_0, g_\eta$ respectively. To see this, note that because $\eta \in H^2(\mathcal Z_0)$, in Fermi coordinates the pullback metric $g_\eta$ has entries of the form $h(t) g_1(t,x,y)$ where $h(t)\in H^1(\mathcal Z_0)$ and $g_1$ is smooth (cf. Lemma \ref{pullbackmetric} below). Since $H^1(\mathcal Z_0)\hookrightarrow C^0(\mathcal Z_0)$ by the Sobolev embedding, the two volume forms induce equivalent norms. The Christoffel symbols of the connection $\nabla_{B_0}$ formed from the spin connection $g_\eta$ and $B_0$, have one lower regularity, thus include terms of the form $f(t)g_2(t,x,y)$ for $f(t)\in L^2(\mathcal Z_0)$ where $g_2$ is continuous. The equivalence of norms is then a consequence of the ``mixed dimension'' Sobolev multiplication on the solid torus $\mathcal Z_0\times D^2$ $$\|f(t) \ph\|_{L^2(\mathcal Z_0 \times D^2)}\leq C \|f\|_{L^2(\mathcal Z_0)} \|\ph\|_{rH^1_e},$$
\noindent for $f\in L^2(S^1)$ and $\ph \in rH^1_e$. To prove the latter, simply observe that $I(t)^2=\int_{t\times D^2}|\ph|^2 dx dy$ is $L^\infty$ by the Sobolev restriction theorem and then apply Fubini's theorem.

\end{proof}

 \subsection{Universal Linearization}
\label{section5.2}

Using the trivialization constructed in Lemma \ref{trivializationlem}, we may now calculate the (vertical component of the) derivative of the universal Dirac operator considered as a map 
\be \text{d}_{(\mathcal Z_0,\Phi_0)}\slashed{\mathbb D}: H^2(\mathcal Z_0;N\mathcal Z_0)\times rH^1_e(Y\setminus \mathcal Z_0; S_0) \lre L^2(Y\setminus \mathcal Z_0; S_0),\label{universalDeriv}\ee

\noindent where $S_0=S_{g_0}\otimes \ell_{\mathcal Z_0}$ as in Section \ref{section2}. After trivializing, differentiating with respect to a deformation $\eta$ of the singular set becomes differentiation of the Dirac operator with respect to the family of metrics $g_{s\eta}$ for $s\in [0,1]$.
\begin{prop}
\label{abstractfirstvariation}
In the local trivialization provided by $\Upsilon$, the linearization of the universal Dirac operator  
on the spaces (\refeq{universalDeriv}) is given by 
\be \text{d}_{(\mathcal Z_0,\Phi_0)}\slashed {\mathbb D} (\eta, \psi)= \mathcal B_{\Phi_0}(\eta) + \slashed D_{\mathcal Z_0} \psi\label{abstractvariation1}\ee
where 
$$\mathcal B_{\Phi_0}(\eta)= \left( \d{}{s}\Big |_{s=0} \tau_{g_0}^{g_{s\eta}}\circ \slashed D^{g_{s\eta}}_{\mathcal Z_0} \circ (\tau_{g_0}^{g_{s\eta}})^{-1}\right) \Phi_0$$

\noindent is the first variation of the Dirac operator with respect to the family of metrics $g_{s\eta}$ acting on $\Phi_0$.

\end{prop}

\begin{rem}\label{trivializationsremark}
(Cf. Section 4.1 of \cite{DonaldsonMultivalued}) Since the configuration $(\mathcal Z_0, \Phi_0)$ does not lie along the zero-section in $\mathbb H^1_e(\mathcal E_0)$, there is no canonical splitting $$T_{(\mathcal Z_0,\Phi_0)} \mathbb H^1_e(\mathcal E_0) \simeq T_{\mathcal Z_0}\mathcal E_0 \oplus  r H^1_e(Y\setminus \mathcal Z_0).$$

\noindent Thus expression of the derivative (\refeq{universalDeriv}) relies on a choice of connection on the Banach vector bundle $\mathbb H^1_e$\ --- here we have implicitly chosen the pullback of the product connection by $\Upsilon$. Different choices of trivialization will result in different connections and different expressions for the derivative $\text{d}\slashed{\mathbb D}$. Concretely, this choice manifests as the dependence of the family of metrics $g_\eta$ on our choice of diffeomorphisms $F_\eta$. A different choice of family of diffeomorphisms differs from our choice of $F_\eta$ by composing with (a family of) diffeomorphisms fixing $\mathcal Z_0$. Although there are many possible choices (see \cite{PartIII} and \cite{TaubesWuS2}) this choice simplifies many expressions. Of course, the salient properties of the linearization are independent of these choices.   
\end{rem}

\begin{proof}[Proof of Proposition \ref{abstractfirstvariation}] Take a path \bea \gamma: (-\epsilon, \epsilon) &\to& \mathbb H^1_e(\mathcal E_0)
 \\  s &\mapsto &(\mathcal Z_{\eta(s)}, \Phi(s))\eea
\noindent  such that $\gamma(0)=(\mathcal Z_0, \Phi_0)$. Using the chart $\text{Exp}: H^2(\mathcal Z_0;N\mathcal Z_0)\to \mathcal E_0$, we may assume that $\eta(s)=s\eta$.  Let $\mathcal H$ be the section of $\mathbb H^1_e(\mathcal E_0)$ obtained from radial parallel transport of $\Phi_0$ in the connection induced by the trivialization $\Upsilon$. That is, set $$\mathcal H= \Upsilon^{-1}(\mathcal E_0 \times \{\Phi_0\}). $$
 
 \noindent We may write each $\Phi(s) \in rH^1_e(Y\setminus \mathcal Z_{s\eta})$ as the point in $\mathcal H$ plus a vertical vector $\phi(s)= \Upsilon^{-1}(\psi(s))$, i.e. $$\gamma(s)=(\mathcal Z_{s\eta}, \Upsilon_{s\eta}^{-1}(\Phi_0)+ \phi(s) \ )= (\mathcal Z_{s\eta}, \Upsilon_{s\eta}^{-1}(\Phi_0+ \psi(s)) \ ).$$   

\noindent  The derivative in the trivialization given by $\Upsilon$ is \be \d{}{s}\Big |_{s=0} \Upsilon_{s\eta} \circ \slashed {\mathbb D}(\mathcal Z_{s\eta}, \Upsilon_{s\eta}^{-1}(\Phi_0+\psi))= \d{}{s}\Big |_{s=0}  \Upsilon_{s\eta}\circ  \slashed D_{\mathcal Z_{s\eta}}^{g_0} \circ  \Upsilon_{s\eta}^{-1}(\Phi_0+\psi)\label{tocalculatetriv}\ee
\noindent where $\Upsilon$ denotes the trivialization for both $\mathbb H^{1}_e(\mathcal E_0)$ and $\mathbb L^2(\mathcal E_0)$.

Recalling the definition of $\Upsilon_{s\eta}:=(\tau_{g_0}^{g_{s\eta}})^{}\circ\iota\circ F_{s\eta}^*$ from Step 3 in the proof of Lemma \ref{trivializationlem}, the following diagram commutes, where the rightmost vertical arrow is the expression (\refeq{tocalculatetriv}) which we wish to calculate.

\begin{center}
\tikzset{node distance=5.5cm, auto}
\begin{tikzpicture}
\node(A){$rH^1_e(S_{g_{s\eta}}\otimes \ell_{\mathcal Z_0} )$};
\node(B)[above of=A,yshift=-3.3cm]{$L^{2}(S_{g_{s\eta}}\otimes \ell_{\mathcal Z_0})$};
\node(C)[right of=A]{$r H^1_e(S_{g_0}\otimes  \ell_{\mathcal Z_0})$};
\node(D)[right of=B]{$L^2(S_{g_0}\otimes \ell_{\mathcal Z_0})$};
\node(C')[left of=A]{$rH^1_e(S_{g_0}\otimes \ell_{\mathcal Z_{s\eta}})$};
\node(D')[above of=C',yshift=-3.3cm]{$L^2(S_{g_0}\otimes \ell_{\mathcal Z_{s\eta}})$};
\node(E')[below of=A,yshift=4.3cm]{$\begin{pmatrix}\text{varying } g_{s\eta}\\ \text{fixed }\mathcal Z_0\end{pmatrix}$};
\node(F')[left of=E']{$\begin{pmatrix}\text{varying }\mathcal Z_{s\eta}\\ \text{fixed }g_0\end{pmatrix}$};
\node(F')[right of=E']{$\begin{pmatrix}\text{fixed } g_0\\ \text{fixed }\mathcal Z_0\end{pmatrix}$};
\draw[->] (A) to node {$\slashed D_{\mathcal Z_0}^{g_{s\eta}}$} (B);
\draw[->] (A) to node {$\tau_{g_0}^{g_{s\eta}}$} (C);
\draw[<-] (A) to node {$\iota_{s\eta} \circ F_{s\eta}^*$} (C');
\draw[->] (C') to node {$\slashed D_{\mathcal  Z_{s\eta}}^{g_0} $} (D');
\draw[->] (D') to node {$\iota_{s\eta} \circ F_{s\eta}^*$} (B);
\draw[->] (B) to node {$\tau_{g_0}^{g_{s\eta}}$} (D);
\draw[->] [swap](C) to node {$\Upsilon_{s\eta}  \slashed D_{\mathcal Z_{s\eta}}^{g_0}\Upsilon_{s\eta}^{-1} $} (D);
\end{tikzpicture}\end{center}
\noindent  The middle vertical arrow denotes Dirac operator on the bundle $S_{g_{s\eta}}\otimes \ell_{\mathcal Z_0}$ formed using the pullback metric $g_\eta$ and the unique flat connection on $\ell_{\mathcal Z_0}$.

By commutativity, the rightmost vertical arrow is equivalent to the conjugation of the middle arrow by $\tau_{g_0}^{g_{s\eta}}$ and its inverse. Consequently, using the product rule (noting as well that $\psi(0)=0$ and $\tau_{g_0}^{g_{\eta(0)}}=\text{Id}$), 
\bea
\d{}{s}\Big |_{s=0} \Upsilon_{s\eta}\circ  \slashed D_{ \mathcal Z_{s\eta}}^{g_0} \circ  \Upsilon_{s\eta}^{-1}(\Phi_0+ \psi(s))&=& \d{}{s}\Big |_{s=0} \left( \tau_{g_0}^{g_{s\eta}}\circ \slashed D_{\mathcal Z_0}^{g_{s\eta}} \circ (\tau_{g_0}^{g_{s\eta}})^{-1}\right) (\Phi_0 + \psi(s))\\
&=&  \left( \d{}{s}\Big |_{s=0} \tau_{g_0}^{g_{s\eta}}\circ \slashed D_{\mathcal Z_0}^{g_{s\eta}} \circ (\tau_{g_0}^{g_{s\eta}})^{-1}\right)\Phi_0 +  \slashed D^{g_0}_{ \mathcal Z_0} \dot \psi(0).
\eea
as claimed. 
\end{proof} 

\subsection{First Variation Formula}
\label{section5.3}

In order to analyze the derivative of the universal Dirac operator calculated in Proposition \ref{abstractfirstvariation}, a more explicit formula is needed for the variation of the Dirac operator with respect to metrics ($\mathcal B_{\Phi_0}(\eta)$ in \refeq{abstractvariation1}). The formula for this variation is originally due to Bourguignon and Gauduchon \cite{Bourguignon}. A concise proof (in English) was later given in \cite{DiracVariationBar}. See also \cite{DiracVariationNonlinear}. 

Suppose, forgetting any reference to the above situation momentarily, that $g_s$ is a path of metrics on a Riemannian spin manifold $W$. Let $\dot g_s$ denote the derivative of this path at $s=0$, and let $$\tau_{g_0}^{g_s}: S_{g_s}\to S_{g_0}$$ be the isomorphism of spinor bundles defined in (\refeq{taudef}). We obtain a 1-parameter family of operators $$\tau_{g_0}^{g_s} \circ \slashed D_{g_s} \circ (\tau_{g_0}^{g_s})^{-1}: \Gamma(S_{g_0})\to \Gamma(S_{g_0})$$ as the right arrow in the commutative diagram 
\begin{center}
\tikzset{node distance=2.5cm, auto}
\begin{tikzpicture}
\node(A){$\Gamma(S_{g_s} )$};
\node(C)[above of=A, yshift=-1cm]{$\Gamma(S_{g_s})$};
\node(C')[right of=A]{$\Gamma(S_{g_0})$};
\node(D')[above of=C',yshift=-1cm]{$\Gamma(S_{g_0})$};
\draw[->] (A) to node {$\slashed D_{g_s}$} (C);
\draw[->] (A) to node {$\tau_{g_0}^{g_s}$} (C');
\draw[->] (C') to node {$ $} (D');
\draw[<-][swap] (D') to node {$\tau_{g_0}^{g_s}$} (C);
\end{tikzpicture}\end{center}
for every $s$. Letting $\{e_i\}$ be an orthonormal frame for the metric $g_0$ and $\{e^i\}$ its dual frame, Bourguignon and Gauduchon  calculate: 

\begin{thm} {\bf (Bourguignon--Gauduchon \cite{Bourguignon})}
The first variation of the Dirac operator with respect to the family of metrics $g_s$ is given by 

\begin{equation}\left(\d{}{s}\Big |_{s=0} \tau_{g_0}^{g_s}\circ \slashed D_{g_s} \circ (\tau_{g_0}^{g_s})^{-1}\right)\Psi= -\frac{1}{2}\sum_{ij} \dot g_s(e_i,e_j) e^i . \nabla^{g_0}_j \Psi + \frac{1}{2} d \text{Tr}_{g_0}(\dot g_s). \Psi +\frac{1}{2} \text{div}_{g_0}(\dot g_s).\Psi\label{bourguignon}\end{equation}
where $.$ denotes Clifford multiplication in the $g_0$ metric. 
\label{bourguignon}
\end{thm}  
\noindent Note that the first term is independent of the choice of frame for the same reason as the standard Dirac operator. Here, in an orthonormal frame, the $\text{div}_{g_0}(k)$ is the 1-form $-( e_i \intprod\nabla_i  k_{ij}e^i)e^j.$ To give some quick intuition for this slightly unappetizing formula, the first term comes from differentiating the symbol of the Dirac operator (Clifford multiplication), and the second two terms arise from differentiating the Christoffel symbols.

We will apply Bourguignon-Gauduchon's formula (\refeq{bourguignon}) in the case that the family of metrics is the one given by the pullbacks 
\be \dot g_{\eta}:= \d{}{s}\Big |_{s=0} g_{s\eta}= \d{}{s}\Big |_{s=0} F^*_{s\eta} g_0. \label{getadot}\ee

\noindent As in Definition \ref{fermi}, the metric in  Fermi coordinates $(t,x,y)$ on the tubular neighborhood $N_{r_0}(\mathcal Z_0)$ has the form $$g_0= dt^2 + dx^2 + dy^2 + h \hspace{1cm}\text{ where  \ \ \ } |h_{ij}|\leq Cr.$$
\begin{lm}
The derivative of the family of pullback metrics (\refeq{getadot}) is given by 
\begin{equation}
\dot g_\eta = \begin{pmatrix}  0 &  \eta_x' \chi & \eta'_y\chi \\  \eta'_x \chi& 2\eta_x \del_x \chi & \eta_x \del_y \chi + \eta_y \del_x \chi \\  \eta'_y\chi &   \eta_x \del_y \chi + \eta_y \del_x \chi  & 2\eta_y \del_y \chi \end{pmatrix} + h_1 + h_2 
\end{equation}
where 

\begin{itemize}
\item $h_1$is a $O(1)$ term whose entries are formed from products of derivatives of $h_{ij}$ and $\eta$.
 \item $h_2$  is a $O(r)$ term whose entries are formed from products of $h_{ij}$ and products of $\eta, \eta'$. 

\end{itemize}
Here, $\eta=\eta_x + i \eta_y$ and $\eta'=\tfrac{d}{dt} \eta$ and $\dot g_{\eta}$ is as in (\refeq{getadot}). 
\label{pullbackmetric}
\end{lm}

\begin{proof}
Since the diffeomorphism $F_{s\eta}$ is supported in the tubular neighborhood, it suffices to do the calculation in Fermi coordinates.

First, consider the case that $h=0$. Recall $$F_{s\eta}(t,x,y)=(t, x + s\chi(r)\eta_x(t), y+ s \chi(r)\eta_y(t)),$$ hence $$\text{d} F_{s\eta}= \begin{pmatrix} 1 & 0 & 0 \\ s\chi \eta'_x  & 1+ s\del_x \chi \eta_x  &  s \del_y \chi \eta_x  \\ s\chi \eta'_y  & s\del_x \chi \eta_y  & 1+  s \del_y\chi \eta_y   \end{pmatrix}.$$
A quick calculation shows in this case the pullback metric is 

\begin{eqnarray} \dot g_{\eta}&=&\d{}{s}\Big |_{s=0} (\text{d} F_{s\eta})^T  g_0 (  \text{d} F_{s\eta})\\  \\&= &   \begin{pmatrix}  0 &  \eta_x' \chi & \eta'_y\chi \\  \eta'_x \chi& 2\eta_x \del_x \chi & \eta_x \del_y \chi + \eta_y \del_x \chi \\  \eta'_y\chi &   \eta_x \del_y \chi + \eta_y \del_x \chi  & 2\eta_y \del_y \chi \end{pmatrix}.
\label{pullbacks2} 
\end{eqnarray}

 Now in the case that $h\neq 0$, let $\widetilde h_{ij}=h_{ij}(t, z+  F_{s\eta})$. Then the term added to the above is

\begin{eqnarray}
&=& \d{}{s}\Big |_{s=0}( \text{d}F_{s\eta})^T \cdot h(t, z+ F_{s\eta})\cdot ( \text{d}F_{s\eta}) \\
&=& \d{}{s}\Big |_{s=0} (\text{d} F_{s\eta})^T \begin{pmatrix} \widetilde h_{11}  +s\chi ( \widetilde h_{12} \eta'_x + \widetilde h_{13}\eta'_y)& \widetilde h_{12}  + s\del_x \chi (\widetilde h_{12}\eta_x+ \widetilde h_{13}\eta_y)& \widetilde h_{13}  + s\del_y \chi (\widetilde h_{12}\eta_x + \widetilde h_{13}\eta_y)\label{idplussa}\\ 
\widetilde h_{21}+ s\chi (\widetilde h_{22} \eta'_x + \widetilde h_{23}\eta'_y )&\widetilde h_{22}  + s\del_x \chi (\widetilde h_{22}\eta_x+ \widetilde h_{23}\eta_y)& \widetilde h_{23}  + s\del_y \chi (\widetilde h_{22}\eta_x + \widetilde h_{23}\eta_y) \\ 
\widetilde h_{31}+s\chi( \widetilde h_{32} \eta'_x + \widetilde h_{33}\eta'_y) & \widetilde h_{32}  + s\del_x \chi (\widetilde h_{32}\eta_x+ \widetilde h_{33}\eta_y)& \widetilde h_{33}  + s\del_y \chi (\widetilde h_{32}\eta_x + \widetilde h_{33}\eta_y)\end{pmatrix} \nonumber.
\end{eqnarray}

\noindent Write the matrix above as $\widetilde h_{ij} + sA_{ij}$, so that e.g. $A_{11}=\chi \widetilde h_{12}\eta'_x + \widetilde h_{13}\eta'_y$. Then since 
$$\text{d}F_{s\eta}^T= \text{Id} + s\begin{pmatrix} 0 &  \chi \eta'_x & \chi \eta'_y    \\ 0   &  \del_x \chi \eta_x  &   \del_x \chi \eta_y  \\ 0   & \del_y \chi \eta_x  &    \del_y\chi \eta_y   \end{pmatrix} $$
and $\begin{pmatrix}\widetilde h_{ij}\end{pmatrix}$ is symmetric, (\refeq{idplussa}) becomes

\noindent \begin{eqnarray}  & =& \d{}{s}\Big|_{s=0} \left[\begin{pmatrix} \widetilde h_{ij}\end{pmatrix} +s \begin{pmatrix}A_{ij}+ A^T_{ij}\end{pmatrix} + O(s^2)\right] \label{pullbacks22}=\underbrace{ \d{}{s} \Big |_{s=0}\begin{pmatrix}\widetilde h_{ij}\end{pmatrix}}_{:=h_1} + \underbrace{\begin{pmatrix}A_{ij}+ A_{ij}^T\end{pmatrix}}_{:=h_2}. \label{derivativeh}
\end{eqnarray} 

\noindent Call these terms $h_1$ and $h_2$ as indicated. Since  
\bea \d{}{s}\Big |_{s=0}\widetilde h_{ij}&=& \d{}{s}\Big|_{s=0} h_{ij}(t, x+ s\chi \eta_x , y+ s\chi \eta_y)= (\del_x  {h}_{ij})\chi \eta_x + (\del_y {h}_{ij})\chi \eta_y\\
A_{ij} \Big|_{s=0}&= & h_{k\ell} \chi \eta_\alpha'  \ \ \text{  \ or  \ }  \ \ h_{k\ell} \del_\alpha\chi \eta_\beta
\eea

\noindent where $\alpha,\beta$ range over $x,y$ and (summation is implicit in the expression for $A$), these are respectively of the forms claimed for $h_1$ and $h_2$.
\end{proof}

Combining the formula for the linearization of the universal Dirac operator of Proposition \ref{abstractfirstvariation} with the formula of Bourguignon-Gauduchon (Theorem \ref{bourguignon}) and the calculation of the pullback metric in Lemma \ref{pullbackmetric} allows us to immediately deduce the following more concrete expression for the linearization.   

\begin{cor}
The linearization of the universal Dirac operator at $(\mathcal Z_0,\Phi_0)$ is given by 

\begin{eqnarray} \text{d}_{(\mathcal Z_0,\Phi_0)}\slashed {\mathbb D}(\eta,\psi)&=&\left(-\frac{1}{2}\sum_{ij} \dot g_{\eta}(e_i,e_j) e^i . \nabla^{g_0}_j  + \frac{1}{2} d \text{Tr}_{g_0}(\dot g_{\eta}). +\frac{1}{2} \text{div}_{g_0}(\dot g_{\eta}).  + \mathcal R(B_0,\chi\eta).\right)\Phi_0 \ \ \label{firstlinedD}\\ & & + \slashed D\psi\end{eqnarray}
where $\mathcal R(B_0,\eta)$ is a smooth term involving up to first derivatives of $B_0$ and linear in $\chi \eta$, and $.$ denotes Clifford multiplication using the metric $g_0$. Explicitly,  $\dot g_{\eta}$ is given in Fermi coordinates by $$\begin{pmatrix}  0 &  \eta_x' \chi & \eta'_y\chi \\  \eta'_x \chi& 2\eta_x \del_x \chi & \eta_x \del_y \chi + \eta_y \del_x \chi \\  \eta'_y\chi &   \eta_x \del_y \chi + \eta_y \del_x \chi  & 2\eta_y \del_y \chi \end{pmatrix} + h_1 + h_2 $$
with $h_1,h_2$ as in the above Lemma \ref{pullbackmetric}. 
\label{formofDD}
\end{cor}

\begin{proof}
In the case that $B_0=0$, this follows immediately from Theorem \ref{bourguignon} and the above calculation of the pullback metric in Lemma \ref{pullbackmetric}. The line bundle is fixed after pulling back by $F_\eta$ and plays no role. The perturbation $B_0$ pulls back to $F_{s\eta}^*B_0$, and differentiating this yields the term $\mathcal R(B_0,\chi \eta)$. 
\end{proof}

A word of caution to the reader:  the formula for this linearization is slightly deceptive in the following sense. The expression for $\mathcal B_{\Phi_0}(\eta)$ , which is the first line in (\refeq{firstlinedD}) (cf. (\refeq{abstractvariation1})), appears to be a first order term plus a zeroeth order term. But these are the orders in the {\it spinor} $\Phi_0$, and we are viewing it as an equation in the {\it deformation} $\eta$. The variation of the pullback metrics $\dot g_{\eta}$, as above, contains first derivatives of $\eta(t)$, and so the trace and divergence, which contain derivatives of $\dot g_{\eta}$ contain second derivatives of $\eta(t)$. Thus this equation is actually {\it second order} in $\eta$, with the second and third terms being leading order. This is the reason $\eta$ must be taken to be at least $H^2$ in order for this partial derivative to be bounded into $L^2$. 

\begin{rem}\label{nonlinearpullbackmetric}
For later use, we note that the proof of Lemma \ref{pullbackmetric} shows that the complete formula for the pullback metric can be written 
$$g_{s\eta}= g_0 + s \dot g_{\eta} + \frak q(s\eta,s\eta)$$ 

\noindent where $\frak q(s\eta,s\eta)$ is a matrix whose entries are $O(s^2)$ and are formed from finite sums of terms of the following form

\begin{itemize} \item Products of at least two terms of the form \   $\chi \eta'_\alpha$ \ , or \  $\del_\beta \chi \eta_\alpha$ \ , or \   $(\widetilde h-h)\leq C |\chi \eta|$. 
\item Higher order terms of the form $(\widetilde h - h- h_1)\leq C|\chi \eta|^2.$

\end{itemize}
 
\noindent where the bounds on the terms involving $\widetilde h$ follow from Taylor's theorem. \qed
\end{rem}

\bigskip 

\section{Fredholmness of Deformations}
\label{section6}

This section proves Theorem \ref{maina} by calculating the obstruction component of the linearized universal Dirac operator. For the duration of this section, we continue to assume that $(\mathcal Z_0, \ell_0, \Phi_0)$ is a regular (Definition \ref{regulardef}) $\Z_2$-harmonic spinor.  

 Working in the trivialization of Lemma  \ref{trivializationlem} and splitting the domain and codomain into their summands, the linearization has the following block lower-triangular form, where $\Pi_0:L^2\to \text{\bf Ob}(\mathcal Z_0)$ denotes the orthogonal projection as in Definition \ref{Obdefinition}: 

\be \text{d}_{(\mathcal Z_0,\Phi_0)}\slashed {\mathbb D} = \begin{pmatrix}  \Pi_0 \mathcal B_{\Phi_0} & 0 \\  \\  (1-\Pi_0)\mathcal B_{ \Phi_0}& \slashed D\end{pmatrix} \ : \  \begin{matrix} H^2(\mathcal Z_0; N\mathcal Z_0)\\ \oplus \\ r H^1_e(Y\setminus \mathcal Z_0; S_0) \end{matrix} \ \  \lre \ \   \begin{matrix}  \text{\bf Ob}(\mathcal Z_0)\\ \oplus \\ \text{range}(\slashed D|_{rH^1_e}).\end{matrix}\label{blockdecomp}\ee

\noindent Composing with the inverse of the isomorphism $(\text{ob},\iota): L^2(\mathcal Z_0;\mathcal C_0) \oplus \R \to \text{\bf Ob}(\mathcal Z_0)$ from Proposition \ref{cokproperties}, the upper left entry of  (\refeq{blockdecomp}) can be written as $(T_{\Phi_0}, \pi_1)$ where $\pi_1$ is the $L^2$-orthogonal projection onto $\R \Phi_0$, and $T_{\Phi_0}$ is the composition:

    \begin{center}
\tikzset{node distance=3.2cm, auto}
\begin{tikzpicture}
\node(A){$H^2(\mathcal Z_0;N\mathcal Z_0)$};
\node(C)[right of=A]{$\text{\bf Ob}(\mathcal Z_0)^\perp$};
\node(C')[right of=C, yshift=0cm]{$L^{2}(\mathcal Z_0;\mathcal C_0 ), $};
\draw[->] (A) to node {$ \Pi_0^\perp\mathcal B_{\Phi_0}$} (C);
\draw[->] (C) to node {$ \text{ob}^{-1}$} (C');
\draw[->,>=stealth] (A) edge[bend right=20]node[below]{$T_{\Phi_0}$} (C');
\end{tikzpicture}\end{center}

 \noindent with $\Pi_0^\perp$ as in Definition \ref{Obdefinition}. In particular, $T_{\Phi_0}$ is a map of Hilbert spaces of sections of vector bundles on $\mathcal Z_0$.

 The main result of the current section is the following theorem, which is a more precise statement of Theorem \ref{maina} in the introduction.
 
 \begin{thm} 
\label{mainaprecise} The composition $T_{\Phi_0}$ is an elliptic pseudo-differential operator of order 1/2. In particular, as a map 
\smallskip 
\be T_{\Phi_0}: H^2(\mathcal Z_0;N\mathcal Z_0)\lre H^{3/2}(\mathcal Z_0;\mathcal C_{0})\label{TPhifredholmness}\ee it is Fredholm, and has index 0.
\smallskip   
\end{thm}

 \medskip

Using the block-diagonal decomposition (\refeq{blockdecomp}), Theorem \refeq{mainaprecise} and standard bootstrapping imply the following. Here, recall that $\text{\bf Ob}^{m}=\text{\bf Ob}(\mathcal Z_0)\cap H^m_\text{b}$. 

\begin{cor}\label{linearizedDinvertible} 
The linearized universal Dirac operator extends to a Fredholm of index 0 
 \be  \text{d}_{(\mathcal Z_0,\Phi_0)}\slashed{\mathbb D}: H^{m+2}(\mathcal Z_0;N\mathcal Z_0)\oplus  rH^{m,1}_{\text{b},e} \lre \text{\bf Ob}^{m+3/2} \oplus \Big(\text{range}(\slashed D) \cap  H^m_\text{b}\Big).
\ee

\noindent for every $m\geq 0$.\qed
\end{cor} \medskip 

 \begin{rem}\label{obconvention} The order of $T_{\Phi_0}$ depends on the choice of isomorphism $\text{ob}: L^2(\mathcal Z_0;\mathcal C_0)\to {\bf Ob}(\mathcal Z_0)$. For instance, one could just as easily have defined $\text{ob}$ as the composition of the current version with $(\Delta+1)^s$ for any $s\in \R$. Writing $\text{Id}=\text{ob}\circ \text{ob}^{-1}$, however, the obstruction component $\mathcal B_{\Phi_0}=\Pi^\perp_0\text{d}\slashed{\mathbb D}$ {\it which is independent of the choice of} $\text{ob}$ may be factored as

    \begin{center}
\tikzset{node distance=3.0cm, auto}
\begin{tikzpicture}
\node(A){$H^2(\mathcal Z_0;N\mathcal Z_0)$};
\node(C)[right of=A]{};
\node(C')[right of=C, yshift=0cm]{$ \text{\bf Ob}(\mathcal Z_0)^\perp \cap H^{3/2}_\text{b}(Y\setminus \mathcal Z_0) $};
\node(D)[below of=C, yshift=1.4cm]{$H^{3/2}(\mathcal Z_0;\mathcal C_0 ). $};
\draw[->] (A) to node {$ \Pi^\perp_0\mathcal B_{\Phi_0}$} (C');
\draw[->][swap] (A) to node {$T_{\Phi_0}$} (D);
\draw[->][swap] (D) to node {$ \text{ob}$} (C');
\end{tikzpicture}\end{center}
 \end{rem}

 \noindent In particular, the operator $\Pi_0^\perp \mathcal B_0$ has ``order'' 1/2, insofar as its image on $H^2(\mathcal Z_0)$ is $\text{\bf {Ob}}\cap H^{3/2}_\text{b}$ independent of the choice of $\text{ob}$, possibly up to a finite-dimensional subspace. Here order is used only loosely, as $\Pi_0\mathcal B_{\Phi_0}$ is not itself a pseudodifferential operator. Most importantly, the loss of regularity in Theorems \ref{maina} and \ref{mainb} is intrinsic to the geometric problem and cannot be avoided by simply revising conventions. 
 
  The conventions here are chosen so that $\text{ob}$ has order zero, i.e. so that the manifestation of $\Pi^\perp_0\mathcal B_{\Phi_0}$ as a true pseudodifferential operator --- this being $T_{\Phi_0}$ --- acts on spaces of the same regularity. Other authors may adopt the convention that $\text{ob}$ has order $-1/2$, which natural from the viewpoint of the Poisson operator as in (\refeq{Poissondef}).

\subsection{Conormal Regularity}
\label{section6.1}
The remainder of Section \ref{section6} proves Theorem \ref{mainaprecise}. Before beginning the proof in earnest, the current section studies the regularity of the projection operator $\Pi_0$.

The loss of regularity in Theorem \ref{mainaprecise} is a consequence of the fact that $\text{\bf Ob}$ does not simply inherit the obvious notion of regularity from $Y\setminus \mathcal Z_0$. Instead, one has
\smallskip

\begin{center}
{\it Key Observation:} The regularity of $\Pi_0(\psi) \in \text{\bf Ob}(\mathcal Z_0)$ depends on both the regularity of $\psi$ and $\text{  \ \  \hspace{1cm}\ \ \ its order of growth}$ along $\mathcal Z_0$.\hspace{5.8cm}.  
\end{center}  

\smallskip 
\noindent To elaborate, Proposition \ref{cokproperties2} shows that the regularity of $\Pi_0(\psi)$ is a question about the rate of decay in $|\ell|$ of the sequence of inner products 
\be \big \{ \ \br \psi, \Psi_\ell \kt_{\C} \  \big\}_{\ell \in \Z}.\label{sequenceofinnerproducts}\ee
Because the basis elements $\Psi_\ell$ concentrate exponentially around $\mathcal Z_0$ as $|\ell|\to \infty$, this rate of decay is intertwined with the growth of $\psi$ along $\mathcal Z_0$. If, for example, $\psi$ is compactly supported away from $\mathcal Z_0$, then Proposition \ref{cokproperties2} implies the sequence  (\refeq{sequenceofinnerproducts}) decays faster than polynomially and the projection $\text{ob}^{-1}\Pi_0(\psi)\in C^\infty(\mathcal Z_0;\mathcal C_0)$ is smooth regardless of the regularity of $\psi$ on $Y$. The rest of this subsection characterizes this phenomenon more precisely. Although the regularity of $\Pi_0\psi \in \text{\bf Ob}(\mathcal Z_0)$ is different from the ambient regularity on $Y\setminus \mathcal Z_0$, our convention (see Remark \ref{obconvention}) means that the regularity on either side of the isomorphism $\text{ob}$, i.e. of $\Pi_0\psi \in \text{\bf Ob}(\mathcal Z_0)$ and $\text{ob}^{-1}\Pi_0\psi \in L^2(\mathcal Z_0;\mathcal C_0)$ coincide by Lemma \ref{regularitycoincides}.

\begin{defn}Suppose that a spinor $\psi$ can be written locally in Fermi coordinates and an accompanying trivialization as \be \psi=\chi \begin{pmatrix} f^+(t) h^+(\theta) \\ f^-(t) h^-(\theta)\end{pmatrix}  r^p  \label{conormalform}\ee
where $f^\pm \in H^k(S^1;\C)$, $h^\pm$ are smooth, and $\chi$ is a cutoff function supported in a neighborhood $N_{r_0}(\mathcal Z_0)$. 
Then the quantity $$s=\boxed{ 1+ k +p} $$
is called the {\bf conormal regularity} of $\psi$. 
\end{defn}   

The following simple lemma gives the fundamental relationship between the conormal regularity and the regularity of the projection. In it, we denote by H the Hilbert transform as defined preceding Definition \ref{calderondef}.  

\begin{lm}
\label{conormalCaseA}
Suppose that $\psi \in L^2$ has conormal regularity $s$. Then $\text{ob}^{-1}\Pi_0(\psi)\in H^s(\mathcal Z_0;\mathcal C_{0})$
and $$\|\text{ob}^{-1}\Pi_0(\psi)\|_{s}\leq C_s( \|f^+\|_{H^{k}}+ \|f^-\|_{H^{k}}).$$
\end{lm}

\begin{proof} Using Proposition \ref{cokproperties2}, $\text{ob}^{-1}\Pi_0(\psi)$ is calculated by the sequence of inner products 
$$ \br \psi, \Psi_\ell\kt=\br \psi, \Psi_\ell^\circ + \zeta_\ell + \xi_\ell\kt \hspace{1cm}\text{ where }\hspace{1cm}\Psi_\ell^\circ =\chi \sqrt{|\ell|}e^{i\ell t} e^{-|\ell|r } \begin{pmatrix} \frac{1}{\sqrt{z}} \\  \frac{\text{sgn}(\ell)}{\sqrt{\overline z}} 
\end{pmatrix}.$$

\noindent Assume first that $g_0=dt^2 + dx^2 + dy^2$ on $N_{r_0}(\mathcal Z_0)$. Using the expression (\refeq{conormalform}) for $\psi$, the inner product with $\Psi_\ell^\circ$ yields 
\bea
\br \psi , \Psi_\ell^\circ \kt &=&  \Big  \br \chi \begin{pmatrix} f^+ h^+ \\ f^- h^- \end{pmatrix}r^{p}  \ , \ \sqrt{|\ell|} e^{i\ell t}\begin{pmatrix}\tfrac{e^{-|\ell| r}}{ \sqrt{z}}   \\ \text{sgn}(\ell)\tfrac{e^{-|\ell| r}}{\sqrt{\overline z}}\end{pmatrix} \Big \kt \\ &\leq &  \Big |\int_{S^1} \br  f^+ + Hf^-, e^{i\ell t}\kt \int_{\R^2}  \sqrt{|\ell|}  e^{- |\ell| r} r^{p-1/2} \chi(r) \|h^\pm\|_{C^0} r dr d\theta dt  \Big |  \\
&\leq& C \Big |\int_{S^1} \br  f^+ + Hf^-, e^{i\ell t}\kt dt \Big |  \  \int_{0}^\infty  \sqrt{|\ell|}  e^{- |\ell| r} r^{p+1/2} dr    \\
&\leq &C\big \br  \tfrac{1}{|\ell|^{p+1}}\left( f^+(t) + Hf^-(t))\right), e^{i\ell t}\big  \kt_{L^2(S^1;\C)} 
\eea
\noindent Since $f^\pm \in H^k(S^1;\C)$, then $(f^+(t) + H f^-(t))\in H^k(S^1;\C)$ as well, thus after applying the Fourier multiplier $1/{|\ell|^{p+1}}$ it lies in $H^{1+k+p}(S^1;\C)$ as desired. For the case of a general metric, the integrands differ by a factor of $1 + O(r)$ from the volume form and the latter contributes only a term of higher regularity.  

It is easy to show  that the contributions to the inner product arising from $\zeta_\ell + \xi_\ell$ satisfy the same bounds by invoking Corollary \ref{higherordercokernel} and integrating by parts. Since these terms are dealt with explicitly in the proof of Theorem \ref{mainaprecise}, the details are omitted here. 
\end{proof}

The following additional cases are a straightforward extension of the above.
 
 \begin{cor}\label{CaseBC} Let $\psi\in L^2(Y\setminus \mathcal Z_0; S_0)$ 
 
 \begin{itemize}
 \item[(B)] Suppose that $\text{supp}(\psi) \Subset Y\setminus \mathcal Z_0$. Then $\text{ob}^{-1}\Pi_0(\psi) \in H^s(\mathcal Z_0;\mathcal C_{0})$
for all $s>0$, and its $H^s$-norm is bounded by $C_s \|\psi\|_{L^2}$. 
 \item[(C)] Suppose $\psi$ has the form 

\be \psi=\begin{pmatrix} f^{+}(t) \ph^+(t,r,\theta) \\ f^-(t) \ph^-(t,r,\theta)\end{pmatrix}\label{conormalform2}\ee
\noindent where $f^\pm \in H^k(S^1;\C)$ and $\ph^{\pm}$ satisfy $|\ph^\pm| + |\nabla_t \ph^\pm| + \ldots + |\nabla_t^{k}\ph^\pm|< C(\ph)r^p$ pointwise. Then $\text{ob}^{-1}\Pi_0(\psi)\in H^{s}(\mathcal Z_0;\mathcal C_{0})$ for $s=1+k+p$, and its $H^{s}$-norm is bounded by $C_s  C(\ph)\|f^\pm\|_{H^k(S^1)}.$
 \end{itemize}

 \end{cor}
 
  \begin{rem} Before calculating the operator $T_{\Phi_0}$ explicitly, Corollary \ref{CaseBC} already implies that a loss of regularity is an inevitable consequence of the $\sqrt{r}$ asymptotics of $\Z_2$-harmonic spinors. Indeed, Corollary \ref{formofDD} shows that $\mathcal B_{\Phi_0}(\eta)$ schematically has the form $\eta'. \nabla \Phi_0 + \eta''. \Phi$. Since $\eta \in H^{2}$, and  $\Phi_0=O( r^{1/2})$ with  $\nabla \Phi_0=O(r^{-1/2})$, these terms have conormal regularity $s=1+1-1/2$ and $s=1+0+1/2$ respectively. It follows that $\Pi_0\mathcal B_0 \in H^{3/2}$ but $(1-\Pi_0)\mathcal B_0$ is in general no better than $L^2$. 
 \label{conormalrem}
 \end{rem}

\subsection{Obstruction Component of Deformations}
\label{section6.2}

This subsection proves Theorem \ref{mainaprecise}, except for the index statement, by calculating $T_{\Phi_0}$ explicitly.

The formula for $T_{\Phi_0}$ is expressed in terms of standard operators and the following zeroth order operator, for which we recall from Proposition \ref{transformationlem} that $c(t) \in N\mathcal Z_0^{-1}$ and $ d(t)\in N\mathcal Z_0$ denote the leading order (i.e. $r^{1/2}$) coefficients of $\Phi_0$. Define an operator \begin{eqnarray} \mathcal L_{\Phi_0}: L^2(\mathcal Z_0; N\mathcal Z_0)&\lre& L^2(\mathcal Z_0; \mathcal C_{0}) \label{mathcalTdef}\\ \mathcal \xi(t)&\mapsto& H(c(t) \xi(t)) - \overline \xi(t)d(t). \label{mathcalTdef}  \end{eqnarray}
\noindent where  $H$ is the Hilbert Transform as preceding Definition \ref{calderondef}. Recall here that $\mathcal C_0\simeq \C$ is canonically trivial, hence multiplication in the definition of $\mathcal L_0$ is the dual pairing $N\mathcal Z_0\otimes N\mathcal Z_0^{-1}\to \C$.

\begin{lm}
\label{mainapreciseformula}
For $\eta(t) \in H^2(\mathcal Z_0;N\mathcal Z_0)$, $T_{\Phi_0}$ as in Theorem \ref{mainaprecise} is given by 

\smallskip

\be
\label{formofTPhi0}
T_{\Phi_0}(\eta(t))= -\tfrac{3|\mathcal Z_0|}{2} (\Delta+1)^{-\tfrac34} \hspace{.07cm}\mathcal L_{\Phi_0}(\eta''(t)) \ + \ K(\eta)
\ee 

 \medskip
 
\noindent where $|\mathcal Z_0|$ denotes the length,  $\Delta$ denotes the positive-definite Laplacian on $\mathcal C_{0}$, $\mathcal L_{\Phi_0}$ is as in  (\refeq{mathcalTdef}) above, and $\eta''(t)$ denotes the (covariant) second derivative on $N\mathcal Z_0$. $K$ is a lower-order term.
 \end{lm}

\bigskip

\begin{rem}\label{frequencyfunction}
Lemma \ref{mainapreciseformula} shows that, up to composing with the appropriate power of $(\Delta+1)$, the symbol of $T_{\Phi_0}$ is given by $\mathcal L_{\Phi_0}$ as in (\refeq{mathcalTdef}, which is determined entirely by the leading coefficients of $\Phi_0$. Thus strict ellipticity of $T_{\Phi_0}$ is equivalent to non-degeneracy (Definition \ref{regulardef}), i.e. to the condition that the frequency function of $\Phi_0$ (in the sense of \cite{ZhangRectifiability,TaubesZeroLoci}) has order $1/2$ everywhere along $\mathcal Z_0$. Since Nash-Moser is already required, it seems likely that the non-degeneracy assumption could be weakened to consider e.g. the hypoelliptic case. 
\end{rem}

Lemma \ref{mainapreciseformula} is proved by calculating the sequence of inner products 

 \be \text{ob}^{-1}(\Pi_0^\perp \mathcal B_{\Phi_0}(\eta)) =\sum_{\ell} \br \mathcal B_{\Phi_0}(\eta) , \Psi_\ell  \kt_\C \  e^{i\ell t} \label{tocalculate}\ee
\noindent 
quite explicitly, where $\mathcal B_{\Phi_0}(\eta)$ is as in Corollary \ref {formofDD}. The proof consists of five steps: Steps 1--2 calculate (\refeq{tocalculate}) in the case that $g_0$ is locally the product metric and $\Phi_0$ is given  by its leading order term, and Steps 3--5 show that the small parade of error terms arising from higher order contributions result in a lower-order operator $K$.

\begin{proof}[Proof of Lemma \ref{mainapreciseformula}] 
Suppose, to begin, that all the structure are given locally by the Euclidean ones. That is, assume $$g_0 = dt^2 + dx^2 + dy^2 \hspace{1cm} \Phi_0 = \begin{pmatrix} c(t)\sqrt{z} \\ d(t)\sqrt{\overline z} 
\end{pmatrix}  \hspace{1cm} \dot g_\eta= \begin{pmatrix}  0 &  \eta_x' \chi & \eta'_y\chi \\  \eta'_x \chi& 2\eta_x \del_x \chi & \eta_x \del_y \chi + \eta_y \del_x \chi \\  \eta'_y\chi &   \eta_x \del_y \chi + \eta_y \del_x \chi  & 2\eta_y \del_y \chi \end{pmatrix},$$

\noindent and $B_0=0$; also assume that the obstruction elements of Proposition \ref{cokproperties} have $\zeta_\ell+\xi_\ell =0$ so that $$\Psi_\ell =  \chi \sqrt{|\ell|}e^{i\ell t}e^{-|\ell| r}\begin{pmatrix} \frac{1}{\sqrt{z}}  \\  \frac{\text{sgn}(\ell)}{\sqrt{\overline z}} 
\end{pmatrix}$$

\noindent {\it Step 1: product case, divergence term.} Let $e_i$ for $i=1,2,3$ denote an orthonormal frame for $g_0$ with $e^i$ the dual frame. Recall that for a symmetric 2-tensor $k$, $\text{div}_{g_0}k=(-\nabla_i k_{ij})e^j.$
\bea
\tfrac{1}{2}\text{div}_{g_0}(\dot g_\eta).\Phi_0& =& -\frac{1}{2}\left[ \sigma_2 \chi \eta_x '' + \sigma_3 \chi \eta_y''\right] \begin{pmatrix}c(t)\sqrt{z} \\ d(t)\sqrt{\overline z}\end{pmatrix} + {\bf (I)}\\
            &=& -\frac{1}{2} \left[ \chi \eta''_x\begin{pmatrix}-d(t)\sqrt{\overline z} \\ c(t)\sqrt{z}\end{pmatrix}+ \chi \eta''_y\begin{pmatrix}id(t)/\sqrt{\overline  z} \\ ic(t)/\sqrt{ z}\end{pmatrix} \right] + {\bf (I)}\\
            & =& -\frac{1}{2} \left[ \begin{pmatrix} -\overline \eta'' d(t) \chi \sqrt{\overline z } \\  \ \eta'' c(t) \chi \sqrt{z}\end{pmatrix}\right]+ {\bf (I)}
\eea
\noindent where we have written $\eta(t)= \eta_x(t) + i \eta_y(t)$, and 
$${\bf (I)}=  -\frac{1}{2} \Big[(\del_x \chi \eta'_x + \del_y \chi \eta'_y)\sigma_t + ( 2 \del_{xx} \chi \eta_x + \del_{xy}\chi \eta_y + \del_{yy}\chi \eta_x) \sigma_x \newline  +(\del_{xy}\chi \eta_y + \del_{yy}\chi \eta_x + 2\del_{yy}\chi \eta_y) \sigma_y\Big]  .\Phi_0.$$

Taking the inner product of the first term with $\Psi_\ell$ yields

\bea
\br\tfrac{1}{2}\text{div}_{g_0}(\dot g_\eta).\Phi_0 , \Psi_\ell \kt&=& -\frac{1}{2}  \Big \langle \chi \begin{pmatrix} -\overline \eta'' d(t) \sqrt{\overline z}  \\  \eta'' c(t) \sqrt{z}  \end{pmatrix}  \ , \ \sqrt{|\ell|}e^{i\ell t} \chi \begin{pmatrix} {e^{-|\ell|r}}/{\sqrt{z}}\\ \text{sgn}(\ell) {e^{-|\ell|r}}/{\sqrt{\overline z}}  \end{pmatrix} \Big \rangle_\C \ + \ \br ({\bf I}), \Psi_\ell\kt_\C \\
&=&-\frac{1}{2}\int_{S^1}\Big \langle  \begin{matrix}-\overline \eta'' d(t) \\ \eta'' c(t) \end{matrix} \ , \ \begin{matrix} e^{i\ell t}\\ \text{sgn}(\ell)e^{i\ell t}  \end{matrix} \Big \rangle_\C  \ dt  \ \ \int_{\R^2}\sqrt{|\ell|} \chi^2 e^{-|\ell| r} r dr d\theta \ + \ \br ({\bf I}), \Psi_\ell\kt_\C   \\
&=& -\frac{1}{2} \br \text{sgn}(\ell) \eta'' c  - \overline \eta'' d  \ , \ e^{i\ell t} \kt_{L^{2}(\mathcal Z_0)} \ \int_{\R^2}\sqrt{|\ell|} \chi^2(r) e^{-|\ell| r} r dr d\theta \ + \ \br ({\bf I}), \Psi_\ell\kt_\C \\
&=& \br  -\frac{1}{2} \frac{|\mathcal Z_0|}{|\ell|^{3/2}}  \mathcal L_{\Phi_0}(\eta''), e^{i \ell t} \kt_\C \ + \ \br K , e^{i\ell t}\kt_\C
\eea

\noindent where $K$ is as follows. First, note $$ \int_{0}^\infty\sqrt{|\ell|} e^{-|\ell| r} r dr d\theta=\frac{1}{|\ell|^{3/2}} $$
and the presence of $\chi^2(r)$ results in a difference from this of size $O(e^{-|\ell | r_0})$, let this remainder be the first part of $K$. 

Then, since $$\frac{1}{|\ell|^{3/2}}  = \frac{1}{(|\ell|^2+1)^{3/4}}+ O \left(\frac{1}{|\ell|^{3}}\right),$$ we can write $$\text{ob}^{-1}(\tfrac{1}{2}\text{div}_{g_0}(\dot g_\eta).\Phi_0) = -\tfrac{|\mathcal Z_0|}{2}(\Delta+1)^{-\tfrac{3}{4}} \mathcal L_{\Phi_0}(\eta'') + K$$ 
\noindent where the lower order psuedo-differential operator from $O(|\ell|^{-3})$ is absorbed into $K$. Finally, the term $\text{\bf (I)}$ is a sum of terms compactly supported away from $\mathcal Z_0$, hence by Case (B) of Corollary \ref{CaseBC}, it contributes a smoothing operator which we may likewise absorb into $K$.

\medskip 

\noindent {\it Step 2: product case, symbol term.} The ``symbol'' term from $\mathcal B_{\Phi_0}(\eta)$ is given by 
\bea
-\frac{1}{2}\dot g_\eta(e_i, e_j)e^i. \nabla_j\Phi_0& =& -\frac{1}{2}\left[\chi \eta_x' \sigma_t  \nabla_x \Phi_0  + \chi \eta_y' \sigma_t \nabla_y \Phi_0\right]   + {\bf (II)}\\ 
  &=&-\frac{1}{4} \left[ \chi \eta'_x\begin{pmatrix}ic(t)/\sqrt{z} \\ -id(t)/\sqrt{\overline z}\end{pmatrix}+ \chi \eta'_y\begin{pmatrix}-c(t)\sqrt{z} \\ d(t)\sqrt{\overline z}\end{pmatrix} \right] + {\bf (II)}\\
&=&  -\frac{1}{4} \left[ \begin{pmatrix} i\eta' c(t)\chi /\sqrt{z} \\ -i\overline \eta' d(t)\chi /\sqrt{\overline z}\end{pmatrix}\right] + {\bf (II)}
  \eea

\noindent where 
\bea
{\bf (II)}&=& -\frac{1}{2}\Big[ (\chi \eta'_x \sigma_x + \chi \eta'_y \sigma_y)\nabla_t \Phi_0
  +(2\del_x \chi \eta_x \sigma_x + \del_x\chi \eta_y \sigma_y + \del_y \chi \eta_x \sigma_y )\nabla_x \Phi_0 \\ & & \ \ \  + \ (2\del_y \chi \eta_y \sigma_y + \del_x\chi \eta_y \sigma_x + \del_y \chi \eta_x \sigma_x)\nabla_y \Phi_0  \Big].
\eea

Taking the inner product of the first term with $\Psi_\ell$ yields the following. This calculation is almost identical to the previous one, but with an additional integration by parts. 
\bea
\br\tfrac{1}{2}\dot g_\eta(e_i, e_j)e^i. \nabla_j\Phi_0 , \Psi_\ell \kt&=& -\frac{1}{4}  \Big \langle \chi \begin{pmatrix} i \eta' c(t)/\sqrt{z}  \\   -i \overline \eta' d(t)/\sqrt{\overline z} \end{pmatrix}  \ , \ \sqrt{|\ell|}e^{i\ell t} \chi \begin{pmatrix} {e^{-|\ell|r}}/{\sqrt{z}}\\ \text{sgn}(\ell) {e^{-|\ell|r}}/{\sqrt{\overline z}}  \end{pmatrix} \Big \rangle_\C \\
&=&-\frac{1}{4}  \Big \langle \chi \begin{pmatrix} i \eta' c(t)/\sqrt{z}  \\   -i \overline \eta' d(t)/\sqrt{\overline z} \end{pmatrix}  \ , \ \tfrac{\sqrt{|\ell|}}{i |\ell| \text{sgn}\ell} \del_t e^{i\ell t} \chi \begin{pmatrix} {e^{-|\ell|r}}/{\sqrt{z}}\\ \text{sgn}(\ell) {e^{-|\ell|r}}/{\sqrt{\overline z}}  \end{pmatrix} \Big \rangle_\C \\ 
&= & -\frac{1}{4}  \Big \langle \chi \del_t \begin{pmatrix}  \eta' c(t)/\sqrt{z}  \\    -\overline \eta' d(t)/\sqrt{\overline z} \end{pmatrix}  \ , \ \tfrac{1}{\sqrt{|\ell|}} e^{i\ell t} \chi \begin{pmatrix} \text{sgn}(\ell){e^{-|\ell|r}}/{\sqrt{z}}\\  {e^{-|\ell|r}}/{\sqrt{\overline z}}  \end{pmatrix} \Big \rangle_\C \\ 
&=&-\frac{1}{4}\int_{S^1}\Big \langle  \begin{matrix}\del_t(\eta' c(t)) \\ -\del_t (\overline \eta' d(t)) \\  \end{matrix} \ , \ \begin{matrix} \text{sgn}(\ell) e^{i\ell t}\\ e^{i\ell t}  \end{matrix} \Big \rangle_\C  \ dt  \ \ \int_{\R^2}\frac{1}{\sqrt{|\ell|} }\chi^2 e^{-|\ell| r}  dr d\theta  
\eea

\noindent In the second line we have multiplied the second argument by $1$  in the form $1=\tfrac{i\ell}{i |\ell| \text{sgn}\ell}$ and noted $i\ell \Psi_\ell = \del_t \Psi_\ell$, and then integrated by parts. Then,

\bea 
\phantom{\br\tfrac{1}{2}\dot g_\eta(e_i, e_j)e^i. \nabla_j\Phi_0 , \psi_\ell \kt}&=& -\frac{1}{4} \br \text{sgn}(\ell) \eta'' c  - \overline \eta'' d  \ , \ e^{i\ell t} \kt_{L^{2}(S^1)} \ \int_{\R^2}\frac{1}{\sqrt{|\ell|}} \chi^2(r) e^{-|\ell| r} r dr d\theta\\
& &   \ \ \  -\frac{1}{4} \br \text{sgn}(\ell) \eta'c'  - \overline \eta' d'  \ , \ e^{i\ell t} \kt_{L^{2}(S^1)} \ \int_{\R^2}\frac{1}{\sqrt{|\ell|}} \chi^2(r) e^{-|\ell| r} r dr d\theta\\ 
&=& \br  -\frac{1}{4} \frac{|\mathcal Z_0| }{|\ell|^{3/2}}  \mathcal L_{\Phi_0}(\eta''), e^{i \ell t} \kt_\C + \br  -\frac{1}{4}\frac{|\mathcal Z_0|}{|\ell|^{3/2}} \mathcal L_{\nabla_t \Phi_0} (\eta'),e^{i\ell t}\kt_\C+  \br  K , e^{i \ell t}\kt 
\eea

\noindent Where $K$ is again an error of size $O(e^{-|\ell| r_0})$ and $\mathcal L_{\nabla_t \Phi_0}$ is defined exactly as $\mathcal L_{\Phi_0}$ but with $c'(t), d'(t)$ in place of $c(t), d(t)$. Both $\mathcal L_{\nabla_t \Phi_0}$ and the term {\bf (II)} are lower order by  Lemma \ref{conormalCaseA} and Case (B) of Corollary \ref{CaseBC}, so they may be absorbed into $K$. To see this, note both of these are comprised of terms of the form form $\eta' \nabla_t \Phi_0= \eta' r^{1/2}$, hence of conormal regularity $s=5/2$ or have a factor of $d\chi$ so are compactly supported away from $\mathcal Z_0$. The term same applies to the term $\tfrac{1}{2} d \text{Tr}_{g_0}(\dot g_\eta).\Phi_0$, which we likewise absorb into $K$.

\begin{rem}
It appears that a coincidence has occurred in Steps 1--2: Lemma \ref{conormalCaseA} implies that the two leading order terms from {\it Step 1} and {\it Step 2} are both order $1/2$ as they have the same conormal regularity. The calculation shows they are actually {\it the same} up to a constant multiple and lower order terms. Steps 1--2 can be calculated in other ways, where this coincidence is related to the fact that $\Phi_0, \Psi_\ell$ being harmonic implies the stress-energy tensor is divergence-free. 
\end{rem}

 We now return to the general case. 

\medskip 

\noindent {\it Step 3:} By Proposition  \ref{asymptoticexpansion}, $\Phi_0$ can in general be written as $$\Phi_0 =\begin{pmatrix}c(t)\sqrt{z} \\ d(t)\sqrt{\overline z}\end{pmatrix} + \Phi_1$$
where the higher order terms satisfy
 \be |\Phi_1| + |\nabla_t^k\Phi_1| \leq C_k r^{3/2} \hspace{2cm}|\nabla_z\Phi_1| + |\nabla_t^k(\nabla_z\Phi_1)|\leq C_k r^{1/2}\label{Or^3/2}\ee
 
 \noindent for any $k\in \N$ and identically for $\nabla_{\overline z}$. The resulting contribution to $\mathcal B_{\Phi_0}(\eta)$ is 
 
 \begin{equation}-\frac{1}{2}\dot g_\eta(e_i, e_j)e^i.\nabla_j \Phi_1 + \frac{1}{2} d \text{Tr}_{g_0}(\dot g_\eta). \Phi_1+\frac{1}{2} \text{div}_{g_0}(\dot g_\eta).\Phi_1 \label{higherordervariation}\end{equation} 
 \noindent and using (\refeq{Or^3/2}) and Part (C) of Corollary \ref{CaseBC} shows that each term has conormal regularity one higher than the corresponding term for the leading order of $\Phi_0$. (\refeq{higherordervariation}) therefore contributes an operator of order $-1/2$ which can be absorbed into $K$. 
 
 \medskip 
 
 \noindent {\it Step 4:} As in Definition \ref{fermi}, the metric in Fermi coordinates in general has the form $$g_0=dt^2 + dx^2 + dy^2 + h$$ \noindent where $h=O(r)$. Compared to the case of the product metric, we now have $e_i= \del_i + O(r)$ and 
  \begin{eqnarray}\nabla_i^{g_0}=\del_i + \Gamma_i.& & \hspace{2cm} \dot g_\eta= \dot g_\eta^\text{prod} + h_1 + h_2  \\ dV_{g_0}= (1+O(r))rdrd\theta dt & & \hspace{2cm} \br -, -\kt_{g_0}= (1+O(r)) \br -,-\kt_{\text{Euc}}.  \label{3x3matrix} \end{eqnarray}  
 
 \noindent where $h_1, h_2$ are as in Corollary \ref{formofDD}. As such, each additional term in $\mathcal B_{\Phi_0}(\eta)$ has {\it either} an additional power of $r$ {\it or} one fewer derivative of $\eta$ compared to the terms for the product case. Using Corollary \ref{CaseBC} and the bounds $$|\Phi_0| + |\nabla_t^k\Phi_0|\leq C_kr^{1/2} \hspace{2cm} |\nabla_z \Phi_0|+ |\nabla_t^k(\nabla_z \Phi_0)|\leq C_kr^{-1/2}$$
 
\noindent we see that all such terms have conormal regularity at least $s=5/2$. The term $\mathcal R(B_0,\chi\eta)=O(1)\eta$ arising in the case that $B_0\neq 0$ likewise has conormal regularity $s>5/2$. In addition, changing $\tfrac{d}{dt}$ to the covariant derivative only contributes to the lower order term $K$.  

\medskip 

 \noindent {\it Step 5:} By Proposition \ref{cokproperties2} we may in general write $$ \Psi_\ell = \chi \Psi_\ell^\circ + \zeta_\ell^{(m)}+ \xi_\ell^{(m)}$$
 \noindent where the latter satisfy the bounds of Corollary \ref{higherordercokernel}.  Set \be  K_1(\eta):=\sum_{\ell} \br \mathcal B_{\Phi_0}(\eta), \zeta^{(m)}_\ell\kt e^{i\ell t} \hspace{2cm}K_2(\eta):=\sum_{\ell} \br \mathcal B_{\Phi_0}(\eta), \xi^{(m)}_\ell \kt e^{i\ell t}. \ee
 
 We claim that the second factors through the map $K_2: H^2\to H^{5/2}\to H^{3/2}$ hence contributes a compact term. By Cauchy-Schwartz and the bound $\|\xi_\ell^{(m)}\|_{L^2}\leq C_m|\ell|^{-2-m}$ from Corollary Corollary \ref{higherordercokernel},  
 
  \bea \|K_2(\eta)\|^2_{H^{5/2}} &=& \sum_{\ell} |\br \mathcal B_{\Phi_0}(\eta), \xi^{(m)}_\ell \kt|^2 |\ell|^{5} \\  &\leq& \sum_{\ell} \|\mathcal B_{\Phi_0}(\eta)\|^2_{L^2} \  \|\xi^{(m)}_\ell\|^2_{L^2} \  |\ell|^{5} \\ &\leq &C \|\mathcal B_{\Phi_0}(\eta) \|^2_{L^2}\sum_{\ell} \frac{|\ell|^{5}}{|\ell|^{4+2m}}  
  \ \ \ \  \leq  \   C \|\eta\|_{H^{2}}^2   \sum_{\ell}\frac{1}{|\ell|^{2m-1}}\leq C \|\eta\|^2_{H^2}  \eea
\noindent for, say, $m=2$. In the last line we have used that $|\mathcal B_{\Phi_0}(\eta)|\leq (|\eta|+ |\eta'|+|\eta''|)r^{-1/2}$ and the latter is integrable on normal disks. 

Likewise, we claim $K_1$ factors through the inclusion $H^{3/2+\delta}\hookrightarrow H^{3/2}$ for $\delta<1/2$. This time, we apply Cauchy-Schwartz on each annulus $A_{n\ell}$ (defined in \ref{cokproperties2}). Write $K_1= K_1' + K_1''$ where $$K_1'(\eta)=  \br  \tfrac{1}{2} d \text{Tr}_{g_0}(\dot g_\eta). \Phi_0 +\tfrac{1}{2} \text{div}_{g_0}(\dot g_\eta).\Phi_0, \zeta_\ell. \kt \hspace{2cm} K_1''(\eta)=\br -\tfrac{1}{2}\dot g_{\eta}(e_i, e_j)e^i. \nabla_j \Phi_0, \zeta_\ell \kt $$
\noindent and we keep the superscript $(m)$ implicit. For the first of these,  

\bea \|K_1'(\eta)\|^2_{H^{3/2+\delta}} &\leq&C \sum_{\ell}\sum_n \| \eta''|\Phi_0|\|^2_{L^2(A_{n\ell})} \  \|\zeta_\ell\|_{L^2(A_{n\ell})}^2 \  |\ell|^{3+2\delta} \\  &\leq&C \sum_{\ell} |\ell|^{3+2\delta}\sum_{n} \|\eta'' r^{1/2}\|^2_{L^2(A_{n\ell})} \  \frac{1}{|\ell|^2}\text{Exp}\left(-\frac{n}{c_1}\right)
\eea
Then, since $r\sim \tfrac{(n+1)R_0}{|\ell|}$ on $A_{n\ell}$, and each has area $O(|\ell|^{-2})$,  the above is bounded by 
\bea &\leq&C\|\eta''\|^2_{L^2(S^1)} \sum_{\ell} |\ell|^{3+2\delta}\sum_{n}  \frac{(n+1)^3}{|\ell|^5}\text{Exp}\left(-\frac{n}{c_1}\right)
  \ \ \leq  \ \  C\|\eta''\|^2_{L^2(S^1)} \sum_{\ell} \frac{1}{|\ell|^{2-2\delta}} \ \leq  \ C\|\eta\|^2_{H^2}.
\eea

\noindent The $K_1''$ term is the same except we first use the Fourier mode restriction that $\zeta_\ell$ has only Fourier modes $p$ with $\ell-\tfrac{|\ell|}{2}\leq p \leq \ell +\tfrac{|\ell|}{2}$ to write $1\sim \tfrac{i\del_t}{|\ell|}$ and then integrate by parts as in {\it Step 2}. 

\end{proof} 

\subsection{The Index of $\mathcal L_{\Phi_0}$}
\label{section6.3}
This section completes the proof of Theorem \ref{mainaprecise} by showing $T_{\Phi_0}$ has Fredholm index 0. The key role and Fredholmness of a similar map was originally observed in \cite{RyosukeThesis}. Here, we present a simplified proof.

\begin{lm} \label{TPhi_0lem}When non-degeneracy as in Definition \ref{regulardef} holds, $$\mathcal L_{\Phi_0}: L^2(\mathcal Z_0;N\mathcal Z_0)\to L^2(\mathcal Z_0; \mathcal C_{0})$$
is an elliptic pseudo-differential operator of index 0. 
\end{lm}

To begin, we have the following fact. Let $a(t)\in C^\infty(\mathcal Z_0;\C)$ be a smooth function and let $$[H,a]=H\circ a(t)-a(t)\circ H$$ denote the commutator. 

\begin{claim}
\label{commutator}
The commutator $$[H,a]: H^{m}(\mathcal Z_0;\C)\to H^{m+1}(\mathcal Z_0;\C)$$
is a smoothing operator of order 1. 
\end{claim}  

\begin{proof}Multiplication by $a(t)$ and $H$ are both elliptic pseudodifferential operators of order 0 (with lower order terms of integer order), hence so is the commutator. Using the composition property of principal symbols, its principal symbol of order 0 is $$\sigma_0({[H,a]})= \sigma_0(H) \sigma_0({a}) - \sigma_0(a)\sigma_0(H)=0$$
hence it is a pseudodifferential operator of order $-1$. 
\end{proof}

We now prove the lemma: 

\begin{proof}[Proof of Lemma \ref{TPhi_0lem}] Given $\xi \in L^2(\mathcal Z_0; \mathcal C_{0})\simeq L^2(\mathcal Z_0; \C)$ we define a pseudo-inverse. Set \be \mathcal L^\star_{\Phi_0}(\xi(t))=\overline c(t)H\xi(t) - d(t) \overline{\xi(t)}.\ee
\noindent Using Claim \ref{commutator} to move $H$ past combinations of the smooth functions $c(t), d(t)$ and their conjugates (and noting $H^2=\text{Id}$), we obtain   
 
     \bea
    \mathcal L_{\Phi_0}\circ \mathcal L^\star_{\Phi_0}(\xi(t))&=&((Hc(t) - d(t)\circ  \text{conj}))(\overline c(t)H - d(t)\circ \text{conj})(\xi(t)))\\
    &=& H c \overline c H \xi -d c \overline {H\xi} - Hc d \overline \xi + d \overline d \xi \\
    &=& (|c|^2 + |d|^2)\xi + [H, |c|^2] H \xi - dc (\overline {H\xi}+ H\overline \xi) - [H, cd]\overline \xi \\
    &=&  ((|c|^2 + |d|^2)\text{Id} + K )\xi 
    \eea
    
    \noindent for a smoothing operator $K$ of order $\leq -1$. In the last line we have used $\overline {H\xi}+ H \overline \xi=2\xi_0$ where $\xi_0$ is the zeroeth Fourier mode, which is clearly a smoothing operator. It follows that $$\frac{1}{|c|^2 + |d|^2}\mathcal L^\star_{\Phi_0}$$ provides a right pseudo-inverse for $\mathcal L_{\Phi_0}$ (commuting the scaling factor and $H$ only contributes to the compact term). An equivalent calculation for the reverse composition shows it is also a left pseudo-inverse, thus $\mathcal L_{\Phi_0}$ is Fredholm. 
    
    A fixed choice of Fermi coordinates induces an isomorphism $N\mathcal Z_0\simeq \C$. Since $\pi_1(\C^2-\{0\}, *)$ is trivial, the pair $(c(t), d(t))$ is homotopic through pairs satisfying the condition $|c(t)|^2 + |d(t)|^2>0$ to the constant pair $(1,0)$. The operator $\mathcal L_{\Phi_0}$ is therefore homotopic to the identity through Fredholm operators hence has index 0. 
\end{proof}

Theorem \ref{mainaprecise} is now immediate: 

\begin{proof}[Proof of Theorem \ref{mainaprecise}] Lemma \ref{mainapreciseformula} shows that the operator $$\text{ob}^{-1}(\Pi_0^\perp B_{\Phi_0}(\eta))= -\frac{3|\mathcal Z_0|}{2}(\Delta +1)^{-\frac{3}{4}}\mathcal L_{\Phi_0}(\eta''(t))+K$$ is given as the sum of following compositions:     
    \begin{center}
\tikzset{node distance=4.0cm, auto}
\begin{tikzpicture}
\node(A){$H^{2}(\mathcal Z_0;N\mathcal Z_0)$};
\node(C)[right of=A]{$H^{3/2+\delta}(\mathcal Z_0;\mathcal C_0 )$};
\node(C')[above of=A, yshift=-1.5cm]{$H^{2}(\mathcal Z_0;N\mathcal Z_0)$};
\node(D')[right of=C']{$L^{2}(\mathcal Z_0;N\mathcal Z_0)$};
\node(E')[right of=D']{$L^{2}(\mathcal Z_0;\mathcal C_{0})$};
\node(F')[right of=E']{$H^{3/2}(\mathcal Z_0;\mathcal C_0)$};
\draw[->] (A) to node {$ K$} (C);
\draw[->] (C') to node {$-\tfrac{3|\mathcal Z_0|}{2}\d{}{t^2}$} (D');
\draw[->] (D') to node {$\mathcal L_{\Phi_0}$} (E');
\draw[->] (E') to node {$(\Delta+1)^{-\frac{3}{4}}$} (F');
\draw[<->] (A) to node {$\simeq $} (C');
\draw[->] (C) to node {$\iota $} (F');
\end{tikzpicture}\end{center}
 \noindent where the diagonal arrow is the inclusion, hence compact. All the top arrows are Fredholm of index 0 using Lemma  \ref{TPhi_0lem}; the conclusion therefore follows from the composition law for pseudodifferential operators.

\end{proof}

Given Theorem \ref{mainaprecise}, we now impose one more tacit assumption that this Fredholm operator of index zero is actually invertible. This is expected to hold generically (see \cite{SiqiZ3}), though we do not prove such a result here. At the end of Section \ref{section8}, this assumption can be removed by the use of standard Kuranishi methods.

\begin{assumption}
    
\label{assumption5}
The index zero operator $T_{\Phi_0}: H^{2}(\mathcal Z_0;N\mathcal Z_0) \to H^{3/2}(\mathcal Z_0, \mathcal C_0)$ is an isomorphism. 
\end{assumption}

\bigskip 

\section{Nash-Moser Theory}
\label{section7}

As explained in introduction, deducing the non-linear deformation result (Theorem \ref{mainb}) from the linear one (Theorem \ref{maina}, Theorem \ref{mainaprecise}) requires the Nash-Moser Implicit Function Theorem because of the loss of regularity in the operator $T_{\Phi_0}$. This section gives a brief and practical introduction to the framework of Nash-Moser Theory and states the relevant version of the implicit function theorem. The most complete reference for the full abstract theory is \cite{HamiltonNashMoser}. Here, we more closely follow the expositions in \cite{SecchiNashMoser, StraymondNashMoser, PsiDONashMoser} which are more modest in scope but suffice for our purposes.

\subsection{Tame Fr\'echet  Spaces}
\label{section9.1}
Let $\mathcal X,\mathcal Y$ be Fr\'echet spaces given as the intersection of families of Banach spaces 
\be \mathcal X:= \bigcap_{m\geq 0} X_m \hspace{2cm} \mathcal Y:= \bigcap_{m\geq 0} Y_m\label{frechetnotation}\ee
\noindent whose norms are monotonically non-decreasing so that $$\|x\|_0\leq \|x\|_{1}\leq\ldots $$ and likewise for $\mathcal Y $. The topologies on $\mathcal X,\mathcal Y$ are the ones generated by the countable collection of norms, i.e. a set $U$ is open if and only if for each point $x\in U$ there are $r>0$ and $m\geq 0$ such that the ball $\{ x \ | \ \|x\|_m < r \}\subset U$ measured in the $m$-norm is contained in $U$.

\begin{defn} \label{tamedef}A Fr\'echet  space $\mathcal  X$ is said to be {\bf tame} if it satisfies the two following additional criteria: 

\begin{enumerate}
\item[(I)] For all $m_1<m< m_2$ there are constants $C_{m,m_1,m_2}$ such that the interpolation inequalities
 $$\|x\|_{m}\leq C_{m,m_1, m_2} \|x\|_{m_1}^\alpha\|x\|_{m+2}^{1-\alpha}$$

\noindent hold where $\alpha=\tfrac{m_2-m}{m_2-m_1}$.
\item[(II)] $\mathcal X$ is equipped with a family of smoothing operators $$S_\e: \mathcal X\to \mathcal X$$ for all $\e \in (0,1]$ satisfying the following conditions. 

\begin{enumerate}
\item[(i)] $\|S_\e x\|_{n}\leq C_{mn} \e^{m-n}\|x\|_m $ for $n\geq m$ \ \ \  and  \ \ \  $\|S_\e x\|_m \leq C_{mn}\|x\|_n$ for $n\leq m$. 
\item[(ii)] $\|S_\e x -x\|_{m}\leq C_{mn} \e^{n-m} \|x\|_n$ for $n\geq m.$ 
\item[(iii)]$ \|\d{}{\e} S_\e x\|_n\leq C_{mn} \e^{m-n-1} \|x\|_{m}$ for all $m,n\geq 0.$ 
\end{enumerate}
\end{enumerate}

\end{defn}
\bigskip 

In practice, most reasonable choices of families of norms coming from Sobolev or H\"older norms are tame. Roughly speaking, smoothing operators $S_\e$ are usually constructed by truncating local Fourier transforms at radius $\e^{-1}$ in Fourier space. The Fr\'echet spaces used in the proof of Theorem \ref{mainb} are introduced in Section \ref{section8.35}, and their smoothing operators are constructed in Appendix \ref{appendixbdreg}

\bigskip

Given two tame Fr\'chet spaces $\mathcal  X$ and $\mathcal Y$, 
\begin{defn} A {\bf tame Fr\'echet map} on an open subset $U\subseteq \mathcal X$ 
$$
\mathcal F: U\to \mathcal  Y
$$
 is a smooth map of vector spaces such that for some $r\in \N$ the estimate 
\begin{equation}\label{7.firsttameinequality}
\|\mathcal F(x)\|_{m}\leq C_m  \ ( 1 + \|x\|_{m+r})
\end{equation}
 holds for all sufficiently large  $m$. 
\end{defn}

The definitions of tame spaces and maps extend naturally to define a category of tame Fr\'echet  manifolds with tame Fr\'echet  maps between them (see \cite{HamiltonNashMoser} for details). 
The key point about tame estimates is that each norm depends only on a fixed finite number $r$ of norms larger than it.  Thus, for example, a map with an estimate of the form \eqref{7.firsttameinequality} where $r=2m$ would not be tame. \\

\subsection{The Implicit Function Theorem}
\label{section7.2}

Before stating a precise version of the Nash-Moser Implicit Function Theorem, let us briefly give some intuition. Here, our exposition follows \cite{SecchiNashMoser}.  

Suppose that $\mathcal F: \mathcal  X\to \mathcal  Y$ is a map with $\mathcal F(0)=0$, and we wish to solve \be \mathcal F(x)=f\label{NMtosolve}\ee

\noindent for $f \in \mathcal Y$ small. When $ \mathcal  X$ and $ \mathcal  Y$ are Banach spaces, the (standard)  Implicit Function Theorem is proved using Newton iteration and 
 the Banach  Fixed Point Theorem. More specifically, one begins with an initial approximation $x_0=0$, and (provided that $\text{d}_{x_0}\mathcal F$ is invertible)  defines \be x_{k+1}=x_k + (\text{d}_{x_0}\mathcal F)^{-1}(f-\mathcal F(x_k))\label{baditeration}.\ee

\noindent The sequence $x_k\to x_\infty$ then converges to a unique fixed point solving equation (\refeq{NMtosolve}) for $f\in\mathcal Y$ sufficiently small. Alternatively, one can modify the iteration step \eqref{baditeration} by inverting $\text{d}\mathcal F$ at $x_k$ instead of at $x_0$, taking
\be x_{k+1}=x_k + (\text{d}_{x_k}\mathcal F)^{-1}(f-\mathcal F(x_k))\label{improvediteration}.\ee

\noindent This  iteration scheme has a much faster rate of convergence: like $\sim 2^{-2^k}$. 

Consider now the case of $\mathcal  X, \mathcal  Y$ tame Fr\'echet  spaces when $\text{d}\mathcal F$ displays a loss of regularity of $r$. Given an initial bound on $f\in Y_m$, then $x_1$ is bounded only in $X_{m-r}$ thus $f-\mathcal F(x_1)\in Y_{m-r}$ and $x_2\in X_{m-2r}$ and so forth. In this way, the standard Newton iteration scheme will exhaust the prescribed regularity in a finite number of steps. To circumvent this  loss of regularity, Nash introduced iteration employing smoothing operators at each stage. More precisely, for some $\e_k\in (0,1]$, we set    

\be x_{k+1}=x_k + (\text{d}_{S_{\e_k}(x_k)}\mathcal F)^{-1} S_{\e_k}(f-\mathcal F(x_k))\label{NMiteration},\ee

\noindent where the smoothing operators are those on $\mathcal  X$ and $\mathcal Y$ respectively when applied to $x_k$ and $f-\mathcal F(x_k)$. The key point is that the rate of convergence is rapid enough to overcome the disruption of the smoothing operators, but only if we use this smoothing to modify the improved iteration (\refeq{improvediteration}), rather than the original iteration 
\eqref{baditeration}. Thus, unlike to the Implicit Function Theorem on Banach spaces, the Nash-Moser Implicit Function Theorem requires the linearization be invertible on a neighborhood of the initial guess,  and requires bounds on the second derivatives to control the linearization over this neighborhood. Specifically, the theorem requires the following hypotheses  on a tame map $\mathcal F: U\to \mathcal  Y$:

\begin{hyp1*}
\label{NM1}
There exists a $\delta_0>0$ and an $m_0\geq 0$ such that for  $x\in U_0=B_{\delta_0 }(0, m_0) \cap \mathcal X$, the open ball of radius $\delta_0 $ measured in the $m_0$ norm, then $$\text{d}_x\mathcal F: \mathcal  X\to \mathcal Y$$ is invertible.  
\end{hyp1*}

\begin{hyp2*}
\label{NM2}
  With $x\in U_0$ as above, there are fixed $s, s' \in \N$ such that the unique solution $u$ of $$\text{d}_x \mathcal F(u)=f$$
satisfies the tame estimate
\be  \|u\|_{m}\leq C_m \ \Big( \|f\|_{m+s} +\|f\|_{m_0}\cdot  \|x\|_{m+s'}\Big).
\ee
\end{hyp2*}

\begin{hyp3*}\label{NM3}
With $x\in U_0$ as above, there are fixed $r, r' \in \N$ such that the second derivative satisfies the tame estimate
\be  \|\text{d}^2_x\mathcal F(u, \upsilon)\|_{m}\leq C_m \ \Big( \|u\|_{m+r}\|\upsilon\|_{m_0} \ + \  \|u\|_{m_0}\|\upsilon\|_{m+r}  \ + \ \|u\|_{m_0}\|\upsilon\|_{m_0}\cdot (1+ \|x\|_{m+r'})\Big).
\ee
\end{hyp3*}

\bigskip

For our purposes, we  require a slight extension of the standard Nash-Moser Implicit Function Theorem 
that keeps track of subspaces that have some specified additional property, denoted $(\text{P})$.

\begin{defn}  \label{propagated}A property (P) that is satisfied on linear (not necessarily closed) subspaces  $\bold P_\mathcal  X\subseteq \mathcal X$ and $\bold P_\mathcal Y\subseteq\mathcal  Y$ is said to be {\bf propagated} by the iteration scheme if
\bea
u \in \bold P_\mathcal  X \ ,  \  f\in \bold P_\mathcal Y & \Rightarrow & S_\e(u)\in \bold P_\mathcal  X\ \  ,\ \   S_\e(f)\in \bold P_\mathcal Y\hspace{1cm} \forall \e\in (0, 1] \\ 
u \in \bold P_\mathcal X  & \Rightarrow & \mathcal F(u)\in \bold P_\mathcal Y\\
x \in \bold P_\mathcal X  \ , \ f \in \bold P_\mathcal Y & \Rightarrow & (\text{d}_x\mathcal F)^{-1}f\in \bold P_\mathcal X. 
\eea
\noindent In particular, in the iteration scheme (\refeq{NMiteration}), if $f$ has property (P) then $x_k$ has property (P) for all $k\geq 0$. 
\end{defn}

\medskip

We will use the following version of the Nash-Moser Implicit Function Theorem. The proof is identical to that in \cite{SecchiNashMoser}, with the additional observation that Hypotheses~{\bf (I)}--{\bf (III)} are only ever invoked at elements $x_k$ occurring in the iteration,  and at linear combinations of the $x_k$ and their smoothings. The proof of smooth dependence on parameters is given in \cite[III.1]{HamiltonNashMoser}.

\begin{thm}\label{nashmoser}{\bf (Nash-Moser Implicit Function Theorem)} Suppose that $\mathcal  X$ and $\mathcal  Y$ are tame Fr\'echet  spaces as in (\refeq{frechetnotation}). Moreover, assume that a property (P) satisfied on linear subspaces $\bold P_\mathcal X\subseteq \mathcal X$ and $\bold P_\mathcal  Y\subseteq \mathcal Y$ is propagated, and that Hypotheses {{\bf (I)}}---{\bf {(III)}} hold for $x \in U_0\cap \bold P_\mathcal X$. 

\begin{enumerate}
\item[(A)] There exists an $m_1 \geq m_0$ depending on $s,s', r, r'$ and a $\delta_1\geq 0$ such that if $f\in \mathcal Y$ with $$ f \in \bold P_\mathcal Y \hspace{1.5cm}\text{and}\hspace{1.5cm} \|f\|_{m_1}\leq \delta_1$$
\noindent then there exists a unique solution $x \in \mathcal X$ of $$\mathcal F(x)=f.$$

\item[(B)] Suppose, in addition, that  $\mathcal F$ and $f$ are parameterized (via a smooth tame map)  by   another tame Fr\'echet  space $\mathcal P$ with $f_{p_0}=0$ at $p_0\in\mathcal P$.  If the Hypotheses {{\bf (I)}}--{\bf {(III)}} hold uniformly on an open neighborhood $V_0 \subset \mathcal P$ of $p_0$ and  $\|f_p\|_{m_1}<\delta_1$ for all $p\in V_0$,   then the unique solution $x_p$ of
$$\mathcal F_p(x)=f_p$$
\noindent also depends smoothly on  $p$ locally near $p_0$. 
\end{enumerate}\qed
\end{thm}
\noindent In case (B), 
 smooth tame dependence on $p$ means that we replace $\|x\|_{m+s'}$ and $\|x\|_{m+r'}$ on the right-hand sides of Hypothesis {\bf (II)} and {\bf (III)} by  $\|(p,x)\|_{m+s'}$ and $\|(p,x)\|_{m+r'}$.  Case (B) is equivalent to the assertion that
 $$\mathcal F^{-1}(f_p) \subset \mathcal P\times \mathcal X$$
is locally a tame Fr\'echet  submanifold that is a graph over $\mathcal P$. 

\bigskip 

\section{Tame Estimates}
\label{section8}

In this final section, we complete the proofs of Theorem \ref{mainb} and Corollary \ref{maind} by verifying the hypotheses of the Nash-Moser Implicit Function Theorem \ref{nashmoser} for the operator 
\be
\overline{\slashed {\mathbb D}}_p: \mathcal P \times \mathcal  X \lre  \mathcal Y \hspace{3cm}\overline{\slashed {\mathbb D}}_p:=( \slashed {\mathbb D}_p-\Lambda\, \text{Id} \ , \ 1-\|\Phi\|^2_{L^2})
\label{extendedunivdirac}\ee
 on tame Fr\'echet spaces $\mathcal X=\{(\eta,\Lambda, \ph)\}$ and $\mathcal Y=\{\psi, c\}$ introduced in Section \ref{section8.35}. Here $\Lambda,c\in \R$ and $\mathcal P=\{(g,B) \}$ is the space of smooth metrics and perturbations (equipped with the standard Fr\'echet structure arising from the $H^{m}$ Sobolev norms on $Y$). 

 In our case, the property (P) that is propagated by the iteration scheme is polyhomogeneity of the spinor. Set: 
\bea \bold P_\mathcal X&:=& \{(\eta, \Lambda, \ph) \in \mathcal X \ | \ \ph \text{ is polyhomogenous with index set } \Z^+\! +\tfrac12 \} \\     \bold P_\mathcal  Y&:=& \{  \ \  \ (\psi,c)\in \mathcal Y \hspace{-.01cm} \  \ | \ \psi \text{ is polyhomogenous with index set } \Z^+\! -\tfrac12  \}\eea

\noindent Here, we use a slightly weaker notion of polyhomogeneity than is given in Definition \ref{asymptoticexpansion}. More specifically, we do not constrain the $\theta$ modes, so that $\ph \in \bold P_\mathcal X, \psi\in \bold P_\mathcal Y$ means that there are respectively asymptotic expansions
\begin{eqnarray} & & \ph\sim \begin{pmatrix} c(t,\theta)  \\ d(t,\theta) \end{pmatrix}r^{1/2}  \  \ + \  \sum_{ n\geq 1  } \sum_{p=0}^{n} \begin{pmatrix}  \ \ c_{n,p}(t,\theta)  \ \  \\  \ \ d_{n,p}(t,\theta)  \ \ \end{pmatrix} r^{n+1/2}(\log r)^p \label{polyhomZ+12} \\  & &\psi\sim \begin{pmatrix} c(t,\theta)  \\ d(t,\theta) \end{pmatrix}r^{-1/2}  +  \ \sum_{ n\geq 1  } \sum_{p=0}^{n} \begin{pmatrix}  \ \ c_{n,p}(t,\theta)  \ \  \\  \ \ d_{n,p}(t,\theta)  \ \ \end{pmatrix} r^{n-1/2}(\log r)^p \label{polyhomZ-12} \end{eqnarray}
\noindent where $c_{n,p}, d_{n,p}\in C^\infty(S^1\times S^1)$ and $\sim$ denotes convergence in the sense of Definition \ref{polyhomogenousdef}.

This section is divided into six subsections. Subsections \ref{section8.1}--\ref{section8.3} cover preliminary material used to verify the hypotheses of the Nash-Moser theorem. Specifically, subsections \ref{section8.1} and \ref{section8.2} are devoted to lemmas used in the verification of the   Hypothesis~{\bf (I)}. Then in subsection \ref{section8.3} the precise form of the derivative and second derivative of $\slashed {\mathbb D}_p$ are derived using the non-linear version of Bourguignon-Gauduchon's Formula (\ref{bourguignon}). Subsection \ref{section8.35} introduces the tame Fr\'echet spaces $\mathcal X,\mathcal Y$, and Subsection \ref{section8.4} derives tame estimates verifying Hypotheses {\bf (I)-(III)}. The final subsection \ref{section8.5} invokes Theorem \ref{nashmoser} to complete the proofs. 
\subsection{The Obstruction Bundle}
\label{section8.1}

This subsection covers preliminary lemmas used in the verification of Hypothesis {\bf (I)} which asserts that the linearization of $\text{d}\slashed {\mathbb D}$ is invertible on a neighborhood of $((g_0,B_0), \mathcal Z_0, \Phi_0)$.  Although the invertibility of the linearization ultimately comes down to the fact that there is an open neighborhood of invertible operators around the identity in a Banach space, the proper context in which to invoke this fact is somewhat subtle.
The first step is to upgrade the obstruction space $\text{\bf Ob}(\mathcal Z_0)$ to a vector bundle. This is the content of the current subsection.

We begin by defining the bundle $\text{\bf Ob}\to V_0$, where $V_0$ is an open ball of radius $\delta_0$ around $p_0\in \mathcal P$ measured in the $m_0$-norm. Here, $m_0 \in \N$ is an integer to be chosen later ($m_0=11$ works). Let $p\in V_0$. By parallel transport on cylinders as in Section \ref{section5.1}, we may think of the Dirac operator for any $p$ as acting on sections of the spinor bundle $S_{0}$; we write 

\be \slashed D_p:= \tau_{0}^p\circ  \slashed D_{g,B}\circ(\tau_{0}^p )^{-1}\label{abuseofnot}\ee

\noindent for this incarnation of the Dirac operator with respect to $p=(g,B)$ (and fixed singular locus $\mathcal Z_0$) on the spinor bundle $S_{0}$ via parallel transport, which we now denote $\tau_0^p$ (rather than $\tau_{g_0}^g$). 

By the (standard) Implicit Function Theorem with the Fredholm operator $\slashed D_p^\star \slashed D_p: rH^1_e \to r^{-1}H^{-1}_e$ and (the second order analogue of) Proposition \ref{asymptoticexpansion} we conclude: 

\begin{lm}\label{secondordereigenvalues}
Provided $0<\delta_0$ is sufficiently small, for every $p\in V_0$ there is a unique eigenvector  $(\Phi_p, \mu_p) \in rH^1_e(Y\setminus \mathcal Z_0; S_0)$ which satisfies $$\slashed D_p^\star \slashed D_p \Phi_p= \mu_p \Phi_p$$ and is equal to $(\Phi_0,0)$ at $p_0=(g_0,B_0)$. Moreover, these satisfy
\bea
\|\Phi_p-\Phi_0\|_{H^{m_0-2,1}_{\text{b},e}} \ + \ |\mu_p|&\leq &C_{m_0}\|p-p_0\|_{m_0}.
\eea  
\noindent and $\Phi_p$ is polyhomogeneous with index set $\Z^++\tfrac12$. 
\qed
\end{lm}

\medskip

Next, let $rH^1_{\perp}$ denote the $L^2$-orthogonal complement of $\Phi_p$ in $rH^1_e$. A trivial extension of the arguments in Section \ref{section2} shows the following lemma. In the statement, $\slashed D_p^\star$ denotes the adjoint of the Dirac operator with respect to the $L^2$-inner product formed using $g_0$.  

\begin{lm}\label{Dpinj}
For $0<\delta_0$ sufficiently small, the following hold for $p\in V_0$:  
\begin{itemize}
\item[(A)] $\slashed D_p: rH^1_\perp \to L^2$ is injective with closed range. 
\item[(B)] $\slashed D_p^\star \slashed D_p: rH^1_\perp \to r^{-1}H^{-1}_e / \R \Phi_p$ is an isomorphism and the solution operators defined by $$P_p(f)=u \ \ \ \ \  \  \text{ s.t.}  \ \ \ \ \ \ \begin{matrix} \slashed D_p^\star \slashed D_p u =f \mod \Phi_p  \\ \br u, \Phi_p \kt_{L^2}=0 \ \  \  \  \ \ \  \ \  \ \ \ \ \ \ \  \end{matrix}$$
\noindent have uniformly (in $p$) bounded norm as operators $P_p:r^{-1}H^{-1}_e\to rH^1_e$.  
\end{itemize}\qed

\end{lm}

\medskip 

As a result of Item (A), $\frak R:=\text{range}(\slashed D_p|_{rH^1_\perp})\subseteq L^2$ is a smooth Banach subbundle, and we may define

\begin{defn} The {\bf Obstruction bundle} denoted $\text{\bf Ob}\to V_0$ is defined as the $L^2$-orthogonal complement of $\frak R$ so that there is an orthogonal splitting of the trivial bundle $$V_0\times L^2(Y\setminus \mathcal Z_0; S_0) = \text{\bf Ob}\oplus \frak R$$ as smooth Banach vector bundles over $V_0$. For $m\leq m_0-3$, we denote the higher-regularity versions by $\text{\bf Ob}^m:= \text{\bf Ob}\cap H^m_\text{b}$ and $\frak R^m:= \frak R \cap H^m_\text{b}$.
\end{defn}

Notice that $\text{\bf Ob}$ is a vector bundle by construction, without any mention of the map $\text{ob}$ constructed in Section \ref{section4}. This prevents any circularity in the following proposition, the notation of which is explained in the proof.

\begin{prop}\label{Obtrivialization}
Provided $\delta_0$ is sufficiently small, then for every $m\leq m_0-3$ and in particular for $m=5/2$, the map 
 \bea \Xi:   V_0\times \left(H^{m}(\mathcal Z_0; \mathcal C_0) \oplus \R\right) &\lre & \text{\bf Ob}^m \\  (p,\xi(t), c)& \mapsto & \tau_{0}^p\circ \text{ob}_p(\xi(t))  \ + \ c \Pi_p \Phi_p  \eea 
\noindent is a trivialization of the vector bundle $\text{\bf Ob}^m$. Moreover, for $(\xi(t),c)=(e^{i\ell t},0)$, the image $\text{ob}_p(e^{i\ell t})=\Psi_\ell^p$ obeys the conclusions of Proposition \ref{cokproperties2}.
\end{prop}

\begin{proof} The proof is a parameterized version of the construction in Section 4. Mimicking the first step in Section 4, define
\be  \text{ob}_p^\text{pre}(\xi(t)=\Pi_p (\chi_1 \frak P^p_N(\xi(t)))\label{obp}\ee
\noindent where $\frak P_N^p$ is the Poisson extension operator on the normal bundle formed as in Section \ref{section4.1} using the metric, perturbation, and Fermi coordinates of $p=(g,B)$. Note that (\refeq{obp}) is an expression in Fermi coordinates and the accompanying trivialization of the spinor bundle $S_p$ corresponding to the parameter $p=(g,B)$, rather than viewed on $S_0$ via parallel transport. The construction of Fermi coordinates in Definition \ref{fermi} depended on a choice of orthonormal frame along $\mathcal Z_0$; to ensure smoothness as $p$ varies, we adopt the convention that a frame is fixed for $p_0$, and the frame at $p$ is defined by orthonormalizing via Gram-Schmidt beginning with the vector tangent to $\mathcal Z_0$. In a slight abuse of notation, (\refeq{obp}) and what follows use $\Pi_p$ and $\text{\bf Ob}_p$ to denote the subspace and $L^2$-orthogonal projection in both $L^2(Y\setminus \mathcal Z_0; S_p)$ and to its image in $L^2(Y\setminus \mathcal Z_0;S_0)$ under parallel transport $(\tau_{0}^{p})^{-1}$. 

The remainder of the constructions in Section 4 can be done in a parameterized way, to yield (i) a map $\text{ob}_p$ obeying the conclusions of Proposition \ref{cokproperties}, and (ii) a basis $\Psi_\ell^p$ of $\text{\bf Ob}_p^\perp$ obeying the conclusions of Proposition \ref{cokproperties2} uniformly for $p\in V_0$. Here $\text{\bf Ob}_p^\perp=\{ \psi \in \text{\bf Ob}_p \ | \ \br  \psi , \Pi_p\Phi_p\kt_{L^2}=0$. Note here that the only distinction from the construction for $p=p_0$ is that we not necessarily have $\Phi_p\in \text{\bf Ob}_p$ in general, so the projection $\Pi_p$ must be included here and in the definition of $\Xi$. 

Our choice of parameterized Fermi coordinates ensures all these constructions are smooth in $p$ (for values of $m$ up to ones comparable to $m_0$, say $m\leq m_0-3$). It follows that $\text{ob}_p$, and thus $\Xi$ are smooth maps of vector bundles. Moreover, at $p=p_0$, $\Xi_0=\text{ob}_0$ is an isomorphism by Proposition \ref{cokproperties}, and the conclusion follows from the openness of invertible maps on Banach spaces. 
\end{proof}

\subsection{Invertibility on a Neighborhood}
\label{section8.2}
This subsection proves a partial version of Hypothesis {\bf (I)}. Namely, we show that the $\text{\bf Ob}_p$ component of the linearization at $(p_0, \mathcal Z_0, \Phi_0)$ is an isomorphism; the complete version of this statement (at a general $(p, \mathcal Z, \Phi_p)$) is a straightforward extension and is completed in Section \ref{section8.4}.  

It is worth drawing the reader's attention to the importance of the upcoming Proposition \ref{Invertiblepreliminary}. In applications of the Nash-Moser Implicit Function Theorem the key point is often to show that the loss of regularity obeys some ``stability'' property with respect to the parameter. That is, to show that linearizations nearby the central parameter $p=p_0$ are bounded into the same function spaces hence are {\it bounded} perturbations of the central linearization. In our situation here, the crux of this comes down to showing that that the notion of conormal regularity from Section \ref{section6.1} is preserved under perturbation of $p$. Proposition \ref{Invertiblepreliminary} below establishes this, and is the most crucial step in the proof of Theorem \ref{mainb} with much of the remainder being essentially routine (but somewhat lengthy) verifications of tame estimates. 

To begin, extend the map  $T_{\Phi_0}$ from Section \ref{section6} to include the $\lambda$-component by setting  

$$\overline T_{\Phi_0} = \text{ob}^{-1}_0 \circ \Pi_0 \Big[(\text{d}_{(\mathcal Z_0, \Phi_0)} \overline{\slashed {\mathbb D}}_0(\eta,0, \lambda)\Big].$$
Assumption \ref{assumption5} and elliptic bootstrapping imply that  $\overline T_{\Phi_0} :H^3(\mathcal Z_0;N\mathcal Z_0)\oplus \R \to H^{5/2}(\mathcal Z_0; \mathcal C_0)\oplus \R $ is an isomorphism. 

\begin{prop}
\label{Invertiblepreliminary}

Provided that $m_0\geq 10$ and $0<\delta_0<<1$ is sufficiently small, then for $p\in V_0$, the $\text{\bf Ob}_p$ component of the linearization at $(p_0, \mathcal Z_0, \Phi_0)$ in the trivialization provided by $\Xi_p$, i.e.  

\be  \Xi^{-1}_p \circ \Pi_p (\text{d}_{(\mathcal Z_0, \Phi_0)}\overline{\slashed {\mathbb D}}_0) : H^{3}(\mathcal Z_0; N\mathcal Z_0) \oplus \R \lre H^{5/2}(\mathcal Z_0; \mathcal C_0)\oplus \R,\label{linearizationonobp}\ee 
is an isomorphism, and the estimate \be  \|\eta\|_{H^3} + \| \lambda\| \leq C \|\text{d}_{(\mathcal Z_0, \Phi_0)} \overline{\slashed {\mathbb D}}_0 (\eta,0,\lambda)\|_{\text{\bf Ob}^{5/2}\oplus L^2}\ee
hold uniformly for $p \in V_0$. 
\end{prop}

\begin{proof} At $p=p_0$, then $\Xi^{-1}_{0}\circ \Pi_{0}= \text{ob}^{-1}\circ \Pi_0$ so the map (\refeq{linearizationonobp}) is simply $\overline T_{\Phi_0}$ thus an isomorphism (by assumption \ref{assumption5}). It therefore suffices to show that for $p\in V_0$ with $m_0\geq 10$ and $0<\delta_0<<1$ sufficiently small, that (\refeq{linearizationonobp}) is bounded. Indeed, given this, Lemma \ref{Obtrivialization} shows that (\refeq{linearizationonobp}) is a continuous family of bounded maps between fixed Banach spaces, hence is an isomorphism for $\delta_0$ sufficiently small. 

Boundedness is equivalent to the assertion that for $\overline \eta=(\eta,\lambda) \in H^{3}(\mathcal Z_0;N\mathcal Z_0)\oplus \R$ one has 
\be \Pi_p \overline{ \mathcal B}_{\Phi_0}(\overline \eta)\in \text{\bf Ob}_p\cap H^{5/2}_\text{b} \hspace{1cm}\text{and}\hspace{1cm} \|\Pi_p(\overline{\mathcal B}_{\Phi_0}(\overline \eta))\|_{H^{5/2}}\leq C \|\overline \eta\|_{H^{3}}.\label{3/2regularityp}\ee 

\noindent where the latter estimate holds uniformly over $p\in V_0$. We establish (\refeq{3/2regularityp}) via a parameterized version of the conormal regularity Lemma  \ref{conormalCaseA} from Section \ref{section6.1}, which completes the proof.  

Let $\text{ob}_p, \Psi_\ell^p, \Phi_p$, and $\tau_{0}^p$ be as in the proof of Proposition \ref{Obtrivialization}, and let $(t_p,x_p,y_p)$ denote the Fermi coordinates formed using the metric of $p=(g,B)$. Since parallel transport $\tau_{0}^{p}$ is an isometry and preserves $H^m_\text{b}$ (for $m\leq m_0-3$, by the differentiability of ODE solutions with respect to parameters), Corollary \ref{higherordercokernel} implies that $\Pi_p\mathcal B_{\Phi_0}(\eta)\in H^m_\text{b}(Y\setminus \mathcal Z_0)$ if and only if the sequence of inner products 

$$\text{ob}_p^{-1}\Pi_p\left((\tau_0^p)^{-1}\mathcal B_{\Phi_0(\eta)}\right)=\sum_{\ell \in \Z }\br  (\tau_{0}^{p})^{-1}\mathcal B_{\Phi_0}(\eta), \Psi_\ell^p\kt_{\C} \ e^{i\ell t_p}=\sum_{\ell \in \Z }\br  \mathcal B_{\Phi_0}(\eta), \tau_{0}^p\Psi_\ell^p\kt_{\C} \ e^{i\ell t_p},$$

\noindent lies in $H^m(\mathcal Z_0;\mathcal C_0)$, where we again commit the abuse of notation of conflating $\Pi_p$ on both sides of the parallel transport. Since $\Pi_p(\Phi_p) \in H^m_\text{b}$ for any $m$ by Lemma \ref{secondordereigenvalues}, the distinction between $\mathcal B_{\Phi_0}, \overline{\mathcal B_{\Phi_0}}$ is immaterial. 

Now, we may write $\mathcal B_{\Phi_0}$ as a collection of terms of the form \be \eta'(t) \Theta_0(t_p, x_p,y_p) \hspace{1cm}\text{or}\hspace{1cm} \eta''(t) \Theta_1(t_p,x_p,y_p),\label{twotypes}\ee 

\noindent and apply Case (C) of Corollary \ref{CaseBC} in the Fermi coordinates of $p$. Since $r_p\sim r$, the bounds $|\nabla_{t_p}^m \Theta_k| \leq Cr^{k-1/2}$ hold for $m\leq m_0$ equally well for $r_p$. It remains to write $\eta'(t),\eta''(t)$ in terms of the Fermi coordinates $(t_p,x_p,y_p)$, since curves parallel to $\mathcal Z_0$ in $g_0$ need not be parallel anymore in $g$. Expanding in Taylor series along $\mathcal Z_0$ in the normal directions, $\eta'(t)=w(t_p) +  F(t_p, x_p, y_p)$ where $w(t_p)\in H^{2}(S^1;\C)$ and 
\bea |F(t_p, x_p,y_p)|&\leq &Cr_p(|\nabla_{x_p} \eta'(t)| +|\nabla_{y_p} \eta'(t)|\\
&\leq &  Cr^2_p |\eta''(t)| 
\eea 
\noindent where we have written $x_p(t)= a_0(t) x + b_0(t)y + O(r) t + \ldots$ and likewise for $y_p$. The crucial point here is that an extra factor of $r_p$ arises since normal planes in the metrics of $p_0$ and $p$ differ to first order in $r_p$ by a linear coordinate change of $x,y$. Similar arguments apply to the second derivative $\eta''(t)$. Applying Corollary \ref{CaseBC} for the parameter $p$ then shows that all terms in (\refeq{twotypes}) have conormal regularity at least $5/2$. This establishes (\refeq{3/2regularityp}), completing the proof. 
\end{proof}

\subsection{Quadratic and Error Terms}
\label{section8.3}
This section calculates linearization, the second derivative, and the initial error $f_p$ at any arbitrary tuple $(p,\mathcal Z, \Phi)$ near $(p_0, \mathcal Z_0, \Phi_0)$. This is done with the non-linear version of Bourguignon-Gauduchon's formula \cite{Bourguignon} for the metric variation of the Dirac operator. 

To state Bourguignon-Gauduchon's formula, let $p=(g,B)$ be a parameter pair of a metric and perturbation on $Y$. The Dirac operator $\slashed D_p$ is viewed as an operator on sections of $S_0$ as in (\refeq{abuseofnot}).  Let $a_{g_0}^g, \frak a\in \text{End}(TY)$ be defined respectively by $$g(V,W)=g_0(a_{g_0}^g V, W) \hspace{3cm} \frak a=(a_{g_0}^g)^{-1/2}$$  
where the latter is understood via the eigenvalues of $(a_{g_0}^g)^\star a_{g_0}^g$, which are non-zero for $h$ sufficiently close to $g_0$. 

\begin{thm} {\bf (Bourguignon-Gauduchon, \cite{Bourguignon})} The Dirac operator $\slashed D_p$ is given by \be\slashed D_p\Psi=\left( \sum_{i} e^i. \nabla^{B}_{\frak a(e_i)} \  + \  \frac{1}{4}\sum_{ij}e^i e^j.   \left(\frak a^{-1}(\nabla^{g_0}_{\frak a(e_i)} \frak a) e^j + \frak a^{-1}(\nabla^g-\nabla^{g_0})_{\frak a(e_i)} \frak a(e^j) \right).\right) \Psi  \label{NLBG}\ee 
 where $e^i$ and $.$ are an orthonormal basis and Clifford multiplication for $g_0$, and $\nabla^g$ denotes the unperturbed spin connection of the metric $g$ and likewise for $g_0$. 
 \label{nonlinearBG}
 \end{thm}
 
 \subsubsection{Error Terms:}
 \label{errortermssection} We begin by applying Theorem  \ref{nonlinearBG} to calculate the initial error terms $f_p$ for the application of the Nash-Moser Implicit Function Theorem (\ref{nashmoser}). The initial error is given by \be \boxed{f_p:= \slashed D_p \Phi_0.}\label{errortermdef}\ee
 
 Let $U_1\subset \mathcal P$ denote the ball around $p_0$ of radius $\delta_1$ measured in the $m_1+3$ norm. Here, $m_1$ (like $m_0$) is an integer to be chosen later. To simplify notation, we omit the reference to the spaces from the notation from the norms, so that e.g. $\|-\|_{m}$ means the $H^{m,1}_{\text{b},e}$-norm for elements of the domain, the $H^m_\text{b}$-norm for elements of the codomain. 
 \begin{lm} \label{errortermslem}The Dirac operator at parameter $p$ can be written \be \slashed D_p= \slashed D_0 + \frak D_p\label{Diracatp}\ee where the latter satisfies \be \|\frak D_p \ph\|_{m_1}\leq C_{m_1} \| p\|_{m_1+3} \|\ph\|_{m_1}\label{812}.\ee
 It follows that $\|f_p\|_{m_1}\leq C \delta_1$. 
 \end{lm}
 
 \begin{proof}
 Write $p= (g_0, B_0)+ (k,b)$ for $\|(k,b)\|_{m_1+3}\leq \delta$. In an orthonormal frame for $g_0$ we have $a_{g_0}^{g_0+k}=\text{Id}+k$ where we also use $k$ to denote the corresponding matrix in this orthonormal frame. Then $\frak a= (\text{Id}+k)^{-1/2}$. Substituting this into (\refeq{NLBG}) shows that $$\slashed D_p\ph= \slashed D_0\ph + \frak d_1 \ph + \frak d_0\ph$$
 where $\frak d_1, \frak d_0$ are respectively a first order and zeroth order operator satisfying $\|\frak d_1 \ph \|_{m_1} \leq C\|p\|_{m_1+3} \|\ph\|_{m_1}$ and $\|\frak d_0 \ph\|_{m_1}\leq C\|p\|_{m_1+3}\|\ph\|_{m_1}$. To see this, note that the coefficients of $\frak d_1$ are formed from sums and products of entries of $k$ (by expanding $(\text{Id}+k)^{-1/2})$, and these all lie in $C^{m_1+1}(Y)\hookrightarrow H^{m_1+3}(Y)$ by the Sobolev embedding and the fact that $C^{m_1+1}$ is an algebra. Likewise, coefficients of $\frak d_0$ lie in $C^{m_1}(Y)$ because they are formed from sums and products of up to first derivatives of $k,b$. Since every term is at least linear in $p$, and $\|(k,b)\|_{m_1+3}\leq \delta_1<<1$, the bound (\refeq{812}) follows. 
 
 Since $f_p= \slashed D_p \Phi_0= \frak d_0 \Phi_0 +\frak d_1 \Phi_0$ and $\|\Phi_0\|_{m_1}\leq C_{m_1}$, the second statement is then immediate for $p\in U_1$.     \end{proof}
 
 \subsubsection{Quadratic Terms}
 
 For the tame estimates on $\text{d}\slashed{\mathbb D}_p$ and $\text{d}^2\slashed{\mathbb D}_p$, we must first investigate the higher-order terms of $\slashed {\mathbb D}_p$. Expanding, we may write $$\slashed {\mathbb D}_p((\mathcal Z_0,\Phi_0) + (\eta,\ph))=f_p \ + \  \text{d}_{(\mathcal Z_0,\Phi_0)}\slashed{\mathbb D}_p(\eta,\ph)  \ + \  Q_p(\eta,\ph)$$ 
 
 \noindent where $Q_p$ is comprised of second order and higher terms.

 The middle term at $p_0$ is given by Corollary \ref{formofDD}. For a general $p $, we can write the derivative of pullback metric as \be \d{}{s} \Big |_{s=0} F_{s\eta}^* (g_0 +k)= \dot g_\eta + \dot k_\eta,\label{linearizedpullbackp}\ee
 where $\dot g_\eta$ is as calculated in (\refeq{pullbacks2}) and analogously for $k$. Analogous to the formula for $\mathcal B_{\Phi_0}(\eta)$ in Corollary \ref{formofDD}, we set 
 
 \be  \frak B_{\Phi_0,p}(\eta):=\left(-\frac{1}{2}\sum_{ij} \dot k_{\eta}(e_i,e_j) e^i . \nabla^{g_0}_j  + \frac{1}{2} d \text{Tr}_{g_0}(\dot k_{\eta}). +\frac{1}{2} \text{div}_{g_0}(\dot k_{\eta}).  + \mathcal R(b,\chi\eta).\right)\Phi_0 \label{BPhiatp} \ee
 to be the term arising from the perturbation $(k,b)$ to $p_0$. Here $\mathcal R(b,\chi\eta)$ is a zeroth order term in $\eta$ with coefficients depending on the perturbation $b$ to $B_0$ and its derivatives. The proof of the following proposition is given in Appendix \ref{appendixC}.
 
 \begin{prop} \label{fullynonlinear}
The universal Dirac operator at the parameter $p \in U_1$ for at a point $(\mathcal Z_0,\Phi_0)+(\eta,\ph)$ with $\|(\eta,\ph)\|_{m_0}\leq C\delta$ is given by 

\be \slashed {\mathbb D}_p((\mathcal Z_0,\Phi_0) + (\eta,\ph))=f_p \ + \  \text{d}_{(\mathcal Z_0,\Phi_0)}\slashed{\mathbb D}_p(\eta,\ph)  \ + \  Q_p(\eta,\ph)\ee
\noindent where 

\begin{enumerate}
\item[(A)] $f_p=\slashed D_p\Phi_0$ as in Lemma  \ref{errortermslem}. 
\item[(B)] The derivative is given by 

$$ \text{d}_{(\mathcal Z_0,\Phi_0)}\slashed{\mathbb D}_p(\eta,\ph)=\Big(\mathcal B_{\Phi_0}(\eta) + \slashed D_0\ph\Big) \  + \  \Big ( \mathfrak B_{\Phi_0,p}(\eta)+ \frak D_p(\ph) \Big )$$

\noindent  where $\mathcal B_{\Phi_0}(\eta)$ is as defined in \ref{abstractfirstvariation} (cf. Corollary \ref{formofDD}),  and $\mathfrak D_p,\mathfrak B_{\Phi_0,p}$ are as in (\refeq{Diracatp}) and (\refeq{BPhiatp}) respectively.

\item[(C)] The non-linear terms may be written
\bea Q_p(\eta,\ph)&=&(\mathcal B_\ph + \mathfrak B_{\ph, p})(\eta) \ \ + \ \ M^1_p(\eta', \eta')\nabla( \Phi_0+ \ph)  \ \ + \ \  M^2_p (\eta',\eta'') (\Phi_0+\ph) \ \ + \ \ F_p(\eta, \Phi_0+ \ph)  \eea

\noindent where \begin{enumerate}
\item[(i)] $\mathcal B_{\ph}, \mathfrak B_{\ph,p}$ are defined identically to $\mathcal B_{\Phi_0}, \mathfrak B_{\Phi_0,p}$ but with $\ph $ replacing $\Phi_0$. 
\item[(ii)] $M_1^p$ a finite sum of terms involving quadratic combinations of $\chi \eta', \eta d\chi, \chi \eta$, and linearly depending on $\nabla(\Phi_0+ \ph)$ and smooth endomorphisms $m_i$, e.g.   $$m_i(y) ( \chi \eta' )(\chi \eta') \nabla_j (\Phi_0 + \ph)$$

where $m_i(y)$ depend on $g_0+k$ (and no derivatives). 

\item[(iii)]  $M_2^p$ a finite sum of terms involving quadratic combinations of $  \eta''\chi , \eta'd\chi, \eta d^2 \chi ,\eta'\chi, \eta d\chi,  \eta\chi$, with at most one factor of $\eta''$, and linearly depending on $\Phi_0+ \ph$ and smooth endomorphisms $m_i$, e.g.   $$m_i(y) ( \chi \eta'' )(d\chi \eta').(\Phi_0 + \ph)$$

where $m_i(y)$ depend on up to first derivatives of  $g_0+k$ and $B_0+b$. 
\item[(iv)] $F_p$ is formed from a finite sum of similar terms but involving cubic and higher combinations of $\eta, \eta', \eta''$, with at most one factor of $\eta''$. 
\end{enumerate} 
\end{enumerate}

\end{prop}

Straightforward differentiation now shows the following precise forms for the first and second derivatives. In these formulas, we use the notation that e.g. $F(p^3, q^2, s)$ denotes a term depending cubicly on $p$ and its derivatives, quadratically on $q$ and its derivatives, and linearly on $s$ and its derivatives:  

\begin{cor}\label{tamefirstderiv}
The derivative at a point $(\mathcal Z_0, \Phi_0)+ (\eta,\ph)$ is given by 
\bea
\text{d}_{(\eta,\ph)}\slashed {\mathbb D}_p(v,\phi) &= & \ \ \ \  \ \text{d}_{(\mathcal Z_0, \Phi_0)}\slashed {\mathbb D}_p(v,\phi)  \\  & & \ +  \ (\mathcal B_{\ph} + \frak B_{\ph})(v) \ + \ (\mathcal B_{\phi} + \frak B_{\phi})(\eta)   \\
& &  \ +  \ M^1(\eta', v')\nabla (\Phi_0 + \ph) \ + \  M^1(\eta' , \eta')\nabla \phi \\
& &  \ +  \ M^2(\eta', v'')(\Phi_0 + \ph) \ \ + \  M^2(v', \eta'') (\Phi_0 + \ph) \ + \   M^2(\eta' , \eta'')\phi \\
 & &   \ + \  F^1(\eta^2, v, \Phi_0 + \ph) \ + \ F^2(\eta^3, \phi)
\eea
where the subscript $p$ is kept implicit on the right hand side. Moreover, provided $\|\eta\|_{C^1}\leq 1$, the cubic and higher order terms $F^1$ and $F^2$ are of the form 
\be F^2(\eta^3, \phi)=O(\|\eta'\|_{C^1}) \Big[M^1(\eta',\eta')\nabla \phi \Big]+ O(\|\eta'\|_{C^1})\Big[ M^2(\eta',\eta'')\phi\Big]\label{Fhigherorder}\ee

\noindent and likewise for $F^1$, i.e. they include the same orders of derivatives as the quadratic terms but with additional powers of $\eta', \eta$.   
\end{cor}

 Alternatively, the terms linear in $\phi$ in the expression of Corollary \ref{tamefirstderiv} combine to form the Dirac operator \be \slashed D_{p_\eta} \phi =\slashed D_p\phi +  \ (\mathcal B_{\phi} + \frak B_{\phi})(\eta) +M^1(\eta' , \eta')\nabla \phi + M^2(\eta' , \eta'')\phi + F^2(\eta^3, \phi)\label{petapullback}\ee
with respect to the pullback metric  and perturbation $p_\eta:= F^*_\eta (p)$. In particular, no terms involving $\phi$ have components in $\text{\bf Ob}^p$, provided $\br\phi, \Phi_p \kt_{L^2}=0$. 

\begin{cor}\label{tamesecondderiv}
The second derivative at a point $(\mathcal Z_0, \Phi_0)+ (\eta,\ph)$ is given by 
\bea
\text{d}^2_{(\eta,\ph)}\slashed {\mathbb D}_p\Big((v,\phi), (w,\psi)\Big) &= & \   \ (\mathcal B_{\psi}+\frak B_\psi)(v) \ + \ (\mathcal B_{\phi}+\frak B_\phi)(w)   \\
& & \ + \ M^1(w', v')\nabla(\Phi_0 + \ph)  \ + \  M^1(\eta', v')\nabla \psi  \ + \  M^1(\eta', w')\nabla \phi \\ 
& & \ + \  M^2(w', v'')(\Phi_0 + \ph) \ +  \ M^2(v', w'') (\Phi_0 + \ph) \  + \  M^2(\eta', v'')\psi   \\ & &  \   + \ M^2(v', \eta'')\psi \ +  \ M^2(\eta', w'')\phi \  +  \ M^2(w', \eta'')\phi\\ & & \ + \ F^3(\eta^2, w,  \phi) \  + \  F^4(\eta^2, v, \psi) \  + \  F^5(\eta, v,w, \Phi_0+\ph)
\eea
\noindent where the subscript $p$ is kept implicit on the right hand side. The higher order terms again have the form (\refeq{Fhigherorder}) as in Corollary \ref{tamefirstderiv}.\qed
\end{cor}
\subsection{Tame Fr\'echet Spaces}
\label{section8.35}

This section introduces the tame Fr\'echet spaces used in the proof of Theorem \ref{mainb} and Corollary \ref{maind}. 

 While there is a natural Fr\'echet space of normal vector fields $\eta$ (this being $C^\infty(\mathcal Z_0;N\mathcal Z_0)$ with the Fr\'echet structure arising from the $H^{m}$-norms), there are several possible choices of Fr\'echet spaces for the spinors, arising from different versions of the boundary and edge spaces. The relevant spinors are those lying in $\bold P_\mathcal X$ and $\bold P_\mathcal Y$ for the domain and codomain respectively, i.e. those spinors with polyhomogeneous expansions (\refeq{polyhomZ+12}-\refeq{polyhomZ-12}). While the spaces $\bold P_\mathcal X$ and $\bold P_\mathcal Y$ are themselves tame Fr\'echet spaces, these Fr\'echet structures are rather unwieldy and it is advantageous to enlarge the domain and codomain to spaces where it is easier to obtain estimates and then invoke Theorem \ref{nashmoser} with the property (P) of polyhomogeneity which holds on $\bold P_\mathcal X$ and  $\bold P_\mathcal Y$.  
 
The mixed boundary and edge spaces $rH^{m,1}_{\text{b},e}$ and $H^m_\text{b}$ defined in Section \ref{section2} enlarge the domain and codomain and their norms facilitate much easier estimates using the material of Sections \ref{section2}-\ref{section4}. Unfortunately, these spaces are slightly too large and it is impossible to control the higher order terms of the expansions (\refeq{polyhomZ+12}-\refeq{polyhomZ-12}) simply in terms of of these norms.  To balance these conflicting advantages of $rH^{m,1}_{\text{b},e}$ and $\bold P_\mathcal X$, we opt for intermediate spaces which supplement the $rH^{m,1}_{\text{b},e}$ and $H^m_\text{b}$-norms with the norm of the higher order terms in (\refeq{polyhomZ+12}-\refeq{polyhomZ-12}) using a stronger weight.

Analogously to $rH^{m,1}_{\text{b},e}$ and $H^m_\text{b}$ denote $r^{1+\nu}H^{m,1}_{\text{b},e}$ and $r^\nu H^m_\text{b}$ the spaces formed by adding an overall weight of $r^{-2\nu}$ in the norm (\refeq{bnorm}). Equivalently, $$\ph \in  rH^{m,1}_{\text{b},e} \ \  \Leftrightarrow  \ \ r^\nu \ph \in r^{1+\nu}H^{m,1}_{\text{b},e}$$ so that the multiplication map $r^\nu$ is a linear isomorphism bounded by a constant depending only on $m$, and similarly for $r^\nu H^m_\text{b}$. Fix $\nu=0.9$ and define Banach spaces
 \bea
 r\mathcal H^{m,1}&: =& \left\{ \ph \ \ \Big | \ \ \|\ph\|_{r\mathcal H^{m,1}}:=\left(\|\ph\|^2_{rH^{m,1}_{\text{b},e}}  \ + \  \|(r\del_r - \tfrac{1}{2}) \ph\|^2_{r^{1+\nu} H^{m-1,1}_e} \ \right)^{1/2}<\infty \right\}\\
  \mathcal H^{m,0}&: =& \left\{ \psi \ \ \Big | \ \ \|\psi\|_{\mathcal H^{m,0}}:=\left( \  \|\psi\|^2_{H^{m}_{\text{b}}} \   +  \  \|(r\del_r + \tfrac{1}{2}) \psi\|^2_{r^\nu H^{m-1}_\text{b}} \ \right)^{1/2}<\infty \right\}\\
   r^{-1} \mathcal H^{m,-1}&: =& \left\{ \psi \ \ \Big | \ \ \|\psi\|_{r^{-1}\mathcal H^{m,-1}}:=\left( \  \|\psi\|^2_{r^{-1}H^{m,-1}_{\text{b},e}} \   +  \  \|(r\del_r + \tfrac{3}{2}) \psi\|^2_{r^{-1+\nu} H^{m-1,-1}_{\text{b},e}} \ \right)^{1/2}<\infty \right\}.
 \eea
with the indicated norms. These spaces are defined using the Fermi coordinates and norms of the base parameter $p_0$ and do not depend on $p\in \mathcal P$.  

Using these, we now define the spaces used in the proofs of Theorem \ref{mainb} and \ref{maind}. 

\begin{lm}\label{spaces}
The spaces $$\mathcal X:= \bigcap_{m\geq 0} X_m' \oplus \R \oplus  X_m''   \hspace{2cm}\mathcal Y:= \bigcap_{m\geq 0} Y_m' \oplus Y_m'' \oplus \R$$
\noindent where
\begin{eqnarray}X_m'&:=&H^{m}(\mathcal Z_0; N\mathcal Z_0) \hspace{4.1cm} Y_m' := \text{\bf Ob}\cap H^{m+2}_\text{b}(Y\setminus \mathcal Z_0; S_0)\\
X_m''&:= & r\mathcal H^{m,1}(Y\setminus \mathcal Z_0; S_0) \hspace{3.5cm} Y_m'' := {\frak R} \ \cap \  \mathcal H^{m,0}(Y\setminus \mathcal Z_0; S_0)
\end{eqnarray}
\noindent and $\R$ has the standard norm are tame Fr\'echet spaces as in Definition \ref{tamedef}, and on an open neighborhood $U\subset \mathcal X$,  $$\overline{\slashed{\mathbb D}}_p: \mathcal P\times U\to \mathcal Y$$ is a tame Fr\'echet map. \end{lm}

\begin{proof} The interpolation inequalities in item (I) of Definition \ref{tamedef} are immediate from those on the standard spaces $H^{m}(\mathcal Z_0; N\mathcal Z_0)$ and those from Lemma \ref{interpolationinequalities} which apply equally well for different weights. The smoothing operators whose existence is the content of item (II) of Definition \ref{tamedef} are constructed in Appendix \ref{appendixII}.

 That $\overline{\slashed{\mathbb D}}_p$ is a tame Fr\'echet map is obvious for the $H^m_\text{b}$-norms, and for the $(r\del_r - \tfrac{1}{2})$ terms follows from the commutation relations in the upcoming Lemma \ref{DoncalH}. 
\end{proof}

\begin{rem}
Since $ \text{\bf Ob}$ consists of solutions of the elliptic edge operator (this being $\slashed D$ or $\slashed D-\Lambda_p\text{Id}$) which have expansions with index set $\Z^+-\tfrac12$, edge bootstrapping (see \cite{MazzeoEdgeOperators} Equation (7.7) and the accompanying discussion)  implies that for $\psi \in \text{\bf Ob}$ obeys 

$$\|(r\del_r + \tfrac{1}{2}) \psi \|_{rH^{m-1}_\text{b}}\leq C \|\psi\|_{H^m_\text{b}}.$$

\noindent Since $\nu\leq 1$ it follows that the $H^m_\text{b}$ and the $\mathcal H^{m,0}$ norms are equivalent on $Y_m'$. 
\end{rem}

As explained at the beginning of the subsection, the point is that the additional terms allows control of the higher order terms of expansions (\refeq{polyhomZ+12}-\refeq{polyhomZ-12}) in $\bold P_\mathcal X\cap r\mathcal H^{m,1}$. The following key lemma, proved in Appendix \ref{appendixII} makes this precise:

\begin{lm}\label{coefficientregularity} Suppose that $\ph \in  r\mathcal H^{m+4,1}$ is a spinor. Then the bound

 $$|(\nabla^\text{b})^m \ph|\leq C_m r^{1/2} \|\ph\|_{r\mathcal H^{m+4,1}}$$

\noindent holds pointwise on $Y\setminus \mathcal Z_0$.
 \qed
\end{lm}

 The final two lemmas needed before the verification of Hypotheses {\bf (I)-(III)} are effectively bookkeeping that show the Dirac operator $\slashed D: r\mathcal H^{1,1}\to \mathcal H^{1,0}$ has the same semi-Fredholm properties as on the spaces from Section \ref{section2}. Fixing a parameter $p$, let $r\mathcal H^\perp$ be the $L^2$-orthogonal complement of $\Phi_p$ in $r\mathcal H^{1,1}\subset rH^1_e$. Additionally, fixing a spinor $\ph$ near $0$, denote the extended Dirac operator with the $\R$-factor included at a parameter $p$ by

$$\overline{\slashed D}_{p}= \begin{pmatrix} \slashed D_{p} & 0 \\  \br -, \Phi_0 + \ph\kt  & \br -, \Phi_0 + \ph\kt   \end{pmatrix}  \begin{matrix} r \mathcal H^\perp\\ \oplus \\ \R \Phi_{p}  \end{matrix} \ \  \lre \ \   \begin{matrix} \frak R_{p}\\ \oplus \\   \R,\end{matrix}$$

\noindent which arises as the partial derivative at $\Phi_0+\ph$ of (\refeq{extendedunivdirac}) with respect to the spinor.  

\begin{lm}\label{DoncalH} Provided that $m_0\geq 10$ and $0<\delta_0<<1$ is sufficiently small, then for $p\in V_0$ the extended Dirac operator $$\overline{\slashed D}_p : r\mathcal H^{1,1}\oplus \R \to (\frak R_p \cap \mathcal H^{1,0}) \oplus \R$$ is an isomorphism and the estimate \be\|\psi\|_{r\mathcal H^{1,1}}\leq C \|\overline{\slashed D}_p\ph \|_{ \mathcal H^{1,0}\oplus \R}\label{814ellipticest}\ee holds uniformly for $p \in V_0$.

\end{lm}

\begin{proof} We begin by showing that the (unextended) Dirac operator $\slashed D$ satisfies the following estimate: if $\slashed D\psi = f$ then  
\begin{eqnarray}
 \|\psi\|_{rH^{1,1}_{\text{b},e}} + \|(r\del_r - \tfrac{1}{2}) \psi\|_{r^{1+\nu} H^1_e} \leq C\Big (\|f \|_{H^1_\text{b}} + \|(r\del_r + \tfrac{1}{2})f\|_{r^\nu L^2}) +  \|\psi\|_{r^{\nu}H^1_\text{b}} \Big ).\label{calHelliptic} \end{eqnarray}

 \noindent Since $\nu<1$ by choice, the inclusion $rH^{1,1}_{\text{b},e}\hookrightarrow r^{\nu} H^1_\text{b}$ is compact, hence so is the final term on the right. 
 
 We first prove (\refeq{calHelliptic}) for $p=p_0$. That the first term is bounded by the right-hand side immediate from the estimate for $\slashed D: rH^{1,1}_{\text{b},e}\to H^1_\text{b}$ (Corollary \ref{higherordersemielliptic}). For the second term, we apply the elliptic estimate
 \be \|\psi\|_{r^{1+\nu}H^1_\e}\leq C\Big(\| \slashed D \psi\|_{r^\nu L^2} + \|\psi\|_{r^\nu L^2}\Big)\label{ellipticest2}\ee
 
\noindent  for $\slashed D: r^{1+\nu}H^1_e \to r^{\nu}L^2$ to term $(r\del_r -\tfrac{1}{2})\psi$. This estimate cannot be derived by integration by parts as in Section \ref{section2} and instead follows from parametrix methods (see Theorem 6.1 of  \cite{MazzeoEdgeOperators} or \cite{FangyunThesis}).  Then, since the commutation relations
 \begin{eqnarray}
 (r\del_r + \tfrac{1}{2})\del_r& =& \del_r (r\del_r -\tfrac{1}{2})\label{commutation1}\\
  (r\del_r + \tfrac{1}{2})\tfrac{1}{r}\del_\theta& =& \tfrac{1}{r}\del_\theta (r\del_r -\tfrac{1}{2})\label{commutation2}.
 \end{eqnarray}
 
\noindent hold  writing $\slashed D=\slashed D_0 + \frak d$ as in Lemma \ref{Dlocalexpression} shows that $$\slashed D(r\del_r -\tfrac{1}{2})\psi=(r\del_r +\tfrac{1}{2})\slashed D \psi + B\psi $$
where $B$ is a lower order term such that $B: r^\nu H^1_b \to r^\nu L^2$ is bounded. Applying (\refeq{ellipticest2}) and substituting this expression yields (\refeq{calHelliptic}) for $p=p_0$. 

The estimate (\refeq{814ellipticest}) follows from a standard proof by contradiction (e.g. \cite{LiviuLecturesOnGeometry} Lemma 10.4.9), provided $\ph$ is sufficiently small. For $p\neq p_0$ it is straightforward to show that writing ${\slashed D}_{p}={\slashed D}_{p_0} + \frak d_p$ and using the commutations (\refeq{commutation1}-\refeq{commutation2}) yields $$\|\frak d_p \psi\|_{\mathcal H^{1,0}}\leq C\delta_0 \|\psi\|_{r\mathcal H^{1,1}}$$
completing the lemma.

\end{proof}

Finally, the projection operators to $\text{\bf Ob}$ and ${\frak R}$ are well-behaved on the new spaces analogously to Corollary \ref{higherregularitymappings} item (C). 

\begin{lm}\label{calHprojections}
The projection operators 
$$1-\Pi_p= \slashed D_p P_p \slashed D^\star _p: \mathcal H^m \to \frak R_p\cap  \mathcal H^m \hspace{2cm}  \Pi_p=1-\slashed D_p P_p \slashed D_p^\star :\mathcal H^m \to \text{\bf Ob}_p \cap \mathcal H^m  $$ are bounded. 
\end{lm}  

\begin{proof}For the ${H^m_\text{b}}$-term of the $\mathcal H^m$-norm this follows directly from Corollary \ref{higherregularitymappings}. For the second term, notice that by (\refeq{ellipticest2}) and the analogous estimate for $\slashed D^\star_p \slashed D_p : r\mathcal H^{m,1}\to r^{-1}\mathcal H^{m,-1}$, one has that \be \slashed D_pP_p\slashed D_p^\star: r^\nu H^m_\text{b}\to r^\nu H^m_\text{b}\label{weightedproj}\ee is bounded. Writing $$(r\del_r+\tfrac{1}{2})D_pP_p\slashed D_p^\star=D_pP_p\slashed D_p^\star (r\del_r+\tfrac{1}{2}) + [(r\del_r+\tfrac12),D_pP_p\slashed D_p^\star ]$$
and applying (\refeq{weightedproj}) to the first term, then using that $[(r\del_r+\tfrac12),D_pP_p\slashed D_p^\star ]: H^{m}_\text{b}\to rH^{m-1}_\text{b}\subseteq r^\nu H^{m-1}_\text{b}$ is bounded for the second term yields the result.  

To finish, we therefore prove that the commutator $[(r\del_r+\tfrac12),D_pP_p\slashed D_p^\star ]: H^{m}_\text{b}\to rH^{m-1}_\text{b}\subseteq r^\nu H^{m-1}_\text{b}$ is bounded, beginning with the product metric on $Y_\circ=S^1\times \R^2$ as in Example \ref{euclideanexample}. In the product case, the commutation relations (\refeq{commutation1}-\refeq{commutation2}) imply
\bea \slashed D_0(r\del_r -\tfrac{1}{2})\phi&=&(r\del_r +\tfrac{1}{2})\slashed D_0 \phi - \gamma_t \nabla_t \phi\\
 P_0(r\del_r +\tfrac32)f &=&(r\del_r -\tfrac12) P_0f  + P_0 (\gamma_t \nabla_t \slashed D_0 + \slashed D_0 \gamma_t \nabla_t)P_0f \eea

 \noindent where $P_0f=u \Leftrightarrow \slashed D_0\slashed D_0u=f$, and $\gamma_t=\gamma(dt)$. The latter expression follows from applying the first twice with $\slashed D_0 \slashed D_0$ and then applying $P_0$. Using these, one has 
  \bea
 (r\del_r + \tfrac12)\slashed D_0 P_0 \slashed D_0&=& \slashed D_0(r\del_r -\tfrac12)P_0\slashed D_0  + \gamma_t \nabla_t P_0\slashed D_0 \\
 &=&  \slashed D_0P_0(r\del_r + \tfrac32)\slashed D_0 - \slashed D_0 P_0 (\gamma_t \nabla_t \slashed D_0 + \slashed D_0 \gamma_t \nabla_t)P_0 \slashed D_0   + \gamma_t \nabla_t P_0\slashed D_0\\
  &=&  \slashed D_0P_0\slashed D_0 (r\del_r + \tfrac12)+\slashed D_0 P_0 \gamma_t\nabla_t - \slashed D_0 (P_0 (\gamma_t \nabla_t \slashed D_0 + \slashed D_0 \gamma_t \nabla_t)P_0) \slashed D_0   - \gamma_t \nabla_t P_0\slashed D_0
 \eea
 
 \noindent so that $$[(r\del_r + \tfrac{1}{2}), \slashed D_0 P_0\slashed D_0]= \slashed D_0 P_0 \gamma_t\nabla_t - \slashed D_0 (P_0 (\gamma_t \nabla_t \slashed D_0 + \slashed D_0 \gamma_t \nabla_t)P_0) \slashed D_0   +\gamma_t \nabla_t P_0\slashed D_0.
$$
 
\noindent Each term in the above has a factor of $r$ better than {\it a priori} might be expected: the final term, for instance, is a composition of bounded operators  $$H^m_\text{b} \ \overset{\slashed D_0} \lre \  r^{-1}H^{m,-1}_\text{b} \ \overset{P_0}\lre  \ r H^{m,1}_{\text{b},e} \ \overset{\nabla_t}\lre  \ rH^{m-1,1}_{\text{b},e} \ \hookrightarrow  \ rH^{m-1}_\text{b}   $$
 
 \noindent and similarly for the first and middle terms. 

For the parameter $p=p_0$ the error terms arising from the difference $\frak d=\slashed D_{p_0}- \slashed D_0$ also has an additional factor of $r$. Indeed, $|\frak d\psi|\leq C( r|\nabla \psi| + |\psi|)$ as in Lemma \ref{Dlocalexpression} is bounded $r^{1+\nu}H^{m,1}_{\text{b},e}\to r^{1+\nu}H^m_\text{b}$. For a general $p$, the same argument applies in the Fermi coordinates formed using $p$ for the boundary Sobolev spaces defined using $p$, whose norms are uniformly (tamely) equivalent. 
\end{proof}

\subsection{Tame Estimates for the Linearization}
\label{section8.4}

This subsection verifies Hypotheses {\bf (I)-(III)} of the Nash-Moser Implicit Function Theorem from Section \ref{NM1}, employing the formulas for $\text{d}\slashed{\mathbb D}_p$ and $\text{d}^2\slashed{\mathbb D}_p$ from Corollaries \ref{tamefirstderiv} and \ref{tamesecondderiv}.

Recall that $V_0\subset \mathcal P$ denotes the open ball of radius $\delta_1$ around $p_0$ measured in the $m_0$-norm, and let $U_0\subset \mathcal X$ denote the ball of the same radius around $(\mathcal Z_0, \Phi_0)$, also measured in the $m_0$-norm.

\begin{lm}
\label{hypotheses1check} Hypothesis {\bf (I)} of Theorem \ref{nashmoser} holds for $\overline{\slashed {\mathbb D}}_p$, i.e.  there is an $m_0\geq 0$ such that for $\delta_0$ sufficiently small,  $p \in V_0$ that implies the  linearization 
$$\text{d}_{(\ph, \eta)}\overline{\slashed {\mathbb D}}_p: \mathcal  X\to \mathcal  Y$$ 
is invertible for $(\eta,\ph)\in U_0\cap \bold P_{\mathcal X}$. 

\end{lm}

\begin{proof}
Fix $m_0\geq 11$. We first investigate the obstruction component of the linearization. By decreasing $\delta_1$, we can ensure the $p\in V_0$ and $(\eta,\ph)\in U_0$ implies that the pullback of the parameter by the diffeomorphism $F_\eta$ defined in (\refeq{diffeos}) satisfies \be \label{petadef}p_\eta= F_\eta^*(g,B) \in V_0'\ee

\noindent where $V_0'$ is the ball of radius $\delta_0$ in the $m_0-1=10$-norm, hence Proposition \ref{Invertiblepreliminary} applies. 

 Using Corollary \ref{tamefirstderiv}, the $\text{\bf Ob}_{p_\eta}$-component of the linearization in the trivialization of Lemma \ref{Obtrivialization} acting on $(v,  \lambda, 0 )\in X_0'\oplus \R \oplus  X_0''$ may be written
\begin{eqnarray} \overline T_{p,\eta}&:=&\Xi_{p_\eta}\circ \Pi_{p_\eta} (\text{d}_{(\eta,\ph)}\overline{\slashed {\mathbb D}}_p(v,0, \lambda)) \nonumber \\ &=&\Xi_{p_\eta}\circ \Pi_{p_\eta} (\text{d}_{(\mathcal Z_0,\Phi_0)}\overline{\slashed {\mathbb D}}_0(v,0,\lambda) \ + \ldots+ \  F^2(\eta^2, v, \Phi_0+\ph)) \nonumber \\ &=& \overline T_{0,p}(v,\lambda) \ + \  \frak t_{p_\eta}(v,\lambda)\label{T+S}\end{eqnarray}

\noindent where $\overline T_{0,p}$ denotes the invertible map (\refeq{linearizationonobp}) from Proposition \ref{Invertiblepreliminary}, and $\frak t_{p_\eta}$ encompasses the error terms. Explicitly, via Corollary \ref{tamefirstderiv}) 
\begin{eqnarray}  \frak t_{p_\eta}(v,\lambda)&:=&\Xi_{p_\eta} \circ \Pi_{p_\eta}\Big[ \frak B_{\Phi_0,p}(v) +  \ (\mathcal B_{\ph} + \frak B_{\ph})(v)  \ +  \ M^1(\eta', v')\nabla (\Phi_0 + \ph) \nonumber \\ & & \ + \  M^2(\eta', v'')(\Phi_0 + \ph) \ \ + \  M^2(v', \eta'') (\Phi_0 + \ph)    \ + \  F^2(\eta^2, v, \Phi_0 + \ph) - \lambda \ph\Big].\label{allextraterms}\end{eqnarray} 

\noindent Since $\overline{T}_{0,p }: H^3\oplus \R \to H^{5/2}\oplus \R$ is invertible by Proposition \ref{Invertiblepreliminary}, we show that the perturbation $\frak t_{p_\eta}: H^{3}\oplus \R \to H^{5/2}\oplus \R$ is bounded, i.e. that there is a constant $C_1$ such that 
\be \|\frak t_{p_\eta}(v,\lambda)\|_{H^{5/2}}\leq C \delta_0  \Big (  \|v\|_{H^{3}} \ + \ |\lambda| \Big).\label{delta32regularity}\ee

\noindent holds for $(v,\eta)\in X_1'\oplus \R=H^3(\mathcal Z_0;N\mathcal Z_0)\oplus \R$.

 (\refeq{delta32regularity}) follows from the same argument as (\refeq{3/2regularityp}) in Proposition \refeq{Invertiblepreliminary}. Indeed, Proposition  \ref{fullynonlinear} shows that each term of $\frak t_{p_\eta}$ is of the form either $v'(t) \Theta_0(t_p, x_p,y_p)$ or  $v''(t) \Theta_1(t_p,x_p,y_p)$ just as in (\refeq{3/2regularityp}). Here, $\Theta_0, \Theta_1$ can be written more explicitly using the notation of Corollary \ref{fullynonlinear} as the sum of terms $m_{p_\eta}(y) \nabla (\Phi_0+\ph)$ and $m_0(y)\nabla \ph$ or  $m_{p_\eta}(y)(\Phi_0+\ph)$ and $m_0(y)\ph$, where $m_{p_\eta}, m_0$ are smooth endomorphisms bounded in terms of the $m_0$ norms of $p_\eta, p_0$ respectively. For the subterms of $\Theta_0, \Theta_1$ involving $\Phi_0$, the argument as in Proposition \refeq{Invertiblepreliminary} shows that there are pointwise bounds 

 $$ |\nabla_t^m \Theta_k|\leq Cr^{k-1/2} \|(p-p_0,\eta,\ph)\|_{m+m_0}$$

 \noindent for $k=1,2$. It then follows from Corollary \ref{CaseBC} (applied in the Fermi coordinates of $p_\eta$) that these terms have cornormal regularity at least $5/2$ for $v\in H^3$. For the subterms of $\Theta_0,\Theta_1$ involving $\ph$ rather than $\Phi_0$, the same applies using Lemma \ref{coefficientregularity} in place of the universal bounds (\refeq{polyhom2}) on the expansion of $\Phi_0$. (\refeq{delta32regularity}) follows, and we conclude that \refeq{T+S} is an isomorphism $\overline T_{p,\eta}:H^3\oplus \R \to H^{5/2}\oplus \R$.

For $m=1$, the full linearization acting on $(v,\lambda, \phi)$ now has the block-diagonal form 

\be \text{d}_{(\eta,\ph)}\overline{\slashed {\mathbb D}}_p = \begin{pmatrix} \overline{T}_{p_\eta}  & \pi_{p_\eta}  \\   \\  (1-\Pi_{p_\eta})\mathcal B& \overline {\slashed D}_{p_\eta}\end{pmatrix} \ : \  \begin{matrix} H^{3}(\mathcal Z_0; N\mathcal Z_0)\oplus \R \\ \oplus \\ r \mathcal H^{1,1}  \end{matrix} \ \  \lre \ \   \begin{matrix}  \text{\bf Ob}^{5/2}_{p_\eta}\\ \oplus \\  (\frak R_{p_\eta} \cap \mathcal H^{1,0}) \oplus  \R.\end{matrix}.\label{blockdecomp3}\ee

\noindent where $\mathcal B=\text{d}_{(\eta,\ph)}\overline{\slashed {\mathbb D}}_p- \overline {\slashed D}_{p_\eta}$ is the partial derivative with respect to $v$ (Cf. \refeq{petapullback}), and $\Pi_{p_\eta}$ is the $L^2$-orthogonal projection to $\text{\bf Ob}_{p_\eta}$, and \be \pi_{p_\eta}(\phi)=\br \slashed D_{p_\eta}\phi, \Pi_{p_\eta}\Phi_{p_\eta} \kt_{L^2} \Pi_{p_\eta}\Phi_{p_\eta}\label{pipdef}\ee

\noindent has rank 1 ($\Phi_{p_\eta}$ being the eigenvector from Lemma \ref{secondordereigenvalues} with parameter $p_\eta$ from \refeq{petadef}). In writing (\refeq{blockdecomp3}), we commit the minor abuse of notation of conflating $\overline T_{p_\eta}$ and $\text{ob}_{p_\eta}^{} \overline T_{p_\eta}$.  

If the rank $1$ component $\star$ did not appear, (\refeq{blockdecomp3}) would obviously be an isomorphism if $\mathcal B$ were bounded, since we have shown above that (\refeq{T+S}) is an isomorphism, and the bottom right entry is an isomorphism by construction. In fact, the presence of the rank 1 entry makes no difference, because $\pi_{p_\eta}=O(\delta_0)$; indeed, 
\bea \|\Pi_{p_\eta}\Phi_{p_\eta}\|_{5/2}=\|\slashed D_{p_\eta}^\star P_{p_\eta}\slashed D_{p_\eta}\Phi_{p_\eta}\|_{5/2}&=&\|\slashed D_{p_\eta}^\star P_{p_\eta}[\frak d_{p_\eta}(\Phi_{p_\eta})+\slashed D_0 (\Phi_{p_\eta}-\Phi_0) ]\|_{5/2} \\ &\leq & C\|(p-p_0, \eta)\|_{m_0}\leq C\delta_0.\eea

\noindent To conclude (\refeq{blockdecomp3}) is an isomorphism, it therefore only remains to show that $\mathcal B:H^3(\mathcal Z_0;N\mathcal Z_0)\to \mathcal H^{1,0}\oplus \R$ is bounded. Boundedness into $H^1_\text{b}$ is obvious; for the boundedness of $(r\del_r +\tfrac{1}{2})$ into $r^\nu L^2$, note that since $\ph \in r\mathcal H^{1,1}\cap \bold P_{\mathcal X}$ is polyhomogeneous with index set $\Z^++\tfrac12$, the operator $(r\del_r +\tfrac{1}{2})$ annihilates the order $r^{-1/2}$ term of $\nabla (\Phi_0+ \ph)$ and all other terms are $O(r^{1/2})$ so are integrable with the stronger weight in the normal directions. Since $\mathcal B$ consists of product of $\nabla(\Phi_0+\ph)$ or $\Phi_0+\ph$ with terms having integer Taylor expansions along $\mathcal Z_0$, $\mathcal B$ is likewise integrable in the higher weight, hence bounded.

If follows that (\refeq{blockdecomp3}) is an isomorphism for $m=1$. invertibility for higher $m$, thus on $\mathcal X, \mathcal Y$ follows from bootstrapping using the tame estimate in the next lemma.  
\end{proof}

\medskip

\begin{lm}
\label{hypotheses2check}

 Hypothesis {\bf (II)} of Theorem \ref{nashmoser} holds for $\overline{\slashed {\mathbb D}}_p$, i.e.  there are $s,s' \in \N$ such that the following estimates hold provided $\delta_0$ is sufficiently small:  for $p \in V_0$ and $(\ph,\eta)\in U_0$ the unique solution $u=(v,\phi,\lambda)$ of 
$$\text{d}_{(\ph, \eta)}\overline{\slashed {\mathbb D}}_p u=f$$ 
obeys the tame estimate
\be  \|u\|_{m}\leq C_m \ \Big( \|f\|_{m+s} + \|(p-p_0,\eta, \ph)\|_{m+s'} \|f\|_{m_0} \Big).
\ee
uniformly over $V_0\times (U_0\cap \bold P_\mathcal  X)$ for all $m\geq m_0$.

\end{lm}

\begin{proof}
We claim that it suffices show that there are tame elliptic estimates of the following form for $\overline T_{p_\eta}  $ and $\overline{\slashed D}_{p_\eta}$ individually: if $\overline {T}_{p_\eta}v=f_0$ and $\overline{\slashed D}_{p_\eta}\phi=f_1$ then 

\begin{eqnarray}
\|(v,\lambda)\|_{m}&\leq& C_m \Big( \|f_0\|_{m+3/2 }   \ + \  \|(p-p_0,\ph,\eta)\|_{m+s'} \| f_0 \|_{m_0} \Big) \label{Tptame}\\ 
\|\phi \|_{m}&\leq& C_m \Big( \|f_1\|_{m }     \ + \   \|(p-p_0,\ph,\eta)\|_{m+s'} \| f_1 \|_{m_0} \Big) \label{Dptame} 
\end{eqnarray}

\noindent for $m_0=11$. 

Indeed, given (\refeq{Tptame}--\refeq{Dptame}), one concludes the lemma as follows: write $f=(f_0, f_1)$, so that by the decomposition (\refeq{blockdecomp3}) one has $\overline{T}_{p_\eta}(v)= f_0 - \pi_{p_\eta}(\phi)$.
Applying (\refeq{Tptame}) shows   

\begin{eqnarray}
\|(v,\lambda)\|_{m} & \leq& C_m \Big( \|f_0 - \pi_{p_\eta}(\phi)\|_{m+3/2}  \ + \   \|(p-p_0, \ph,\eta)\|_{m+s'} \|f_0-\pi_{p_\eta}(\phi)\|_{m_0}\Big) \\ 
& \leq& C_m \Big( \|f_0 \|_{m+3/2} \  + \  \|\pi_{p_\eta}\phi\|_{m+3/2} \\  &  &\ + \  \  \|(p-p_0, \ph,\eta)\|_{m+s'} \|f_0\|_{m_0} \ + \  \|(p-p_0, \ph,\eta)\|_{m+s'} \|\pi_{p_\eta}(\phi)\|_{m_0}\nonumber \\ 
 &\leq & C_m \Big( \|f_0 \|_{m+3/2} \  +  \  \|(p-p_0,\ph,\eta)\|_{m+s'} \|f\|_{m_0}  \Big)  \label{tosubstitutevlambda}  \end{eqnarray} 

\noindent In the last step, we have used that the 1-dimensional image of $\pi_{p_\eta}$ obeys 

$$\|\pi_{p_\eta}(\phi)\|_{m+3/2} \leq \|\phi\|_{1} \ \cdot \  \| \Pi_{p_\eta}\Phi_{p_\eta}\|_{m+3/2} \leq C_m\|\phi\|_{1}  \cdot \|(p-p_0, \eta,\ph)\|_{m+4} $$

\noindent by Cauchy-Schwartz on (\refeq{pipdef}), and elliptic bootstrapping of the eigenvector $\Phi_{p_\eta}$ using the second-order operator from Lemma \ref{secondordereigenvalues}, and $\|\phi\|_1\leq \|f\|_{m_0}$ by Lemma \ref{hypotheses1check}.  

Similarly, for the second component, (\refeq{Dptame}) shows 
\bea
{\|\phi\|_{m} }&=& C_m \Big( \|f_1 -(1-\Pi_{p_\eta})\mathcal B(v,\lambda) \|_{m} \  +  \  \|(p-p_0,\ph,\eta)\|_{m+s'} \|f_1-(1-\Pi_{p_\eta})\mathcal B(v,\lambda)\|_{0} \Big)  \\ 
 &=& C_m \Big( \|f_1 \|_{m} \  +  \  \|(v,\lambda)\|_{m+s} +  \|(p-p_0,\ph,\eta)\|_{m+s'} \|f\|_{m_0} + \|(p-p_0,\ph,\eta)\|_{m+s'} \|(v,\lambda)\|_{4}   \Big).   
\eea 

\noindent where we have used that there is a (tame) boundedness estimate \be \|(1-\Pi_{p_\eta})\mathcal B(v,\lambda)\|_{m}\leq C_m \Big( \|(v,\lambda)\|_{m+s}+ \|(p-p_0,\eta,\ph)\|_{m+4} \|(v,\lambda)\|_{m_0}\Big)\label{tameprojection}\ee
\noindent for the range components. Such an estimate follows from showing $(1-\Pi_{p_\eta}), \mathcal B$ are individually bounded tame maps; the first of these is Lemma \ref{calHprojections} (in which the boundedness is easily seen to tame), and the second follows from interpolation and Young's inequality (cf. the subsequent Lemma \ref{hypothesis3check}).  Substituting the previous estimate  (\refeq{tosubstitutevlambda}) on $\|(v,\lambda)\|_m$ and using that $\|(v,\lambda)\|_{4}\leq \|g\|_{m_0}$ by Lemma \ref{hypotheses1check}  then shows that 
\bea \|(v,\phi,\lambda)\|_m &\leq& C_m \Big( \|f_1 \|_{m+3/2} \  + \   \|(p-p_0,\ph,\eta)\|_{m+s'} \|f\|_{m_0}  \Big) \eea
as desired. 

To complete the lemma, we now prove (\refeq{Tptame}) and (\refeq{Dptame}). The latter follows from differentiating elliptic estimates in the standard way. To elaborate briefly, we begin with the estimate for $\slashed D_{p_\eta}$ and the $H^m_\text{b}$ term in the norms. One shows by iterating commutators that there is an elliptic estimate of the form 
\be
\|\phi\|_{rH^{m,1}_{\text{b},e}}  \leq C_m \Big ( \|\overline{\slashed D}_{p_\eta} \phi\|_{H^m_\text{b}}  \ + \  \|(p-p_0,\eta,\ph)\|_{s'} \|\phi\|_{rH^{m-1,1}_{\text{b},e}} \ + \  \ldots \  +  \  \|(p-p_0,\eta,\ph)\|_{m+ s'} \|\phi\|_{rH^1_e} \Big)
\ee
for each $m$ and $s'< m_0$. Given such an estimate, the $k^{th}$ middle term can be absorbed into the $k=0,m$ ones by Young's inequality and interpolation with $m_2=m+s'$ and $m_1=s'$ on the first factor and $m_2=m-1$ and $m_2=0$ on the second factor. The tame estimates are then a consequence of induction by substituting the tame estimate on  $ \|\phi\|_{rH^{m-1,1}_{\text{b},e}}$ beginning with the base case provided by Lemma \ref{hypotheses1check}, and using that $\|(p-p_0,\eta,\ph)\|_{s'}\leq 1$.  The same argument applies for the spaces $r\mathcal H^{m,1}$ and $\mathcal H^{m,0}$ using the elliptic estimate and commutation relations from Lemma \ref{DoncalH}. (\refeq{Dptame}) follows.  

Similarly, for (\refeq{Tptame}) it suffices to show 
\be
\|v\|_{{m+2}}  \leq C_m \Big ( \|\overline{T}_{p_\eta} v\|_{{m+3/2}}  \ + \  \|(p-p_0,\eta,\ph)\|_{s'} \|v\|_{{m-1+3/2}} \ + \  \ldots \  +  \  \|(p-p_0,\eta,\ph)\|_{m+ s'} \| v\|_{3/2} \Big),\label{Tpderivs}
\ee
and applying the same combination of interpolation and Young's inequality. (\refeq{Tpderivs}) follows from iterating commutators again for each term of (\refeq{T+S}). 
To use the term $\mathcal B_{\Phi_0}(v)$ as an example, one has $$\mathcal B_{\Phi_0}(\nabla_t^m v)= \nabla_t^m \mathcal B_{\Phi_0}(v) + (\nabla_t \mathcal B_{\Phi_0}) (\nabla_t^{m-1}v)\ +   \ldots  +  \  (\nabla_t^m\mathcal B_{ \Phi_0})(v).$$

\noindent By the same argument as in Proposition \ref{Invertiblepreliminary} and Lemma \ref{hypotheses1check}, $v\in H^{m+3/2}$ implies that the first term has conormal regularity $3/2$, while each of the remaining terms has conormal regularity at least $5/2$. Proceeding now by induction, assume the estimate holds for $m-1$, and applying the $m=1$ estimate from Lemma \ref{hypotheses1check} leads to $\nabla_t^mv$: 
\bea
\|\nabla_t^mv\|_{2} & \leq &C_m \Big ( \|\Xi_{p_\eta}^{-1}\Pi_{p_\eta}(\nabla_t^m \mathcal B_{\Phi_0}(v) + \ldots + \nabla_t^m F^2(\eta^2, v, \ph))\|_{{3/2}} \\ & &  \ + \  \|(p-p_0,\eta,\ph)\|_{s'} \|v\|_{{m-1+3/2}} \ + \  \ldots \  +  \  \|(p-p_0,\eta,\ph)\|_{m+ s'} \| v\|_{3/2} \Big)
\eea
\noindent where the induction hypothesis has been applied to all but the first term. As in (\refeq{tameprojection}), the projection $\Pi_{p_\eta}$ behaves in a tame fashion, and it is easy to check from the construction in Section 4 (e.g. Lemma \ref{Utame}) that Corollary \ref{regularitycoincides} also behaves tamely. Commuting $\nabla_t$ past $\Xi_{p_\eta}^{-1}\Pi_{p_\eta}$ therefore contributes to the lower order terms, and we conclude 

\bea
\|\nabla_t^mv\|_{2} &\leq & C_m \Big ( \| \nabla_t^m \overline T_{p,\eta}(v)\|_{{3/2}}   \ + \  \|(p-p_0,\eta,\ph)\|_{s'} \|v\|_{{m-1+3/2}} \ + \  \ldots \  +  \  \|(p-p_0,\eta,\ph)\|_{m+ s'} \| v\|_{3/2} \Big)  \\
&\leq & C_m \Big( \|f_0\|_{m+3/2 }   \ + \  \|(p-p_0,\ph,\eta)\|_{m+s'} \| f_0 \|_{m_0} \Big)
\eea
which yields (\refeq{Tptame}).
\end{proof}

\medskip

\begin{lm}
\label{hypothesis3check}

 Hypothesis {\bf (III)} of Theorem \ref{nashmoser} holds for $\overline{\slashed {\mathbb D}}_p$, i.e.  there are $r,r' \in \N$ such that the following holds provided $\delta_0$ is sufficiently small: for $p\in V_0$ and $(\ph,\eta)\in U_0$, the second derivative obeys  the tame estimate
\be  \|\text{d}^2_{(\eta,\ph)} \overline{\slashed {\mathbb D}}_p(u, \upsilon)\|_{m}\leq C_m \ \Big( \|u\|_{m+r}\|\upsilon \|_{m_0} \ + \  \|u\|_{m_0}\|\upsilon \|_{m+r}  \ + \ \|u\|_{m_0}\|\upsilon \|_{m_0}\cdot (1+ \|(p,\eta,\ph)\|_{m+r'})\Big).
\ee
for $u,\upsilon \in X$ uniformly over $V_0\times (U_0\cap \bold P_\mathcal X)$ for all $m\geq m_0$. 
\end{lm}

\begin{proof}
This tame estimate follows directly from using the boundedness of the terms comprising $\text{d}^2_{(\eta,\ph)}\overline{\slashed {\mathbb D}}_p$ in conjunction with the interpolation inequalities.

As in Corollary \ref{tamesecondderiv}, the second derivative is given by 
\bea
\text{d}^2_{(\eta,\ph)}\slashed {\mathbb D}_p\Big((v,\phi), (w,\psi)\Big) &= & \   \ (\mathcal B_{\psi}+\frak B_\psi)(v) \ + \ (\mathcal B_{\phi}+\frak B_\phi)(w)   \\
& & \ + \ M^1(w', v')\nabla(\Phi_0 + \ph)  \ + \  M^1(\eta', v')\nabla \psi  \ + \  M^1(\eta', w')\nabla \phi \\ 
& & \ + \  M^2(w', v'')(\Phi_0 + \ph) \  + \  M^2(\eta', v'')\psi  \ +  \ M^2(v', w'') (\Phi_0 + \ph) \\ & &  \   + \ M^2(v', \eta'')\psi \ +  \ M^2(\eta', w'')\phi \  +  \ M^2(w', \eta'')\phi\\ & & \ + \ F^3(\eta^2, w,  \ph) \  + \  F^4(\eta^2, v, \psi) \  + \  F^5(\eta, v,w, \Phi_0+\ph)
\eea

\noindent For the sake of the proverbial deceased horse, we will prove the lemma for the term $M^2(w', v'')(\Phi_0 + \ph)$; it is straightforward to verify that the same argument applies equally well to the remaining terms. 

To begin, we bound that $H^m_\text{b}$-term in the norm. By Proposition \ref{fullynonlinear} Item (C) part (iii), this term is itself a sum of terms of the form $m_p(y) w' v'' (\Phi_0 + \ph)$. Differentiating the part involving $\ph$ of such a term, 

\bea
\|\nabla_\text{b}^m  (m_p(y) w' v''  \ph)\|_{L^2} &\leq   &  C_m \sum_{0\leq k \leq m} \|\nabla_\text{b}^k (v' w''  ) \nabla_\text{b}^{m-k}(m_p \ph))\|_{L^2} \\ 
&\leq &  C_m \sum_{0\leq k \leq m} \|\nabla_\text{b}^k (v' w''  ) \|_{L^{2,2}(S^1)}\| \nabla_\text{b}^{m-k}(m_p \ph)\|_{H^2_\text{b}} \\ 
 \ &\leq &  C_m \sum_{0\leq k \leq m} \| v' w''   \|^{1-\tfrac{k}{m}}_{H^2(S^1)} \ \| v' w''   \|^{\tfrac{k}{m}}_{H^{m+2}(S^1)}  \   \| m_p  \ph \|^{1-\tfrac{k}{m}}_{H^2_\text{b}}\ \| m_p \ph \|^{\tfrac{k}{m}}_{H^{m+2}_\text{b}} \\ 
\eea
where we have used the Sobolev embedding $C^0\hookrightarrow H^2(S^1)$ and then the interpolation inequalities with $m_2 = m+2, m_1= 2$. By Young's inequality with exponents $p= \tfrac{m}{k}$ and $q=\tfrac{m}{m-k}$, one finds the above is bounded by 
\bea
{\phantom {\ } } &\leq  &  C_m\left(  \ \| v' w''   \|_{L^{m+2,2}(S^1)}  \| m_p  \ph \|_{H^{2}(S^1)}  \ + \   \| v' w''   \| _{H^{2}(S^1)} \| m_p \ph \|_{H^{m+2}_\text{b}}\right) \\
 &\leq & C_m\Big(  \ \| v'   \|_{H^{m+4}(S^1)}  \| w'' \|_{H^{4}(S^1)}  \ + \ \| v'   \|_{H^{4}(S^1)}  \| w'' \|_{H^{m+4}(S^1)} \   \\  & & 
 + \   \| v'  \| _{H^{4}(S^1)}\| w'' \| _{H^{4}(S^1)}\left( \| m_p \|_{H^{m+4}_\text{b}}\| \ph \|_{H^{4}_\text{b}}+ \| m_p \|_{H^{4}_\text{b}}\| \ph \|_{H^{m+4}_\text{b}}\right)\Big) \\ 
 &\leq &  C_m\Big(  \ \| v   \|_{H^{m+5}(S^1)}  \| w \|_{H^{6}(S^1)}  \ + \  \| v   \|_{H^{5}(S^1)}  \| w \|_{H^{m+ 6}(S^1)} \ + \    \| v  \| _{H^{5}(S^1)}\| w'' \| _{H^{6}(S^1)} \ \cdot \   \|(p, \eta, \ph)\|_{m+6}\Big) \
\eea
\noindent where we have repeated the interpolation and Young's steps from above with $m_2=m+4$ and $m_1=4$ on both products, and then used the fact that $6\leq m_0$ so that $\|m_p\|_{H^4_\text{b}} +  \|\ph\|_{H^4_\text{b}}
\leq C$.  This shows the desired estimate for $r, r'=6$. The same steps apply to the $r^{\nu}H^{m-1}_\text{b}$ term in the norms (Definition \ref{spaces}) using the commutation relations from Lemma \ref{DoncalH}. The other terms are similar, with the constant term  in $(1+ \|(p,\eta,\ph)\|_{m+r'})$ on the right hand side arising from the terms not involving $(p,\ph,\eta)$ such as $\mathcal B_\psi (v)$. \end{proof}

\subsection{Proofs of Theorem \ref{mainb} and Corollary \ref{maind}}
\label{section8.5}

In this subsection, we invoke the Nash-Moser Implicit Function Theorem \refeq{nashmoser} to conclude the proofs of Theorem \ref{mainb} and Corollary \ref{maind}, beginning with the latter. 

\begin{proof}[Proof of Corollary \ref{maind}] Lemmas \ref{hypotheses1check}, \ref{hypotheses2check}, and \ref{hypothesis3check} verify respectively that hypotheses {\bf (I)}, {\bf (II)}, and {\bf (III)} from Section \ref{NM1} hold on $V_0 \times (U_0 \cap \bold P_\mathcal X)$. Lemma \ref{errortermslem}  (which extends easily to the spaces $\mathcal H^{m,0}$ from Lemma \ref{spaces}) shows that $f_p$ obeys $\|f_p\|_{m}\leq C \|p\|_{m+s}$. Moreover, since term in $f_p=\slashed D_p\Phi_0$ is the product of functions that are smooth across $\mathcal Z_0$ with the polyhomogeneous $\Phi_0$ and its derivatives (cf. Appendix \ref{appendixC}), one has that $f_p\in \bold P_\mathcal Y$ is also polyhomogeneous with index set $\Z^+-\tfrac12$.

 It remains to show that the property (P) of being polyhomogeneous is propagated by the iteration in the sense of Definition \ref{propagated}. Lemma \refeq{spaces} and its proof in Appendix \ref{appendixII} show that the smoothing operators $S_\e, S_\e^\text{b}$ preserve polyhomogeneity.
 The argument above that $f_p\in \bold P_\mathcal Y$, in fact applies equally well to show that $\slashed D_{p_\eta}(\Phi_{0} +\ph)$ is polyhomogeneous for any pullback parameter $p_\eta=F^*_\eta p$, provided $\ph$ is polyhomogeneous with index set $\Z^++\tfrac12$. This is to say that \be \ph \in \bold P_\mathcal X  \ \ \Rightarrow \ \  \overline{\slashed {\mathbb D}}_p(\eta, \Phi_0+\ph) \in \bold P_\mathcal Y\label{propagated2}\ee  preserves polyhomogeneity. To show polyhomogeneity (P) is propagated, we therefore verify that \bea  \ph \in \bold P_\mathcal X   \text{ \ , \ } f\in \bold P_\mathcal Y \ \  \Rightarrow \ \  (\text{d}_{(\eta,\ph)}\overline{\slashed {\mathbb D}}_p)^{-1}f \in \bold P_\mathcal X,\eea

 \noindent this being the third requirement in Definition \ref{propagated}. 

Suppose that $\ph\in \bold P_{\mathcal X}, f\in \bold P_\mathcal Y$ is polyhomogeneous with index set $\Z^+-\tfrac12$, and suppose that $$(\phi, v, \lambda)=(\text{d}_{(\eta,\ph)}\overline{\slashed {\mathbb D}}_p)^{-1}f \in \mathcal X$$ is the unique solution guaranteed by Lemma \ref{hypotheses1check}. By the block-diagonal decomposition (\refeq{blockdecomp3}) from Lemma \ref{hypotheses1check}, this solution obeys 
\begin{eqnarray} \overline T_{p_\eta}(v) + \lambda c_1 \Pi_{p_\eta}\Phi_{p_\eta}&=& \Pi_{p_\eta}f \\ 
(1-\Pi_{p_\eta})\mathcal B_{\ph}(v) + \overline {\slashed D}_{p_\eta}\phi&=&(1-\Pi_{p_\eta})f \label{blockdiagonalpropagated}
\end{eqnarray}

\noindent where $c_1=\br\slashed D_{p_\eta}\phi, \Phi_{p_\eta}\Phi_{p_\eta} \kt_{L^2}$. The projection $\Pi_{p_\eta}$ preserves polyhomogeneity, since \be \Pi_{p_\eta}=\slashed D_{p_\eta}P_{p_\eta} \slashed D_{p_\eta}^\star\label{projector}\ee

\noindent (by Corollary \ref{higherregularitymappings} or Proposition \ref{calHprojections}) is the composition of three operators, all of which individually preserve polyhomogeneity. That each of these preserves polyhomogeneity can be seen by differentiating or solving the expansions term by term (cf. \cite[Prop. 7.17]{MazzeoEdgeOperators} for a general proof). Just as in the proof of (\refeq{propagated2}), $\mathcal B_\ph(v)$ consists of products of functions that are smooth across $\mathcal Z_0$ with the polyhomogeneous $\ph, \nabla \ph$. Rearranging (\refeq{blockdiagonalpropagated}), 
\be  \overline {\slashed D}_{p_\eta}\phi = (1-\Pi_{p_\eta})f \ - \ (1-\Pi_{p_\eta})\mathcal B_{\ph}(v)\label{restrictedpropagated}\ee
\noindent and all the terms on the right side are therefore polyhomogeneous with index set $\Z^+-\tfrac12$. This implies, again solving term by term (cf. \cite[Prop. 7.17]{MazzeoEdgeOperators}) that $\phi$ is polyhomogeneous with index set $\Z^++\tfrac12$, with the caveat that we may {\it a priori} have a logarithm term on the $r^{1/2}$ coefficient. 

To rule out the appearance of logarithm terms with radial dependence $r^{1/2}\log(r)$, we investigate the first term of the polyhomogeneous expansion, which is obtained by formally solving this initial term. In particular, the non-appearance of logarithm terms is a consequences of a restriction on the $\theta$-Fourier modes that appear with the $r^{1/2}$ coefficient. To elaborate, let $(t,r,\theta)$ denote the polar Fermi coordinates (Definition \ref{fermi}) of the metric in the parameter $p_\eta$. Logarithm terms $e^{ik\theta}r^{1/2}\log(r)$ would arise from the right-hand side having terms $r^{-1/2}e^{\pm 3i\theta/2}$. Proceeding by induction, suppose  that $\ph$ hadsno such term at the $N^{th}$ iteration, and we claim no such term can then appear at the $(N+1)^{st}$ stage. Indeed, if $\ph$ has leading order terms $e^{\pm i\theta/2}r^{1/2}$, then the error $f_N$ at this stage also has only leading order terms $e^{\pm \theta/2}r^{-1/2}$ (these leading order terms arise as the leading order of $\nabla_z\ph,\nabla_{\overline z}\ph$ times smooth functions of $t$ via (\refeq{propagated2})). The terms from $\mathcal B_{\ph}$ obey the same restriction, and this restriction is respected by the projection $\Pi_{p_\eta}$ since the same argument applies to each of the three operators (\refeq{projector}). It follows that all the terms on the right side in (\refeq{restrictedpropagated}) have leading order $e^{i\pm \theta/}r^{-1/2}$. Formally solving shows that the correction $\phi$ has the same leading order terms as $\ph$. Since the smoothing operators $S_\e^\text{b}$ were construction to preserve this property (see Appendix \ref{appendixII}), this closes the induction. We conclude that the property (P) of being polyhomogeneous in the sense of having expansions of the form (\refeq{polyhomZ+12})-(\refeq{polyhomZ-12}) without logarithm terms is propagated. 

By the  Nash-Moser Implicit Function Theorem \refeq{nashmoser}, there is therefore an open neighborhood $\mathscr V_0\subset \mathcal P$ of smooth parameters such that for $p\in \mathscr V_0$ there exists a unique solution $(\mathcal Z_p, \Phi_p , \Lambda_p)$ to the equation \be {\slashed {\mathbb D}_p(\mathcal Z_p,\Phi_p)}=\Lambda_p \Phi_p\label{tobeshown}\ee  
and the triples $(\mathcal Z_p, \Phi_p , \Lambda_p)$ define a smooth tame graph over $\mathscr V_0 \times \{0\}\subset \mathscr V_0\times \mathcal  X$. This completes the proof of Corollary \refeq{maind} in the presence of Assumption \ref{assumption5}. In the absence of Assumption \ref{assumption5}, the standard Kuranishi framework (see, e.g. Section 3.3 of \cite{DWDeformations}) applies to show that the set of parameters for which (\refeq{tobeshown}) holds is described by the zero set of a smooth tame map $$\kappa_p : \mathscr V_0\times \R^n\to \R^n$$
where $n=\dim(\ker(T_{\Phi_0}))$ is the dimension of the kernel of the index 0 map from Section \ref{section6.3}. 
 \end{proof}
 
 \bigskip

 \begin{proof}[Proof of Theorem \ref{mainb}] The projection $\pi(\mathscr M_{\Z_2})\subseteq \mathscr V_0 \cap \mathcal P$ of the universal moduli space of $\Z_2$-harmonic spinors to the parameter space is defined by the zero-set $$\pi(\mathscr M_{\Z_2})=\Lambda^{-1}(0)\cap \mathscr V_0$$
 of the eigenvalue $\Lambda: \mathscr V_0\to \R$ in Corollary \ref{maind}, and there is locally a unique $\Z_2$-harmonic spinor $(\mathcal Z_p, \Phi_p)$ up to normalization and sign for each $p\in \Lambda^{-1}(0)$, hence the projection $\pi$ is a local homeomorphism. 
 
To conclude the theorem, we show that Assumption \ref{assumption5} implies that the map $\Lambda: \mathscr V_0\to \R$ is transverse to 0. To see this, let $p(s)$ be a path of parameters with $p(0)=p_0$ to be specified momentarily. By Corollary \ref{maind}, such a path implicitly defines triples $(\mathcal Z_s, \Phi_s, \Lambda_s)$ satisfying (\refeq{tobeshown}) for $s$ sufficiently small. Differentiating (\refeq{tobeshown}) at $s=0$ yields the relation that  \be \left(\d{}{s}\Big |_{s=0}\slashed D_{\mathcal Z(s), p_0}\right) \Phi_0 + \left(\d{}{s}\Big|_{s=0} \slashed D_{\mathcal Z_0, p(s)}\right)\Phi_0 +\slashed D_0 \dot \Phi =\dot \Lambda \Phi_0\label{lambdavariation}\ee

\noindent where $\cdot$ denotes the $s$-derivative at $s=0$. We now choose $p(s)=(g(s), B(s))$ so that the derivative $(\dot g, \dot B)$ has the following properties. Let $\dot B$ be a smooth perturbation supported on a neighborhood disjoint from $N_{r_0}(\mathcal Z_0)$ such that $\br \gamma(\dot B) \Phi_0, \Phi_0\kt\neq 0$. Given this, we define $\dot g$ in terms of $\dot B$ as follows. By Assumption \ref{assumption5}, we know that $\text{ob}_0\circ T_{\Phi_0}: H^2(\mathcal Z_0; N\mathcal Z_0)\to \text{ker}(\slashed D_0|_{L^2})$ is injective, and the closure of its range has 1-dimensional orthogonal complement. Let $\Phi_1$ be the normalized spinor whose span is this 1-dimensional space (note that in general $\Phi_0\neq \Phi_1$, since $T_{\Phi_0}$ need not have image orthogonal to $\Phi_0$).  Decompose $\Pi_0(\gamma(\dot B)\Phi_0)=(c_0\Phi_1, \xi)$, and take the path $g(s)$ such that $\dot g=\dot g_\eta$ where $\text{ob}_0\circ T_{\Phi_0}(\eta)=-\xi$, so that $$\Pi_0(\mathcal B_{\Phi_0}(\eta))=(0, -\xi)\in \R \Phi_1 \oplus \text{Im}(T_{\Phi_0})\simeq \text{\bf Ob}_0$$ 
where $\text{\bf Ob}_0=\text{\bf Ob}(\mathcal Z_0)$ is as in Section \ref{section4}. Moreover, since $\overline T_{\Phi_0}$ is an isomorphism by Assumption \ref{assumption5}, one has $\br\Phi_1, \Phi_0 \kt\neq 0$. Since $\dot B$ is smooth and supported away from $\mathcal Z_0$, $\eta \in C^\infty(\mathcal Z_0;N\mathcal Z_0)$ by Item (B) of Corollary \ref{CaseBC}. By design, the first stage of iteration in this case now requires no correction to $\mathcal Z_0$, hence $\mathcal Z(s)=O(s^2)$ (no smoothing is needed in this first stage as $\eta$ is already smooth). Taking the inner product of (\refeq{lambdavariation}) with $\Phi_1$ then yields

  \be\br\cancel{ \tfrac{d}{ds}|_{s=0}\slashed D_{\mathcal Z(s),p_0} \Phi_0, \Phi_1} \kt + \br -c\Phi_0, \Phi_1\kt +\cancel{\br \slashed D_0 \dot \Phi, \Phi_1\kt } =\dot \Lambda \br\Phi_0, \Phi_1\kt \label{lambdavariation}\ee

\noindent wherein the first term vanishes because $\mathcal Z(s)=O(s^2)$, and the third via integration by parts since $\Phi_1\in \ker(\slashed D_0|_{L^2})$. We conclude that $\dot \Lambda\neq 0$ and that $\Lambda: \mathscr V_0\to \R$ is transverse to 0 at $p_0$, hence on $\mathscr V_0$ after possibly intersecting with a smaller open set. 
 
 When Assumption \ref{assumption5} fails, the standard Kuranishi framework (Section 3.3 of \cite{DWDeformations}) applies again to show that the $\pi(\mathscr M_{\Z_2})$ is given locally by the zero-set of a smooth tame map 
 
 $$(\kappa_p,\Lambda) : \mathscr V_0\times \R^n\to \R^n\times \R$$
 where $n=\dim(\ker(T_{\Phi_0}))$ is the dimension of the kernel of the index 0 map from Section \ref{section6.3}. This concludes the proof of Theorem \ref{mainb}. 
 \end{proof}
\appendix 
\section{Exponential Decay}
\label{appendixI}

This appendix proves Lemma \ref{discretedecay}, which was used to establish the exponential decay estimates in Section \ref{section4}.  To recall the notation, $\slashed D_\circ$ here denotes the Dirac operator on $Y_\circ=(S^1\times \R^2, dt^2 + dx^2 + dy^2)$. The lemma asserted that given a solution \be \slashed D_\circ \slashed D_\circ u_\ell = f_\ell\label{secondordereqB}\ee \noindent of the second order equation where $f_\ell$ satisfies the hypotheses (\refeq{409}--\refeq{expBnl}), then the solution $u_\ell$ satisfies (\refeq{expdecayzetalemma}). The proof relies on a discretized version of the maximum principle based on similar found in \cite[A.2.1]{SivekTaubesECH}. 
 
 \begin{proof}
 The integration by parts arguments of Proposition \ref{semifredholm} and \ref{secondorderoperator} holds equally well on the non-compact $Y_\circ$, since the boundary term at $r\to \infty$ vanishes (and the $L^2$-term is compactly supported). In fact, since $B_0=0$ and the product metric is flat, the Weitzenb\"ock formula implies that $rH^1_e$-kernel is empty. It follows, as in Lemma \ref{secondorderoperator} that $$\slashed D_\circ\slashed D_\circ: rH^1_e(Y_\circ\setminus \mathcal Z_0)\lre rH^{-1}_e(Y_\circ\setminus \mathcal Z_0)$$ is an isomorphism, hence (\refeq{secondordereqB}) admits a unique solution. Since $\slashed D_\circ$ preserves Fourier modes, $u_\ell$ automatically satisfies the same Fourier mode restriction as  $f_\ell$ in (\refeq{409}), and it suffices to prove the decay estimate.

Recall that $A_{n\ell}$ denotes the sequence of annuli (\refeq{annuli}) from Part (B) of Proposition \ref{cokproperties}. Let $\chi_n$ be a cutoff-function equal to $1$ on $A_{n_\ell}$ such that $$\text{supp}(d\chi_n)\subseteq A_{(n-1) \ell} \cup A_{(n+1) \ell} \hspace{2cm} |d\chi_{n}|\leq \frac{c|\ell|}{R_0}.$$
  
  \noindent Taking the inner product of (\refeq{secondordereqB})  with $\chi_n^2 u_\ell$, integrating by parts, and using Young's inequality yields the following for universal constants $c_1,c_2$, where we denote $B_{n\ell}=  A_{(n-1) \ell} \cup A_{n \ell}\cup A_{(n+1) \ell}$ as in the statement of the lemma:

  \begin{eqnarray} 
 \int_{B_{n\ell}}\chi_n^2 |\slashed D_\circ{u_\ell}|^2 \ dV_\circ  &=& -\int_{B_{n\ell}} \br 2\chi_n d\chi_n.  u_\ell, \slashed D_\circ u_\ell \kt \ dV_\circ +  \int_{B_{n\ell}} \br  \chi_n^2 u_\ell, -f_\ell \kt  \ dV_\circ\label{410}\\
  &\leq & 2c_1^2 \frac{|\ell|^2 }{R_0^2} \|u_\ell\|_{L^2(B_{n\ell})} + \frac{1}{2}    \|\chi_n \slashed D_\circ{u_\ell}\|^2_{L^2(B_{n\ell})} \label{411} \\ & & + \frac{1}{2c_2}\|f_{\ell}\|^2_{rH^{-1}_e(B_{n\ell})}    +\frac{c_2}{2}  \|\chi_n^2 u_\ell\|^2_{rH^1_e(B_{n\ell})}\label{412}
  \end{eqnarray}
  
 \noindent where the inner product and volume form $dV_\circ$ are defined using the product metric $g_\circ$. 
 
The $\|\chi_n\slashed D_\circ u_\ell\|^2$ term on the right may be absorbed into the same term on the left because of the factor of $1/2$. Similarly, by choosing $c_2$ small, the $\|\chi_n^2u_\ell\|_{rH^1_e}$ term may be absorbed into the others as follows. The elliptic estimate for $\slashed D_\circ: rH^1_e\to \text{Range}(\slashed D_\circ)$ applied to $\chi_n^2 u_\ell$ shows
 
 \be  \frac{c_2}{2} \|\chi_n^2 u_\ell\|^2_{rH^1_e(B_{n\ell})}\leq \frac{c_2C}{2} \|\slashed D_\circ (\chi_n^2 u_\ell) \|_{L^2(Y_0)}^2 \ \leq \frac{c_2C }{2} \left(\ \|\chi_n^2\slashed D_\circ u_\ell \|_{L^2(B_{n\ell})}^2 +\|  d\chi_n.u_\ell\|^2_{L^2(B_{n\ell})} \right). \label{413}\ee
  By choosing $c_1$ sufficiently small (and using that $|\chi_n^2|\leq |\chi_n|$), the $\|\chi_n^2\slashed D_0u_\ell\|^2$ term can again be absorbed on the left hand side of (\refeq{410}), and the $\|d\chi_n.u_\ell\|^2$ term can be absorbed into (\refeq{411}) by increasing $c^2$. 
  
Applying the elliptic estiamte (\refeq{413}) again without the factor of $c_2$ and substituting the result of (\refeq{410}--\refeq{413}) yields  
   \bea \|u_\ell\|^2_{rH^1_e(A_{n\ell})} \leq \|\chi_n^2 u_\ell\|^2_{rH^1_e(B_{n\ell})} & \leq &  C_1 \frac{|\ell|^2 }{R_0^2} \|u_\ell\|_{L^2(B_{n\ell})}  +\frac{1}{2c_1} \|f_\ell\|^2_{ rH^{-1}_e(B_{n\ell})}\eea which shows, invoking the assumption on $f_\ell$, that 
\begin{eqnarray}  \|u_\ell\|^2_{rH^1_e(A_{n\ell})}&\leq &   \frac{C_1|\ell|^2 }{R_0^2} \|u_\ell\|^2_{L^2(B_{n\ell})}  +\frac{C'_m}{|\ell|^{2+2m}}e^{-2 n/c_m } \label{eq2.4lm},\end{eqnarray}

\noindent for a universal constant $C_1$. In addition, the Fourier mode restriction on $u_\ell$ means that the first term on the right obeys $$  \frac{C_1|\ell|^2 }{R_0^2} \int_{B_{n\ell}} |u_\ell|^2 \ dV  \leq  \frac{4C_1}{R_0^2}\left( \int_{A_{(n-1)\ell}} |\nabla u_\ell|^2 \ dV+  \int_{A_{n\ell}} |\nabla u_\ell|^2 \ dV+  \int_{A_{(n+1)\ell}} |\nabla u_\ell|^2 \ dV\right). $$

Now set $\frak u_n= \|u_\ell\|^2_{rH^1_e (A_{n\ell})}$, and choose $R_0$ so that $4 C_1/R_0^2 < 1/200$. Substituting the above relation from the Fourier mode restriction into (\refeq{eq2.4lm}) yields the discrete differential inequality
 
 $$\frak u_n - \tfrac{1}{100}(\frak u_{n-1} + \frak u_{n+1}) \leq \frak s_n,$$

 \noindent where $\frak s_n= \frac{C'_m}{|\ell|^{2+2m}}e^{-2 n/c_m }$.

 To conclude, we apply a discrete version of the maximum principle: $\frak s_n$ trivially satisfies  
  $$\frak s_n - \tfrac{1}{100}(\frak s_{n-1} + \frak s_{n+1}) \leq \frak s_n$$
  
\noindent because it is positive, thus the difference $\frak r_n= \frak u_n-\frak s_n$ satisfies 
 
 \be \frak r_n - \tfrac{1}{100}(\frak r_{n-1} + \frak r_{n+1}) \leq0 \label{diffineqlm}.\ee 

\noindent Additionally, $\frak r_n\to 0$ as $n\to \infty$ by integrability, and it may be arranged (by increasing $C'_m$)  that $\frak r_0\leq 0$. An interior maximum with $\frak r_n \geq \frak r_{n-1}, \frak r_{n+1}$ would violate (\refeq{diffineqlm}), thus by the ``maximum principle'' we conclude $\frak r_n=\frak u_n- \frak s_n\leq 0$ for all $n\in \N$, thus $u_\ell$ satisfies 

\be \|u_\ell\|^2_{rH^1_e(A_{n\ell})} \leq \frac{C'_m}{|\ell|^{2+2m}}\text{Exp}\left(-\frac{2n}{c_m'}\right)  \ee
\noindent which completes the lemma. 
 \end{proof} 
\section{Boundary and Edge Regularity}
\label{appendixbdreg}

This appendix gives proofs of two facts about regularity in the boundary and edge Sobolev spaces, namely Lemma \refeq{spaces} and Lemma \refeq{coefficientregularity}. 

Recall the Fr\'echet spaces defined in Lemma \refeq{spaces}. To restate the assertion of Lemma \refeq{spaces} succinctly: 

\begin{lm}For $0<\e\leq 1$ there exist smoothing operators $$S_\e: C^\infty(\mathcal Z_0; N\mathcal Z_0)\to C^\infty(\mathcal Z_0; N\mathcal Z_0)  \hspace{2cm} S_\e^{\text{b}} : \bigcap_{m\geq 0}   X_m'' \to \bigcap_{m\geq 0}  X_m''  \hspace{2cm} S_\e^\text{b} :\mathcal  Y\to \mathcal  Y$$ 
satisfying properties (i)-(iii) of Definition \ref{tamedef} and preserving the property (P) of polyhomogeneity defined by  (\refeq{polyhomZ+12}--\refeq{polyhomZ-12}). Additionally, in Fermi coordinates around $\mathcal Z_0$, $S_\e^\text{b}$ does not introduce new Fourier modes in $\theta$. 
\end{lm}

\begin{proof} 
On $\mathcal  X_0=C^\infty(\mathcal Z_0;N\mathcal Z_0)$, $S_\e$ may be defined as a convolution operator using a Schwartz kernel that smoothly approximates the $\delta$-distribution along the diagonal in $(\mathcal Z_0)^2=\mathcal Z_0 \times \mathcal Z_0$. Let $\chi(r)$ be a cut-off function equal to 1 near $r=0$ and vanishing for $r>1$. Fix a collection $U_j \times \C$ for $j=1,..,n$ of trivializations of $N\mathcal Z_0$ on contractible open sets, and for each $j$, choose nested cut-off functions $\xi_j, \beta_j$ such that $\text{supp}(\beta_j)\Subset \{\xi_j=1\}$. Then define \be S_\e(\eta)(t):= \frac{1}{\e}\sum_{j=1}^n\xi_j(t)\int_{\mathcal Z_0} \chi\left(\tfrac{|t-t'|}{\e}\right) \beta_j(t')\eta(t') dt'. \label{smoothingcompact}\ee    

\noindent where the constant $\tfrac1\e$ serves to normalize $\chi$ in $L^2$. Properties (i)-(iii) now follow easily. 

The construction of $S_\e^\text{b}$ is analogous, but now we de-singularize the $\delta$-distribution on the diagonal in the blown-up product defined as follows. Let $B=\mathcal Z_0 \times \mathcal Z_0 \subset Y\times Y$, and let $\mathbb S(B)$ denote the sphere bundle of radius $r_0$ in the normal bundle. Define the blown-up product by $$Y^2_\text{b}:= (Y\setminus (N_{r_0}\mathcal Z_0))^2 \cup \mathbb S(B).$$

\noindent This blow-up is a compact 6-manifold with corners, having three boundary strata of codimension 1 consisting of the interiors of $\del( N_{r_0}\mathcal Z_0\times Y), \del( Y\times N_{r_0}\mathcal Z_0 ), \del(\mathbb S(B))$ which intersect along codimension 2 corners. This space can be given local coordinates $(s,\rho, \theta, \theta', t, t')$ in a neighborhood of the diagonal, where $s=[r, r']$ is a projective coordinate along the blow-up boundary, and $\rho=r'$. 

Away from these strata,  $S_\e^\text{b}$ can be defined analogously to (\refeq{smoothingcompact}); near the boundary strata it is defined as a product $$S_\e^\text{b}:= S_\e^\theta \circ S'_\e$$ where $S_\e^\theta$ is defined by truncation of the $\theta$-Fourier modes in a local trivialization, and $S_\e'$ is given in Fermi coordinates $(s, \rho, t,t')$ by 

\be S'_\e(\psi)(r,t,\theta):= \frac{1}{\e^2}\chi\int_{Y\setminus \mathcal Z_0} \chi\left(\tfrac{|s-1|}{\e}\right)\chi\left(\tfrac{|t-t'|}{\e}\right)\frac{1}{r'}  (\beta \psi)   dt' dr'\label{bsmoothing}\ee
where the factor of $1/r'$ appears because $|r-r'|\sim r's$ and the $\delta$-distribution is homogeneous of order $-1$. 

The properties (i)-(iii) for the spaces $H^m_\text{b}$ now follow analogously to the compact case. That $S_\e^\text{b}$ introduces no new Fourier modes in $\theta$ is manifest from the definition, and the fact that polyhomogeneity is preserved is a consequence of the pushforward theorem or of direct inspection of the integral (\refeq{bsmoothing}) (see \cite[Sec. 3.1]{grieser2001basics}). Since the ratio $r/r'$ is uniformly bounded where $\chi \neq 0$, the commutators $[\nabla^e, S_\e^\text{b}]$ and $[r^\alpha, S_\e^\text{b}]$ are uniformly bounded, properties (i)-(iii) for the space $\bigcap_{m\geq 0}H^{m,1}_{\text{b},e}$ follows from the equivalent description of the norm (\refeq{equivalentbenorm}). The same applies for the terms $(r\del_r \pm \tfrac12)\psi$ and therefore for the spaces $r\mathcal H^{m,1}$ and $\mathcal H^{m,0}$. 

\end{proof}

What remains is to prove Lemma \refeq{coefficientregularity}, which requires several steps.

\begin{lm}
If $\ph \in r^\alpha H^3_\text{b}$ for $\alpha >1$ then the $\ph$ satisfies the pointwise bound 

\be |\ph(x)| \leq C \| \ph\|_{rH^3_\text{b}}.\label{coefficientest1}\ee 
\label{Rafe0.1}
\end{lm}
\begin{proof}
We first prove the lemma in the 1-dimensional case: consider $\R^+=(0,\infty)$ with the measure $rdr$ and suppose that $\ph \in rH^1_b(rdr)$ has $\text{supp}(\ph)\subseteq (0,1]$. Then we claim that there is a constant $C$ so that  \be |\ph(x)| \leq C \| \ph\|_{rH^1_\text{b}(rdr)}= C\left(\int_{\R^+} \frac{|\ph|^2}{r^2} + |\nabla \ph|^2  \  rdr \right)^{1/2}.\label{coefficientest1}\ee 

\noindent This follows from a dyadic decomposition. Since $r$ is uniformly bounded on $[1/2,2]$ and $H^{1}[1/2,2] \hookrightarrow C^0[1/2,2]$ by the standard Sobolev embedding, we have $|\ph(1)|^2\leq c\int |\ph|^2 + |\nabla \ph|^2 dr\leq c\|\ph\|_{rH^1_\text{b}}^2$. Then, by the Fundamental Theorem of Calculus, 
\bea
|\ph(1/2)| &\leq& |\ph(1)| + \int_{1/2}^1 |\ph'(\rho)|d\rho \leq  |\ph(1)| + \left(\int_{1/2}^1 |\ph'(\rho)|^2 \rho d\rho\right)^{1/2}\left(\int_{1/2}^2 \frac{1}{\rho} d\rho \right)^{1/2} \\
&\leq & |\ph(1) | +(\log 2)^{1/2} \|\ph\|_{rH^1_\text{b}([1/2,1], rdr)}
\eea
Similarly,  $|\ph(1/4)| \leq |\ph(1/2) | +(\log 2)^{1/2} \|\ph\|_{rH^1_\text{b}([1/4,1/2], rdr)} \leq |\ph(1) | +(\log 2)^{1/2} \|\ph\|_{rH^1_\text{b}([1/4,1], rdr)}$ where the second inequality follows from substituting the above. In general, using the estimate on $|\ph(1)|$ we conclude that $$|\ph{(2^{-k})}|\leq C\|\ph\|_{rH^1_\text{b}(rdr)}.$$

\noindent (\refeq{coefficientest1}) then follows from applying the Fundamental Theorem of calculus again for $x\in [2^{-k}, 2^{-k+1}]$.  

 In general, for $\ph \in H^m_\text{b}(Y\setminus \mathcal Z_0; S_0)$, the lemma follows from the above by applying (\refeq{coefficientest1}) to rays of constant $(t,\theta)$ and after using the Sobolev restriction theorem $rH^3_\text{b}(Y\setminus \mathcal Z_0)\to rH^1_b( rdr)$, where the one-dimensional space is a normal ray emanating from $\mathcal Z_0$. 

\end{proof}

Next, we have the following fundamental fact about ODEs. For it, we use the 1-dimensional $\text{b}$-spaces $r^\alpha H^{1}_\text{b}([0,1], rdr)$ and $r^\alpha L^2([0,1], rdr)$ defined by the norms  $$\|u\|_{r^\alpha H^{1}_\text{b}} =\left( \int_{0}^1 (|r\del_r u|^2 + |u|^2 )r^{-2\alpha} rdr\right)^{1/2} \hspace{2cm}\|u\|_{r^\alpha L^{2}_\text{b}} =\left( \int_{0}^1 |u|^2 r^{-2\alpha} rdr \right)^{1/2}.$$

\begin{lm}\label{B3} Provided $\alpha > 3/2$ then 

$$(r\del_r - \tfrac{1}{2}): r^\alpha H^{1}_\text{b}(0,1] \to r^\alpha L^2_\text{b}(0,1] $$
is an isomorphism, hence there is a constant $C$ such that

\be \|u\|_{r^\alpha H^{1}_\text{b}} \leq C\|(r\del_r - \tfrac{1}{2})u\|_{r^\alpha L^2_\text{b}}\label{1dimest}\ee
\noindent holds. 
\end{lm}

\begin{proof}Setting $r=e^s$ for $s\in (-\infty,0]$ the problem is equivalent to the analogous statement for

$$\del_s - \tfrac{1}{2}: e^{(1-\alpha)s}L^{1,2}((-\infty,0], ds)\lre e^{(1-\alpha)s}L^{2}((-\infty,0], ds)$$

\noindent which is conjugate to 

$$\tfrac{1}{e^{(\alpha-1)s}}(\del_s - \tfrac{1}{2}) e^{(\alpha-1)s}=(\del_s + \alpha-\tfrac{3}{2}): L^{1,2}((-\infty,0], ds)\to  L^{2}((-\infty,0], ds).$$

\noindent where $L^{1,2}((-\infty,0],ds)$ is the standard Sobolev space on the half-line. The claim then follows directly from integrating by parts since the boundary term $(\alpha-\tfrac32)|u(0)|^2>0$ is strictly positive.  
\end{proof}

We now conclude the proof of Lemma \refeq{coefficientregularity}. 

\begin{proof}[Proof of Lemma \refeq{coefficientregularity}] If $\ph$ is compactly supported away from $\mathcal Z_0$ in $Y\setminus N_{r_0/2}(\mathcal Z_0)$, the lemma is immediate from the standard Sobolev Embedding Theorem since $r$ is bounded below in this region. We may therefore assume that $\ph$ is supported in a tubular neighborhood of $\mathcal Z_0$. Since $\ph \in r\mathcal H^{m,1} \cap \bold P_\mathcal X$ by assumption, we may write $$\ph= A(t,\theta)r^{1/2} + B(t,\theta,r)$$
in local coordinates, after which it suffices to show the bound for each term individually. 

Using Lemma \ref{B3} by applying (\refeq{1dimest}) to derivatives and integrating over the $t,\theta$ variables leads to 

$$\|r\del_r u\|_{r^\alpha H^m_\text{b}}   + \|u \|_{r^\alpha H^m_\text{b}} \leq C\|(r\del_r -\tfrac{1}{2}) u \|_{r^\alpha H^m_\text{b}}$$
 
\noindent for $\alpha>3/2$ and in particular for $\alpha=1+\nu$. Applying this to $B(t,\theta,r)$ and discarding the derivative term on the left shows that

\be \|B(r,t, \theta)\|_{r^{1+\nu} H^\text{m}_\text{b}}\leq C \|(r\del_r -\tfrac{1}{2})B\|_{r^{1+\nu}H^m_\text{b}} = C \|(r\del_r -\tfrac{1}{2})\ph \|_{r^{1+\nu}H^m_\text{b}} \leq C \|\ph\|_{r\mathcal H^{m+1,1}} \label{Bbound}\ee

\noindent since $(r\del_r -\tfrac12)$ annihilates $A(t,\theta)r^{1/2}$. Then, applying Lemma \ref{Rafe0.1} to $B(t,\theta,r)r^{-1/2}$ and substituting (\refeq{Bbound}) shows that 

$$|B(r,t,\theta)r^{-1/2}| \leq  \|B(r,t, \theta)\|_{r^{3/2} H^\text{m}_\text{b}}\leq \|B(r,t, \theta)\|_{r^{1+\nu} H^\text{m}_\text{b}}\leq C \|\ph\|_{r\mathcal H^{m+1,1}}$$

\noindent and the result for $B(t,\theta,r)$ follows after multiplying by $r^{1/2}$. 

For the first term, the triangle inequality and  (\refeq{Bbound}) shows that $$\|A(t,\theta)r^{1/2}\|_{rH^m_\text{b}} = \|A(t,\theta)r^{1/2} + B(r,t,\theta)\|_{rH^m_\text{b}} +  \|B(r,t, \theta)\|_{r H^\text{m}_\text{b}}\leq C \|\ph\|_{r\mathcal H^{m+2,1}}.$$ 

\noindent Finally, since $\|A(t,\theta)r^{1/2} \|_{rH^m_\text{b}}\sim \|A(t,\theta)\|_{H^{m}(T^2)}$, the bound for the first term follows from the Sobolev embedding on $T^2$ after increasing $m+2$ to $m+4$.  
\end{proof}

\label{appendixII}

\section{Non-Linear Deformation Terms}
\label{appendixC}

This appendix proves Proposition \ref{fullynonlinear}. Retaining the notation of that proposition and the preceding discussion, Proposition \ref{fullynonlinear} asserted in Item (C) that the non-linear terms of $\slashed{\mathbb D}_p$ may be written. 

\bea Q_p(\eta,\ph)&=&(\mathcal B_\ph + \mathfrak B_{\ph, p})(\eta) \ \ + \ \ M^1_p(\eta', \eta')\nabla( \Phi_0+ \ph)  \ \ + \ \  M^2_p (\eta',\eta'') (\Phi_0+\ph) \ \ + \ \ F_p(\eta, \Phi_0+ \ph).  \eea

\begin{proof}[Proof of Proposition \ref{fullynonlinear}]

The constant (Item A) and linear (Item B) terms are immediate from, respectively, the definition (\refeq{errortermdef}) and the proof of Corollary \ref{formofDD} using the pullback metric (\refeq{linearizedpullbackp}) in place of $\dot g_\eta$.

We prove the above formula in Item (C) for the parameter $p=p_0$, as the general case differs only in notation. The quadratic terms in Item (C) are calculated by writing the full expression (\refeq{NLBG}) and subtracting off the linear terms (the constant term vanishes for $p=p_0$). To this end, we apply (\refeq{NLBG}) to $s\eta$, and collect the terms quadratic in $s$. By Remark \ref{nonlinearpullbackmetric}, the pullback metric can be written $$g_{s\eta}=
 g_0 + s\dot g_\eta + s^2\frak q(\eta) \hspace{1cm}\text{where}\hspace{1cm} |\frak q(\eta)|\leq C  \big( \ | \eta' \chi |+ |\eta d\chi |+ |\eta \chi| \ \big)^2.$$

 \noindent i.e. $\frak q(\eta)$ is comprised of terms that are quadratic and higher order in $\eta$; in particular, it vanishes to second order at $s=0$ (we omit the dependence of $\frak q$ on $s$ from the notation). Working in an orthonormal frame with respect to $g_0$, the Taylor expansion shows that $\frak a$ in (\refeq{NLBG}) and its inverse are given by
 \bea
 \frak a&=&\text{Id}-\tfrac s2 \dot g_\eta + s^2\frak p_1\\
  \frak a^{-1}&=&\text{Id}+\tfrac s2 \dot g_\eta + s^2\frak p_2,
 \eea

 \noindent where $\frak p_i$ have entries consisting of sums and products of smooth functions depending on the metric and on quadratic and higher combinations of $\eta'\chi, \eta d\chi$ and $\eta \chi$.

We now substitute these expressions into (\refeq{NLBG}). Working in an orthonormal frame of $g_0$, with indices ranging over $1,2,3$ the first term becomes 
\be 
\sum_i e^i . \nabla^{g_0}_{\frak a(e_i)} = \sum_{i} e^i \nabla^{g_0} _i  -\frac{s}{2}\sum_{ij} \dot g_{\eta} (e_i, e_j)e^i. \nabla_j^{g_0} +s^2 {\sum_{ij} \frak p_1(e_i, e_j) e^i .\nabla_j^{g_0}}.\label{aleph1}
\ee

\noindent Therefore, in the trivialization of Lemma 5.1, the contribution of this first term is 
\bea \slashed{\mathbb D}(s\eta, \Phi_0 + s\ph)&=&\left(\sum_{i} e^i \nabla^{g_0} _i  -\frac{s}{2}\sum_{ij} \dot g_{\eta} (e_i, e_j)e^i. \nabla_j^{g_0} +s^2 {\sum_{ij} \frak p_1(e_i, e_j) e^i .\nabla_j^{g_0}}\right) (\Phi_0 + s\ph) \ + \ \ldots \\
&=& s\left(\slashed D_0\ph  -\frac{1}{2}\sum_{ij}\dot g_\eta(e_i,e_j)e^i. \nabla_j^{g_0}\Phi_0\right) \ + \   \\ 
& &  s^2 \left( -\frac{1}{2}\sum_{ij}\dot g_\eta(e_i,e_j)e^i. \nabla_j^{g_0}\ph + \sum_{ij}\frak p_1(e_i,e_j)e^i.\nabla_j^{g_0}(\Phi_0+s\ph)\right)\ + \ \ldots  \\
&=& s \cdot \text{d}\slashed{\mathbb D}(\eta,\ph) + s^2 \Big(\mathcal B_{\ph}(\eta) \ + \  M^1(\eta,\eta')\nabla^{g_0}(\Phi_0 + s\ph) \ + \ F(\eta, \Phi_0 + \ph) \Big)  \ + \ \ldots
\eea

\noindent since $\slashed D_0\Phi_0=0$, where $\ldots$ constitutes the contribution from the remaining terms besides (\refeq{aleph1}). The quadratic terms of $\frak p_1$, by what was said above, contains exactly the type of terms asserted to be part of $M^1(\eta, \eta')$, with $F(\eta,\Phi_0+\ph)$ being the terms of higher than quadratic order.

The remaining terms of (\refeq{NLBG}) proceed in a similar fashion. Explicitly, some (quite a lot actually) of computation shows that they are 
\bea
\frac{1}{4}\sum_{ij} e^i. e^j. ( \frak a^{-1}( \nabla^{g_0}_{\frak a(e_i)}\frak a)e^j)&=&s\cdot  \frac{1}{4}d\text{Tr}_{g_0}(\dot g_\eta) \ + \ s^2 \Big (-\frac{1}{2}d\Tr_{g_0}(\frak p_1) + \frac{1}{8} \Tr_{g_0}((\dot g_\eta + s \frak p_2 ) \nabla^{g_0} (  \dot g_\eta - 2s \frak p_1).   )  \\
 & &  +   \frac{1}{4}\sum_{ijkm\ell}  \Big[(\dot g_\eta -2 s \frak p_1)_{j\ell} (\text{Id} + \tfrac{1}{2}s \dot g_\eta + s^2 \frak p_2)_{jm}[\nabla^{g_0}_\ell (\dot g_\eta -2 s \frak p_1)]_{mk}\Big] e^i. e^j. e^k. \Big ) \eea
 \noindent and
 \bea 
\frac{1}{4}\sum_{ij}e^i e^j.   \left(\frak a^{-1}(\nabla^{g_\eta}-\nabla^{g_0})_{\frak a(e_i)} \frak a(e^j) \right). &=& s\cdot \Big (\frac{1}{2} \text{div}_{g_0}(\dot g_\eta). + \frac{1}{4} d \Tr_{g_0}(s\dot g_\eta)\Big)  \\ 
& + & s^2 \Big ( \frac{1}{4}\sum_{ijk \ell}  e^i. e^j.\Big [ \left(\text{Id} + \tfrac{1}{2}\dot g_\eta + s \frak p_2\right)      \Big( \Gamma_i (-\tfrac{1}{2}\dot g_\eta + s \frak p_1)_{j\ell}e^\ell.   \\ & &  + \   (-\tfrac{1}{2}\dot g_\eta + s \frak p_1)_{ik} \Gamma_k(\text{Id}-\tfrac{1}{2}\dot g_\eta + s \frak p_1) _{j \ell}e^\ell. \Big)  \Big ] \\
 & & + \ \frac{1}{4} \sum_{ij} e^i. e^j. \left[ \left(\tfrac{1}{2} \dot g_\eta + s \frak p_2\right) \left(\tfrac{1}{2} \text{div}_{g_0}(\dot g_\eta). + \tfrac{1}{4} d \Tr_{g_0}(\dot g_\eta)\right)\right]. \Big )
\eea

\noindent where \bea s\Gamma_i e^j&=&(\nabla^{g_\eta}_i -\nabla^{g_0}_i)e^j\\ &=& \frac{s}{2}\sum_{\ell k } (g_\eta)^{-1}_{\ell k} \left( (\nabla_i^{g_0} (\dot g_\eta  + s \frak p_1))_{j\ell}+ (\nabla_j^{g_0} (\dot g_\eta + s \frak p_1))_{i\ell}-(\nabla_\ell^{g_0} (\dot g_\eta + s \frak p_1))_{ij}\right)e^k.
\eea

\noindent is the difference of the Levi-Civita connections. The terms linear in $s$ combine to yield the remaining terms of $\text{d}\slashed{\mathbb D}$ in Item (B), while the $s^2$ terms are combined into $M^2(\eta', \eta'')$ or absorbed into $F(\eta,\ \Phi_0+\ph)$. Note that, by the product rule, each term of $M^2(\eta', \eta'')$ contains are most a single instance of the second derivative $\eta''$.   
\end{proof}


{\small

\medskip

\bibliographystyle{amsplain}
{\small 
\bibliography{BibliographyPartIISecond}}
\end{document}



\end{document}